%% file: surface_rep.tex
\documentclass[]{amsart}
\usepackage[dvips]{graphicx}
\usepackage{epsfig}
\usepackage{amssymb,amsmath}
\usepackage{color}
\usepackage[all]{xy}

\newtheorem{theorem}{Theorem}[section]
\newtheorem{proposition}[theorem]{Proposition}
\newtheorem{lemma}[theorem]{Lemma}
\newtheorem{corollary}[theorem]{Corollary}

\theoremstyle{definition}
\newtheorem{definition}[theorem]{Definition}

\theoremstyle{remark}
\newtheorem{remark}[theorem]{Remark}

\numberwithin{equation}{section}

\newcommand{\SLC}{\mathrm{SL}(2,\mathbb{C})}
\newcommand{\PSLC}{\mathrm{PSL}(2,\mathbb{C})}
\newcommand{\GLC}{\mathrm{GL}(2,\mathbb{C})}
\newcommand{\PGLC}{\mathrm{PGL}(2,\mathbb{C})}
\newcommand{\SLR}{\mathrm{SL}(2,\mathbb{R})}
\newcommand{\PSLR}{\mathrm{PSL}(2,\mathbb{R})}
\newcommand{\PGLR}{\mathrm{PGL}(2,\mathbb{R})}
\newcommand{\Hom}{\mathrm{Hom}}

\DeclareMathOperator{\tr}{tr}

\begin{document}

\title{Parametrization of $\PSLC$-representations of surface groups}
\author{Yuichi Kabaya}
\address{Department of Mathematics, Osaka University, 
Toyonaka, Osaka, 560-0043, JAPAN}
\email{y-kabaya@cr.math.sci.osaka-u.ac.jp}


\begin{abstract}
For an oriented surface of genus $g$ with $b$ boundary components, 
we construct a rational map from a subset of $\mathbb{C}^{6g-6+3b}$ onto an open algebraic subset of the $\PSLC$-character variety 
as an analogue of the Fenchel-Nielsen coordinates.
After taking the quotient by an action of a finite group, we obtain a parametrization of a subset of the $\PSLC$-character variety, 
and similarly for the $\SLC$-character variety.
We can systematically calculate a set of matrix generators by rational functions of the parameters. 
We give transformation formulae under elementary moves of pants decompositions.
\end{abstract}

\maketitle


\section{Introduction}
$\PSLC$-representations of surface groups appear in various areas of low dimensional topology and geometry;
in the study of Kleinian groups, complex projective structures, Teichm\"uller spaces, etc.
In this paper, we give a parametrization of $\PSLC$-representations of a surface group as an analogue of Fenchel-Nielsen coordinates.

Let $S=S_{g,b}$ be a surface of genus $g$ with $b$ boundary components.
In this paper, we assume that the Euler characteristic of $S$ is negative, hence $S$ admits a hyperbolic structure with geodesic boundary.
The $\SLC$-character variety of $S$ is, roughly, the set of all representations of $\pi_1(S)$ into $\SLC$ up to conjugacy.
(See  \S \ref{subsec:char_var} for a precise description of the character variety.)
The $\PSLC$-character variety of $S$ is also defined.
The space of marked hyperbolic structures on $S$ is called the Teichm\"uller space of $S$.
Since a marked hyperbolic structure induces a discrete faithful representation of $\pi_1(S)$ into $\PSLR \subset \PSLC$ (Fuchsian representation), 
the Teichm\"uller space can be regarded as a subspace of the $\PSLC$-character variety.

The Fenchel-Nielsen coordinates give a parametrization of the Teichm\"uller space.
They are defined based on a pants decomposition, a collection of disjoint simple closed curves on $S$ which cut $S$ into three-holed spheres (pants).
We call a curve of a pants decomposition or a boundary curve of $S$ a \emph{pants curve}.
The Fenchel-Nielsen coordinates consist of the length and twist parameters about the pants curves.
It is known that the hyperbolic structure on a three-holed sphere is completely determined by the lengths of the boundary curves.
Thus, to recover the hyperbolic structure on $S$, we have to keep track of how these pairs of pants are glued along the common pants curves. 
The Fenchel-Nielsen twist parameter of an interior pants curve measures how two pairs of pants are twisted along the pants curve (see Figure \ref{fig:FN_twist.pstex_t}).
It is well-known that the set of the length and twist parameters gives a parametrization of the Teichm\"uller space.
The Fenchel-Nielsen coordinates are complexified by Tan \cite{tan} and Kourouniotis \cite{kourouniotis}, which parametrize the quasi-Fuchsian representations. 

We shall give a parametrization of an open algebraic subset of the $\PSLC$-character variety, which contains all quasi-Fuchsian representations.
Since our parametrization is only based on some elementary properties of matrices, we can systematically calculate a set of matrix generators in terms of the parameters. 
(See examples in Section \ref{sec:example}; four-holed sphere, one-holed torus and closed genus two surface.)
These matrices have entries which are rational functions of the parameters, we can also parametrize $\mathrm{PGL}(2,\mathbb{F})$ representations 
for any subfield $\mathbb{F} \subset \mathbb{C}$.
We remark that a set of explicit matrix generators was given by \cite{okai_explicit} and \cite{maskit} for Fuchsian representations in terms of the Fenchel-Nielsen coordinates. 
Our explicit description enables us to give transformation formulae under changes of pants decompositions (Section \ref{sec:change_to_shear_coodinates}).

The idea of this paper is simple: use the eigenvalues of the pants curves instead of the length (or trace) functions.
By using the eigenvalues, we can keep track of the information of the fixed points.
Recall that $\PSLC$ acts on the projective space $\mathbb{C}P^1$.  
Let $A$ be an element of $\SLC$ which has two fixed points $x$ and $y$ on $\mathbb{C}P^1$.
We denote by $e$ the eigenvalue corresponding to the fixed point $x$, i.e. $x$ is the projective class of the eigenvector of $A$ corresponding to $e$.
Hence $y$ is the fixed point corresponding to $e^{-1}$.
Then the matrix $A$ is uniquely determined by the triple $(e,x,y)$. 
(See (\ref{eq:two_fixed_points_case}) for its explicit form.) 
Let $P$ be a pair of pants and take a set of generators $\gamma_1, \gamma_2, \gamma_3$ of $\pi_1(P)$ corresponding to the three boundary curves 
so that $\gamma_1 \gamma_2 \gamma_3 = 1$ (see Figure \ref{fig:generators_for_pants}).
Let $\rho$ be an irreducible $\SLC$-representation such that $\rho(\gamma_i)$ has two fixed points. 
We denote one of the eigenvalue of $\rho(\gamma_i)$ by $e_i$ and let $x_i$ (resp. $y_i$) be the fixed point corresponding to $e_i$ (resp. ${e_i}^{-1}$).
From the relation $\gamma_1 \gamma_2 \gamma_3 = 1$, we will see in Proposition \ref{prop:pants_rep} that $\rho$ is uniquely determined by $(e_i, x_i)_{i=1,2,3}$.
Conversely, for any three complex numbers $e_1,e_2,e_3$ (satisfying some conditions) 
and three distinct points $x_1,x_2,x_3$ of $\mathbb{C}P^1$, 
we can construct such an $\SLC$-representation.
While the triple $e_1,e_2,e_3$ determines a conjugacy class of the representation $\rho$, 
the additional information of the fixed points $x_1, x_2, x_3$ determines a representative in the conjugacy class containing $\rho$.

To glue two representations along a pants curve, the information of the fixed points is still useful. 
Let $P$ and $P'$ be two pairs of pants.
We denote the boundary curves of $P$ by $c_1$, $c_2$, $c_3$ and the boundary curves of $P'$ by $c_1$, $c_4$, $c_5$.
We construct a $\SLC$-representation of the fundamental group of $P \cup_{c_1} P'$ from representations of $\pi_1(P)$ and $\pi_1(P')$.
We take a set of generators $\gamma_1, \gamma_2, \gamma_3$ of $\pi_1(P)$ and a set of generators $\gamma_1, \gamma_4, \gamma_5$ of $\pi_1(P')$ as before. 
We denote the representation of $\pi_1(P)$ obtained from $(e_i,x_i)_{i=1,2,3}$ by $\rho$ 
and the representation of $\pi_1(P')$ obtained from $(e'_i,x'_i)_{i=1,4,5}$ by $\rho'$.
To glue these representations, they must satisfy  $\rho(\gamma_1) = \rho'(\gamma_1)^{-1}$, 
thus both $e'_1 = {e_1}^{-1}$ and $(x_1,y_1) = (x'_1,y'_1)$ where $y_1$ and $y'_1$ are the other fixed point of $\rho(\gamma_1)$ and $\rho'(\gamma_1)$ respectively.
So we assume that $e'_1 = {e_1}^{-1}$.
The condition $(x_1,y_1) = (x'_1,y'_1)$ can be achieved by conjugating $\rho'$ by an element of $\PSLC$.
Although there is a one-parameter family of such conjugating elements, the \emph{twist parameter} of $c_1$ characterizes that in the one-parameter family
by measuring the relative positions of the fixed points $x_2$ and $x'_5$.
(See Section \ref{sec:twist_paramter} for the precise definition.)

In the definition of the twist parameter, we need to introduce an oriented graph dual to the pants decomposition 
to keep track of the relative positions of the fixed points.
This dual graph will also be used to give a presentation of $\pi_1(S)$ (Section \ref{sec:pants_decomp}).
According to this presentation, we will reconstruct a set of matrix generators from the eigenvalue and twist parameters (Section \ref{sec:matrix_generators}).
Thus the set of eigenvalue and twist parameters determines a $\PSLC$-representation up to conjugation.
This gives a rational map from an open algebraic subset of $\mathbb{C}^{6g-6+3b}$ onto an open algebraic subset of the $\PSLC$-character variety.
After taking a covering space of the parameter space, we can also construct a map to the $\SLC$-character variety.

Since there exist two choices of eigenvalues for each pants curve, these parameters are not invariant under conjugation, unlike the length (or trace) functions.
But once we fix all the eigenvalue parameters, the twist parameters are invariant under conjugation.
Thus the system of the eigenvalue and twist parameters gives a multi-valued map from a subset of the $\PSLC$-character variety to $\mathbb{C}^{6g-6+3b}$.
Changing the choices of eigenvalues is described by an action of $(\mathbb{Z}/2\mathbb{Z})^{3g-3+2b}$ which acts as $e_i \to {e_i}^{\pm 1}$.
This action affects the twist parameters, which will be described in Section \ref{sec:action}.
After taking the quotient by the action of $(\mathbb{Z}/2\mathbb{Z})^{3g-3+2b}$, and further taking the quotient by the action of the group related to 
the lifting of $\PSLC$ to $\SLC$, we obtain an injective map from the quotient of the parameter space to a subset of the $\PSLC$-character variety.
The precise statement will be given in Theorem \ref{thm:main_PSLC}, and Theorem \ref{thm:main_SLC} for the $\SLC$-character variety.

This parametrization is also described by using the ideal triangulation of the surface associated to the pants decomposition.
The author briefly noted in \cite{kabaya} the parametrization from this point of view.
We will interpret our eigenvalue and twist parameters in terms of ideal triangulations in Section \ref{sec:developing_map}, 
and apply this to study Fuchsian representations in Section \ref{sec:psl2r}.

Since our coordinates are based on a pants decomposition with a dual graph, 
it is natural to consider their transformation under a change of pants decomposition.
We introduce five types of moves between pants decompositions with dual graphs.
We will show that any two pants decompositions with dual graphs are related by a sequence of these moves and give transformation formulae for these moves.


This paper is organized as follows.
In Section \ref{sec:basicx}, we review basic properties of $\PSLC$ and the definition of the character variety.
In Section \ref{sec:rep_of_pants}, we give a parametrization of the representations of the fundamental group of a pair of pants.
In Section \ref{sec:pants_decomp}, we define pants decompositions and their dual graphs. 
We define the eigenvalue parameters in Section \ref{sec:eigenvalue_paramter} and the twist parameters in Section \ref{sec:twist_paramter}.
In Section, \ref{sec:matrix_generators}, we shall show that the set of the eigenvalue and twist parameters gives coordinates of a subset of the character variety 
and we give a set of matrix generators in terms of the eigenvalue and twist parameters.
After giving examples in Section \ref{sec:example}, we describe the behavior of the twist parameters 
under the action of $(\mathbb{Z}/2\mathbb{Z})^{3g-3+2b}$, which exchanges the choices of the eigenvalues, in Section \ref{sec:action}.
In Section \ref{sec:transformation}, we define moves between pants decompositions with dual graphs and give transformation formulae for each of these moves. 
In Section \ref{sec:developing_map}, we review the notion of developing maps, which will be used in the remaining sections. 
In Section \ref{sec:psl2r}, we restrict our attention to $\PSLR$-representations, especially Fuchsian representations.
We compare our twist parameters with the usual Fenchel-Nielsen twist parameters.
In the last two sections, we compare our coordinates with the exponential shear-bend coordinates.

\noindent
{\bf Acknowledgements.}
I would like to thank Toshihiro Nakanishi for inviting me to the 53th Symposium on Function Theory at Nagoya in November 2010.
This work was started when I was preparing for the talk.
I also thank Akira Ushijima for his interest in this work.
This work was supported by JSPS Research Fellowships for Young Scientists.

\section{Basic facts on $\mathrm{PSL}(2, \mathbb{C})$}
\label{sec:basicx}
\input{sec_basic}

\section{Representations of the fundamental group of a pair of pants}
\label{sec:rep_of_pants}
\input{sec_pants_rep}

\section{Pants decomposition and dual graph}
\label{sec:pants_decomp}
\input{sec_pants_decomp}

\section{Eigenvalue parameter}
\label{sec:eigenvalue_paramter}
\input{sec_eigen}

\section{Twist parameter}
\label{sec:twist_paramter}
\input{sec_twist.tex}

\section{Matrix generators}
\label{sec:matrix_generators}
\input{sec_generators}

\section{Examples}
\label{sec:example}
\input{sec_examples}

\section{Action of $(\mathbb{Z}/2\mathbb{Z})^{3g-3+2b}$}
\label{sec:action}
\input{sec_action}

\section{Transformation formula}
\label{sec:transformation}
\input{sec_transformation}

\section{Developing map}
\label{sec:developing_map}
\input{sec_developing}

\section{$\PSLR$-representations}
\label{sec:psl2r}
\input{sec_psl2r}

\section{Representations of 3-manifold groups}
\label{sec:rep_of_3_manifolds}
\input{sec_3_mfd}

\section{Transformation  to exponential shear-bend coordinates}
\label{sec:change_to_shear_coodinates}
\input{sec_shear_bend}

\bibliographystyle{amsalpha}
\bibliography{surface_rep}

\end{document}

%% file: sec_basic.tex
%
%
\subsection{}
\label{subsec:basic_facts_on_PSL}
Let $\mathbb{F}$ be a subfield of $\mathbb{C}$. 
We let
\[
\mathrm{PGL}(2, \mathbb{F}) = \mathrm{GL}(2, \mathbb{F}) / \mathbb{F}^{*}, \quad \mathrm{PSL}(2,\mathbb{F}) = \mathrm{SL}(2,\mathbb{F}) / \{\pm I\},
\]
where $\mathbb{F}^* = \mathbb{F} \setminus \{0\}$ acts on $\mathrm{GL}(2, \mathbb{F})$ by scalar multiplication.
The inclusion map $\mathrm{SL}(2,\mathbb{F}) \to \mathrm{GL}(2,\mathbb{F})$ induces a homomorphism $\mathrm{PSL}(2,\mathbb{F}) \to \mathrm{PGL}(2,\mathbb{F})$, 
and it has an inverse
\[
\mathrm{PGL}(2,\mathbb{F}) \ni A \mapsto \frac{1}{\sqrt{\det A}} A \in \mathrm{PSL}(2,\mathbb{F}),
\]
if $\mathbb{F}$ has a square root for any element of $\mathbb{F}$ (e.g. $\mathbb{F} = \mathbb{C}$).
We remark that there are two choices of square roots, but the matrix in $\mathrm{PSL}(2,\mathbb{F})$ is uniquely determined.

The left action of $\mathrm{GL}(2;\mathbb{F})$ on $\mathbb{F}^2$ induces a left action of $\mathrm{GL}(2;\mathbb{F})$ (also $\mathrm{PGL}(2;\mathbb{F})$) on 
the projective line $\mathbb{F}P^1$ over $\mathbb{F}$.
If we regard $\mathbb{F}P^1$ as $\mathbb{F} \cup \{ \infty \}$, 
the action is given by 
\[
\begin{pmatrix} a & b \\ c & d \end{pmatrix} \cdot z = \frac{az+b}{cz+d}.
\]
The projective class of an eigenvector of $A \in \mathrm{GL}(2,\mathbb{F})$ corresponds to a fixed point of the action of $A$ on $\mathbb{C}P^1$ 
via the projection map $\mathbb{C}^2 \setminus \{0\} \to \mathbb{C}P^1$.
Any non-identity element of $\mathrm{PGL}(2,\mathbb{F})$ has one or two fixed points on $\mathbb{C}P^1$.

Let $A$ be an element of $\mathrm{SL}(2,\mathbb{F})$ with two distinct fixed points.
Choose one of the eigenvalue $e$ of $A$. 
We let $x$ be the fixed point of $A$ determined by the projective class of the eigenvector of $A$ corresponding to $e$.
Then the other fixed point $y$ is determined by the eigenvector corresponding to $e^{-1}$. 
Geometrically, $x$ is the attractive fixed point and $y$ is the repelling fixed point if $|e| > 1$.
Then $A$ is uniquely determined by $(e;x,y)$ and given by
\begin{equation}
\label{eq:two_fixed_points_case}
\begin{split}
M(e;x,y) &=
\begin{pmatrix} x & y \\ 1 & 1 \end{pmatrix} 
\begin{pmatrix} e & 0 \\ 0 & e^{-1} \end{pmatrix}
\begin{pmatrix} x & y \\ 1 & 1  \end{pmatrix}^{-1} \\
& = \frac{1}{x-y} \begin{pmatrix} ex - e^{-1}y & -(e-e^{-1})xy \\ e-e^{-1} & -ey + e^{-1}x \end{pmatrix}.
\end{split}
\end{equation}
We denote this matrix by $M(e;x,y)$.
For example, $M(e;\infty,0) = \begin{pmatrix} e & 0 \\ 0 & e^{-1} \end{pmatrix}$ and $M(e;0,\infty)=\begin{pmatrix} e^{-1} & 0 \\ 0 & e \end{pmatrix}$.
We have $M(e;x,y) = M(e^{-1}; y,x)$ and $M(e;x,y)^{-1} = M(e^{-1}; x,y) =M(e;y,x)$.
If $x$, $y$ and $e$ are in $\mathbb{F}P^1$, then the matrix is an element of $\mathrm{SL}(2, \mathbb{F})$.
Conversely any element $A$ of $\mathrm{SL}(2,\mathbb{F})$ 
has such a form if $t^2 - (\tr A )t + 1 =0$ has two solutions in $\mathbb{F}$. 

The following well-known facts play important roles in our description of $\mathrm{PSL}(2.\mathbb{F})$-representations.
\begin{lemma}
\label{lem:two_fixed_pts_with_one_point_behavior}
Let $x$ and $y$ be distinct points on $\mathbb{F}P^1$.
Let $z_1$ and $z_2$ be points on $\mathbb{F}P^1$ different from $x$ and $y$. ($z_1$ and $z_2$ may coincide.)  
Then there exists a unique $t \in \mathbb{C}^*$ up to sign such that $M(t;x,y)$ sends $z_1$ to $z_2$.
\end{lemma}
\begin{proof}
Since
\[
M(t;x,y) \cdot z_1 = \frac{(tx - t^{-1}y)z_1 -(t-t^{-1})xy}{ (t-t^{-1})z_1 -ty + t^{-1}x} = z_2,
\]
we have 
\[
t^2 = \frac{(x-z_1)(y-z_2)}{(x-z_2)(y-z_1)}.
\]
(This is equal to the cross ratio $[y:x:z_1:z_2]$, see (\ref{eq:cross_raio}).)
Therefore $t$ is well-defined up to sign.
\end{proof}

\begin{lemma}
\label{lem:thrice_transitive}
There exists a unique element of $\mathrm{PGL}(2,\mathbb{F})$ which sends any three distinct points $(x_1,x_2,x_3)$ of $\mathbb{\mathbb{F}}P^1$ 
to other three distinct points $(x'_1,x'_2, x'_3)$.
The matrix is given by 
\[
\frac{1}{\sqrt{(x_1-x_2)(x_2-x_3)(x_3-x_1)(x'_1-x'_2)(x'_2-x'_3)(x'_3-x'_1)}} \begin{pmatrix} a_{11} & a_{12} \\ a_{21} & a_{22} \end{pmatrix}
\]
where
\[
\begin{split}
a_{11} &= x_1 x'_1(x'_2-x'_3)+ x_2 x'_2(x'_3-x'_1)+x_3 x'_3(x'_1-x'_2), \\
a_{12} &=x_1 x_2 x'_3(x'_1-x'_2)+x_2 x_3 x'_1(x'_2-x'_3)+ x_3 x_1 x'_2(x'_3-x'_1), \\
a_{21} &=x_1(x'_2-x'_3)+x_2(x'_3-x'_1)+x_3(x'_1-x'_2), \\
a_{22} &=x_1 x'_1(x_2-x_3)+x_2 x'_2(x_3-x_1)+x_3 x'_3(x_1-x_2). \\
\end{split}
\]
\end{lemma}
\begin{proof}
Let $A = \begin{pmatrix} a & b \\ c & d \end{pmatrix}$ be an element of $\mathrm{GL}(2,\mathbb{C})$ which sends $(0,\infty, 1)$ to the triple $(x_1, x_2,x_3)$.
For simplicity, first assume that any of $x_i$ is not equal to $\infty$.
Then we have 
\[
\frac{b}{d}=x_1, \quad  \frac{a}{c} = x_2, \quad \frac{a+b}{c+d} = x_3.
\]
Assume $d=1$, then we have $b=x_1$, $c=\frac{x_3-x_1}{x_2-x_3}$ and $a=c x_2$.
Therefore $A$ is equal to
\[
\begin{pmatrix} x_2(x_3-x_1) & x_1(x_2-x_3) \\ (x_3-x_1) & (x_2-x_3) \end{pmatrix}
\]
up to scalar multiplication.
This still holds even if one of $x_i$ is $\infty$ in an appropriate way.
Let $X$ be an element of $\GLC$ which sends $(x_1,x_2,x_3)$ to $(x'_1, x'_2,x'_3)$. 
Since $XA$ sends $(0, \infty, 1)$ to $(x'_1, x'_2,x'_3)$, $XA$ is equal to
\[
B = \begin{pmatrix} x'_2(x'_3-x'_1) & x'_1(x'_2-x'_3) \\ (x'_3-x'_1) & (x'_2-x'_3) \end{pmatrix}
\]
up to scalar multiplication.
Therefore $X =B A^{-1} \in \mathrm{PGL}(2,\mathbb{F})$ is uniquely determined. 
\end{proof}

In this paper, we define the \emph{cross ratio} for four distinct points $x_0,x_1,x_2,x_3$ of $\mathbb{F}P^1$ by 
\begin{equation}
\label{eq:cross_raio}
[x_0:x_1:x_2:x_3] = \frac{x_3-x_0}{x_3-x_1}\frac{x_2-x_1}{x_2-x_0} \in ( \mathbb{F} \setminus \{0,1\} ).
\end{equation}
The cross ratio is invariant under the action of $\mathrm{PGL}(2,\mathbb{F})$, i.e. $[x_0:x_1:x_2:x_3] = [A x_0:A x_1:A x_2:A x_3]$ holds 
for any $A \in \mathrm{PGL}(2,\mathbb{F})$.

In the later sections, most of the arguments can be applied for $\mathrm{PGL}(2,\mathbb{F})$, 
but we will work in the case of $\PGLC$ for simplicity. 

\subsection{Character varieties}
\label{subsec:char_var}
Let $M$ be a manifold.
In this paper, we denote the set of all representations of $\pi_1(M)$ into $\SLC$ by $R_{SL}(M)$.
Since $\SLC$ is an affine algebraic group, $R_{SL}(M)$ is an affine algebraic set.
A representation $\rho$ is called \emph{reducible} if $\rho(\pi_1(M))$ fixes a point of $\mathbb{C}P^1$, otherwise it is called \emph{irreducible}.
The group $\SLC$ acts on $R_{SL}(M)$ by conjugation.
Since the action is algebraic, we can define the algebraic quotient $X_{SL}(M)$ of $R_{SL}(M)$.
This is called the \emph{character variety} because it can be regarded as the set of characters (see \cite{culler-shalen} for details).
If we restrict to the irreducible representations, $X_{SL}(M)$ is nothing but the usual quotient by the action of $\SLC$ \cite{culler-shalen}. 

Similarly we define the $\PSLC$-character variety (see \cite{heusener-porti}, for details).
We denote the set of all representations of $\pi_1(M)$ into $\PSLC$ by $R_{PSL}(M)$.
As before, a representation $\rho$ is called reducible if $\rho(\pi_1(M))$ fixes a point of $\mathbb{C}P^1$, otherwise it is called irreducible.
Since $\PSLC$ also acts on $R_{PSL}(M)$, we can define the algebraic quotient $X_{PSL}(M)$ of $R_{PSL}(M)$.
We can regard $X_{PSL}(M)$ as the set of the squares of the characters \cite{heusener-porti}.
As in the case of $\SLC$-character varieties, 
$X_{PSL}(M)$ is the usual quotient by the action of $\PSLC$ if we restrict to the irreducible representations 
(see \cite{heusener-porti}, \cite{porti}). 

The natural map $R_{SL}(M) \to R_{PSL}(M)$ is not surjective in general since there may exist a $\PSLC$-representation which does not lift to a $\SLC$-representation.
A $\PSLC$-representation lifts to a $\SLC$-representation if and only if the second Stiefel-Whitney class $w_2(\rho)$, 
which is defined in $H^{2}(M; \mathbb{Z}/2\mathbb{Z})$, vanishes.
Therefore if $S$ is a surface with boundary, any $\PSLC$-representation can be lifted to a $\SLC$-representation.
For a closed oriented surface $S$ of genus $g$, the evaluation of $w_2(\rho)$ at the fundamental class is calculated as follows.
$\pi_1(S)$ has the following presentation:
\[
\langle \alpha_1, \dots, \alpha_g, \beta_1, \dots, \beta_g \mid [\alpha_1,\beta_1] \dots [\alpha_g, \beta_g] = 1 \rangle.
\]
For a $\PSLC$-representation $\rho$, let $A_i = \rho(\alpha_i)$ and $B_i = \rho(\beta_i)$.
Take any lifts $\widetilde{A}_i$ and $\widetilde{B}_i$ in $\SLC$.
Then $w_2(\rho)$ evaluated at the fundamental class $[S]$ is equal to the sign of 
\[
[\widetilde{A}_1,\widetilde{B}_1] \dots [\widetilde{A}_g,\widetilde{B}_g] \in \{ \pm I \}.
\]
If a $\PSLC$-representation lifts to a $\SLC$-representation, then any other lift is obtained by the action of $H^1(M;\mathbb{Z}/2\mathbb{Z})$.
Since $H^1(M;\mathbb{Z}/2\mathbb{Z}) \cong \mathrm{Hom}(\pi_1(M), \mathbb{Z}/2\mathbb{Z})$, we can regard an element of $H^1(M;\mathbb{Z}/2\mathbb{Z})$
as a function $\epsilon : \pi_1(M) \to \{ \pm 1 \}$.
Then $\epsilon$ acts on $\rho$ by $(\epsilon \cdot \rho) (\gamma)  = \epsilon(\gamma)\rho(\gamma)$.
$H^1(M;\mathbb{Z}/2\mathbb{Z})$ freely acts on $R_{SL}(M)$.
It also acts on $X_{SL}(M)$, but not freely in general (\cite{boyer-zhang}, \cite{morgan-shalen}).

For a closed surface $S$ of genus $g > 1$,
Goldman showed in \cite{goldman} that $R_{\PSLC}(S)$ has exactly two components, one of which is the set of liftable representations 
and the other of which is the set of non-liftable representations.

%% file: sec_pants_rep.tex
%
%


\subsection{}
\label{subsec:_matrices_of_pants_rep}
Let $P$ be a three-holed sphere, which is often called a \emph{pair of pants}.
Fix a base point $*$ on $P$ and define $\gamma_i \in \pi_1(P, *)$ as indicated in Figure \ref{fig:generators_for_pants},
which satisfy $\gamma_1 \gamma_2 \gamma_3 = 1$.
We say that $\gamma_i$ goes around the boundary in the counterclockwise direction.
\begin{proposition}
\label{prop:pants_rep}
Let $\rho: \pi_1(P) \to \SLC$ be an irreducible representation such that $\rho(\gamma_i)$ has two fixed points $(x_i,y_i)$. 
Let $e_i$ be the eigenvalue of $\rho(\gamma_i)$ corresponding to $x_i$.
Then $\rho$ is described only in terms of $e_i$ and $x_i$: 
\begin{equation}
\label{eq:pants_representation}
\begin{split}
& \rho(\gamma_i) = \frac{1}{e_{i}e_{i+1}(x_{i+1}-x_{i})(x_{i+2}-x_{i})}
\begin{pmatrix}
a_{11} & a_{12} \\ 
a_{21} & a_{22}
\end{pmatrix}, \\
a_{11} &= e_{i}^2 e_{i+1} x_{i} (x_{i}-x_{i+2}) + e_{i+1} x_{i+2} (x_{i+1}-x_{i}) + e_{i} e_{i+2} x_{i} (x_{i+2}-x_{i+1}), \\
a_{12} &= x_{i} ( e_{i}^2 e_{i+1} x_{i+1} (x_{i+2}-x_{i}) + e_{i+1} x_{i+2} (x_{i}-x_{i+1}) + e_{i} e_{i+2} x_{i} (x_{i+1}-x_{i+2}) ), \\
a_{21} &= e_{i}^2 e_{i+1} (x_{i}-x_{i+2}) + e_{i+1} (x_{i+1}-x_{i}) + e_{i} e_{i+2} (x_{i+2}-x_{i+1}), \\ 
a_{22} &= e_{i}^2 e_{i+1} x_{i+1} (x_{i+2}-x_{i}) + e_{i+1} x_{i} (x_{i}-x_{i+1}) + e_{i} e_{i+2} x_{i} (x_{i+1}-x_{i+2}). \\
\end{split}
\end{equation}
The other fixed point $y_i$ of $\rho(\gamma_i)$ (the fixed point corresponding to ${e_i}^{-1}$) is given by
\begin{equation}
\label{eq:other_fixed_point}
y_{i}= \frac{e_{i}^2 e_{i+1} x_{i+1} (x_{i} - x_{i+2}) + e_{i+1} x_{i+2} (x_{i+1} - x_{i}) + e_{i} e_{i+2} x_{i} (x_{i+2}- x_{i+1})}
{e_{i}^2 e_{i+1} (x_{i} - x_{i+2}) + e_{i+1} (x_{i+1} - x_{i}) + e_{i} e_{i+2} (x_{i+2}- x_{i+1})}.
\end{equation}
\end{proposition}
\begin{proof}
Let $A$ be the matrix
\[
\begin{pmatrix} x_2(x_3-x_1) & x_1(x_2-x_3) \\ (x_3-x_1) & (x_2-x_3) \end{pmatrix},
\] 
which sends $(0,\infty,1)$ to $(x_1,x_2,x_3)$.
Then $A^{-1} \rho(\gamma_i) A$ has the fixed points $(0,y'_1)$, $(\infty,y'_2)$, $(1,y'_3)$ for $i=1,2,3$ respectively.
Therefore $A^{-1} \rho(\gamma_i) A$ are uniquely determined by (\ref{eq:two_fixed_points_case}):
\[
\begin{split}
A^{-1} \rho(\gamma_1) A =&\begin{pmatrix} e_1^{-1} & 0 \\ \frac{e_1^{-1}-e_1}{y'_1} & e_1 \end{pmatrix}, \quad
A^{-1} \rho(\gamma_2) A =\begin{pmatrix} e_2 & (e_2^{-1}-e_2)y'_2 \\ 0 & e_2^{-1} \end{pmatrix}, \\
&A^{-1} \rho(\gamma_3) A =\frac{1}{y'_3-1}\begin{pmatrix} e_3^{-1}y'_3-e_3 & (e_3-e_3^{-1})y'_3 \\ e_3^{-1}-e_3 & e_3 y'_3-e_3^{-1} \end{pmatrix}.
\end{split}
\]
From the identity $\rho(\gamma_1)\rho(\gamma_2) = \rho(\gamma_3)^{-1}$, we have
\[
y'_1 =\frac{e_1-e_1^{-1}}{e_2^{-1} {e_3} - e_1^{-1}}, \quad y'_2=\frac{e_2 - e_1e_3^{-1}}{e_2-e_2^{-1}}, \quad y'_3= \frac{e_2-e_1 e_3^{-1}}{e_2 -e_1 e_3}.
\]
Computing $\rho(\gamma_i) = A(A^{-1} \rho(\gamma_i) A)A^{-1}$ and $y_i = A \cdot y'_i$, we obtain the result.
\end{proof}

\begin{figure}
\input{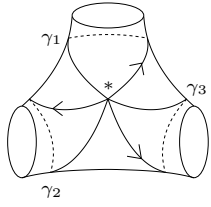}
\caption{A pair of pants.}
\label{fig:generators_for_pants}
\end{figure}

It is known that the conjugacy class of an irreducible $\SLC$-representation of the fundamental group of a pair of pants is 
uniquely determined by the triple $(\tr(\rho(\gamma_1)), \tr(\rho(\gamma_2)), \tr(\rho(\gamma_3)))$. 
Therefore $\rho$ is uniquely determined by $e_i$ up to conjugation.
Proposition \ref{prop:pants_rep} means that additional information about $x_i$ determines the representation in the conjugacy class.

Conversely the representation given by (\ref{eq:pants_representation}) is irreducible for generic values.
\begin{proposition}
\label{prop:domain}
Let $x_1$, $x_2$ and $x_3$ be distinct points on $\mathbb{C}P^1$.
For any triple $(e_1,e_2,e_3)$ $(e_i \neq 0, \pm 1)$, the representation given by (\ref{eq:pants_representation}) is irreducible
unless $e_1 = e_2 e_3$, $e_2 = e_3 e_1$, $e_3 = e_1 e_2$ or $e_1 e_2 e_3 = 1$.
If $e_1 = e_2 e_3$ (resp. $e_2 = e_3 e_1$, $e_3 = e_1 e_2$, $e_1 e_2 e_3 = 1$), 
then $x_1 = y_2 = y_3$ (resp. $y_1 = x_2 = y_3$, $y_1 = y_2 = x_3$, $y_1 = y_2 = y_3$) and therefore $\rho$ is reducible.
\end{proposition}
\begin{proof}
Since $\gamma_1 \gamma_2 \gamma_3 = 1$, if two of $\{\rho(\gamma_i)\}_{i=1,2,3}$ have a common fixed point, 
then the other matrix also fixes the point.
Therefore $\rho$ is reducible if and only if one of the following identities holds:
\[
x_1=y_2=y_3, \quad y_1=x_2=y_3, \quad y_1=y_2=x_3, \quad y_1=y_2=y_3.
\]
($x_i = x_j$ for $i \neq j$ does not occur by assumption.)
We only consider the first case since others easily follow from symmetry.
Since $x_1 = y_2$, we have 
\[
x_1 = \frac{{e_2}^2 e_3 x_3 (x_2 - x_1) + e_3 x_1 (x_3 - x_2) + e_2 e_1 x_2 (x_1- x_3)}
{{e_2}^2 e_3 (x_2 - x_1) + e_3 (x_3 - x_2) + e_2 e_1 (x_1 - x_3)}
\]
by (\ref{eq:other_fixed_point}).
Since all $x_i$ are distinct, this is equivalent to $e_1 = e_2 e_3$.
A similar calculation shows that $y_2 = y_3$ is equivalent to $(1-e_1e_2e_3)(e_1-e_2e_3)=0$.
Therefore $x_1=y_2=y_3$ if and only if $e_1=e_2 e_3$.
\end{proof}

Since the conjugacy class of an irreducible $\SLC$-representation is uniquely determined by the triple 
\[
(\tr(\rho(\gamma_1)), \tr(\rho(\gamma_2)), \tr(\rho(\gamma_3))) = (e_1+{e_1}^{-1}, e_2+{e_2}^{-1}, e_3+{e_3}^{-1}),
\]
we conclude that 
\begin{corollary}
\label{cor:conj_classes_of_pants_group_rep_into_SL}
The set of conjugacy classes of irreducible $\SLC$-representations of the fundamental group of a pair of pants can be identified with
\[
\begin{split}
\{ (e_1,e_2,e_3) \in (\mathbb{C} \setminus \{0, \pm 1\})^3 \mid e_1 \neq e_2 e_3, \quad e_2 \neq e_3 e_1, \\  
e_3 \neq e_1 e_2, \quad e_1 e_2 e_3 \neq 1 \} / (\mathbb{Z}/2\mathbb{Z})^3
\end{split}
\]
where $(\epsilon_1, \epsilon_2, \epsilon_3) \in (\mathbb{Z} /2 \mathbb{Z})^3$ acts as 
\begin{equation}
\label{eq:action_of_Z/2Z^3}
(e_1, e_2, e_3) \mapsto ({e_1}^{\epsilon_1}, {e_2}^{\epsilon_2}, {e_3}^{\epsilon_3}).
\end{equation}
\end{corollary}

\subsection{$\PSLC$-character variety of a pair of pants}
\label{subsec:pants_rep_of_PSL2C}
An $\SLC$-representation given by (\ref{eq:pants_representation}) reduces to a $\PSLC$-representation.
Conversely, any $\PSLC$-representation of the fundamental group of a pair of pants can be lifted to a $\SLC$-representation
since we can choose signs of $\rho(\gamma_1)$ and $\rho(\gamma_2)$ arbitrary (then the sign of $\rho(\gamma_3)$ is automatically determined).
The representation given by (\ref{eq:pants_representation}) can be regarded as a lift of the $\PSLC$-representation to a $\SLC$-representation.
As mentioned in \S \ref{subsec:char_var}, any other lift is obtained by the action of $H^1(P; \mathbb{Z}/2\mathbb{Z})$.
An element of $H^1(P; \mathbb{Z}/2\mathbb{Z}) = \mathrm{Hom}(\pi_1(P); \mathbb{Z}/2\mathbb{Z})$ can be regarded as 
a triple $(\epsilon_1, \epsilon_2, \epsilon_3) \in \{\pm 1\}^3$ satisfying $\epsilon_1 \epsilon_2 \epsilon_3 = 1$, 
and the action is given by
\begin{equation}
\label{eq:action_of_H^1(P,Z/2Z)}
(e_1,e_2,e_3) \mapsto (\epsilon_1 e_1, \epsilon_2 e_2, \epsilon_3 e_3).
\end{equation}
(Observe that $y_i$ in (\ref{eq:other_fixed_point}) is invariant under the action of $H^1(P; \mathbb{Z}/2\mathbb{Z})$, since the fixed points 
do not depend on the choice of a lift.)
Therefore we have: 
\begin{corollary}
\label{cor:conj_classes_of_pants_group_rep_into_PSL}
The set of conjugacy classes of irreducible $\PSLC$-representations of the fundamental group of a pair of pants can be identified with 
\[
\begin{split}
( \{ (e_1,e_2,e_3) \in (\mathbb{C} \setminus \{0, \pm 1\})^3 \mid e_1 \neq e_2 e_3, \quad e_2 \neq e_3 e_1, \\  
e_3 \neq e_1 e_2, \quad e_1 e_2 e_3 \neq 1 \} / (\mathbb{Z}/2\mathbb{Z})^3 / H^1(P; \mathbb{Z}/2\mathbb{Z}),
\end{split}
\]
where the action of $(\mathbb{Z}/2\mathbb{Z})^3$ is the one given by (\ref{eq:action_of_Z/2Z^3}) and the action of $H^1(P; \mathbb{Z}/2\mathbb{Z})$
is given by (\ref{eq:action_of_H^1(P,Z/2Z)}).
\end{corollary}
We remark that the actions of $(\mathbb{Z}/2\mathbb{Z})^3$ and $H^1(P; \mathbb{Z}/2\mathbb{Z})$ commute.

%% file: sec_pants_decomp.tex
%
%


In Section \ref{sec:rep_of_pants}, we parametrized representations of the fundamental group of a pair of pants by their eigenvalues.
For a general surface, we decompose it into pairs of pants and consider the restriction of a representation on each pair of pants.
Then we shall describe how two representations are glued along their boundary curves.
To describe gluing process, we need an additional information: a \emph{dual graph} to the pants decomposition.

In this section, we define a pants decomposition and a graph dual to the pants decomposition.
Taking a maximal tree of a dual graph, we obtain a presentation of the surface group.
In Section \ref{sec:matrix_generators}, we shall use this presentation to describe the matrix generators for our parametrization.

\subsection{Pants decomposition}
\label{subsec:pants_decomposition}
Let $S = S_{g,b}$ be a surface of genus $g$ with $b$ boundary components.
In the following, we assume that the Euler characteristic of $S$ is negative ($2-2g-b < 0$).
A \emph{pants decomposition} $C$ is a disjoint union of simple closed curve on $S$ 
such that $S \setminus N(C)$ is a collection of three holed spheres (pairs of pants), where $N(C)$ is a small open neighborhood of $C$.
Then the number of simple closed curves of $C$ is equal to $3g-3+b$.
We say that a component of $C$ an \emph{interior pants curve},
and a component of $\partial S$ a \emph{boundary pants curve}, or simply boundary curve.
We call a simple close curve which is either an interior pants curve or a boundary curve a \emph{pants curve}.
We will denote the interior pants curves of $C$ by $c_i$ $(i=1,\dots, 3g-3+b)$ 
and the boundary curves of $S$ by $b_i$ $(i=1,\dots, b)$.

We fix a pants decomposition $C = c_1 \cup \dots \cup c_{3g-3+b}$.
We will parametrize the representations $\rho : \pi_1(S) \to \SLC$ (resp. $\rho : \pi_1(S) \to \PSLC$) satisfying the following two conditions:
\begin{itemize}
\item[(C1)] 
For $\gamma_i \in \pi_1(S)$ whose free homotopy class represents $c_i$, $\rho(\gamma_i)$ has exactly two fixed points on $\mathbb{C}P^1$ for $i= 1, \dots , 3g-3+b$, 
and also for boundary curves $b_i$ $(i=1,\dots, b)$.  
\item[(C2)] 
For any pair of pants $P \subset S \setminus N(C)$, the restriction $\rho|_{\pi_1(P)}$ is irreducible.
\end{itemize}
We denote the subset of $R_{SL}(S)$ (resp. $R_{PSL}(S)$) satisfying (C1) and (C2) by $R_{SL}(S,C)$ (resp. $R_{PSL}(S,C)$).
We remark that these representations are generic in the character variety.
\begin{proposition}
$R_{SL}(S,C)$ (resp. $R_{PSL}(S,C)$) is an open algebraic subset of $R_{SL}(S)$ (resp. $R_{PSL}(S)$).
\end{proposition}
\begin{proof}
First, we show that the complement of $R_{SL}(S,C)$ is a closed algebraic subset.
In \cite[Corollary 1.2.2]{culler-shalen}, it is shown that a representation $\rho: \Gamma \to \SLC$ is reducible if and only if $\tr(\rho(c)) = 2$ for all $c \in [\Gamma,\Gamma]$.
Applying this criterion to $\pi_1(P)$ for each pair of pants $P \subset S \setminus N(C)$, the set of the representations which do not satisfy (C2) is a closed algebraic subset. 
Let $\gamma_i$ be an element of $\pi_1(S)$ whose conjugacy class is represented by $c_i$. 
Since $\rho(\gamma_i)$ has two fixed points if and only if $\tr(\rho(\gamma_i)) \neq \pm 2$, the set of the representations which do not satisfy (C1) 
is a closed algebraic subset. 
Therefore the complement of $R_{SL}(S,C)$ is closed in $R_{SL}(S)$.

In \cite{heusener-porti}, it is shown that $\rho : \Gamma \to \PSLC$ is reducible if and only if $\tr([\rho(\gamma), \rho(\eta)]) = 2$ for any $\gamma, \eta \in \Gamma$, 
where the trace is defined by lifting $\rho(\gamma)$ and $\rho(\eta)$ to $\SLC$.
Therefore we can similarly conclude that $R_{PSL}(S,C)$ is an open algebraic subset of $R_{PSL}(S)$.
\end{proof}

When $S$ has boundary, $\pi_1(S)$ is a free group of rank $2g+b-1$. 
Thus $R_{SL}(S) \cong \SLC^{2g+b-1}$, in particular, an irreducible algebraic variety.
So $R_{SL}(S,C)$ is an open algebraic subset of an irreducible variety, therefore open dense in $R_{SL}(S)$.

It is known that there exists a pants decomposition satisfying these two conditions for any non-elementary representation \cite{gallo-kapovich-marden}.

We denote the image of $R_{SL}(S,C)$ (resp. $R_{PSL}(S,C)$) to $X_{SL}(S)$ (resp. $X_{PSL}(S)$) by $X_{SL}(S,C)$ (resp. $X_{PSL}(S,C)$).
Since $R_{SL}(S,C)$ (resp. $R_{PSL}(S,C)$) consists of reducible representations, $X_{SL}(S,C)$ (resp. $X_{PSL}(S,C)$) is the quotient space of 
of $R_{SL}(S,C)$ (resp. $R_{PSL}(S,C)$) by conjugation.
In the next few sections, we shall give a parametrization of $X_{SL}(S,C)$ and $X_{PSL}(S,C)$.

\subsection{ Dual graph}
\label{subsec:dual_graph}
Let $C$ be a pants decomposition. 
Let $G$ be a graph with only trivalent or univalent vertices. 
We assume that all edges of $G$ are oriented.
We call an edge of $G$ whose endpoints are both trivalent vertices an \emph{interior edge} and 
an edge one of whose endpoints is univalent a \emph{boundary edge}.
An embedding $g : G \to S$ is \emph{dual} to the pants decomposition $C$ if the trivalent vertices are mapped to interior points of $S \setminus N(C)$ and  
the univalent vertices are mapped to $\partial S$ satisfying the following properties:
\begin{itemize}
\item For each component $P$ of $S \setminus N(C)$, there is exactly one trivalent vertex $v$ such that $g(v) \in \mathrm{Int}(P)$. 
\item For each component $b_i$ of $\partial S$, there is exactly one univalent vertex $v$ such that $g(v) \in b_i$. 
\item Every interior edge of $G$ intersects $C$ exactly once transversely.
\item Every boundary edge of $G$ does not intersect $C$.
\end{itemize}
\begin{figure}
\input{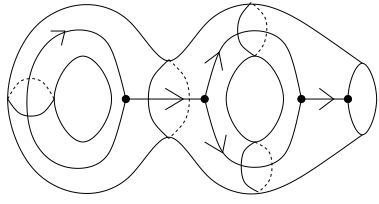}
\caption{A dual graph of a pants decomposition.}
\label{fig:dual_graph}
\end{figure}
An example of a dual embedding is given in Figure \ref{fig:dual_graph}.
From the definition, $G$ has $2g-2+b$ trivalent vertices, $b$ univalent vertices and $3g-3+2b$ edges.
We say that two dual graphs $(g_1,G)$ and $(g_2,G)$ are equivalent if there exists a homeomorphism $h$ isotopic to the identity making the following diagram commute: 
\[
\xymatrix@1@R=5pt{
  & S \ar[dd]^h \\
G \ar[ru]^{g_1} \ar[rd]_{g_2} & \\
  & S
}
\]
Here we mean an isotopy preserving $\partial S$ setwise. 
(We may use an isotopy preserving the boundary pointwise when we consider a subsurface of $S$ in later sections.) 
We simply call an equivalent class of a pair $(g,G)$ a \emph{dual graph} and usually omit the embedding $g$.
Since $S$ is oriented, the embedding $g$ induces a cyclic ordering on the edges at each trivalent vertex of $G$. 
A graph with a cyclic ordering on the edges at each vertex is usually called a \emph{fat graph}.

For a maximal tree $T$ of the dual graph $G$, we give a presentation of $\pi_1(S)$ as follows.
(Concrete examples are given in Section \ref{sec:example}.)
Let $S_0$ be the surface obtained by removing the pants curves intersecting $G \setminus T$.
Since $G \setminus T$ consists of $g$ oriented edges, $S_0$ is a sphere with $2g+b$ holes.
We denote the edges $G \setminus T$ by $u_1, \dots ,u_g$. 
Let $s_{i,+}$ be the component of $\partial S_0$ which $u_i$ enters, and $s_{i,-}$ the component of $\partial S_0$ from which $u_i$ emanates.  
We take a base point $*$ on $T$.
Let $\alpha_i$ (resp. $\alpha_{g+i}$) $(i=1,\dots, g)$ be an element of $\pi_1(S_0,*)$ which starts at $*$ following $G$, 
then goes along $s_{i,+}$ (resp. $s_{i,-}$) once in the counterclockwise direction and ends at $*$ following $G$.
We denote the boundary components of $S$ by $b_1, \dots, b_b$.
Let $\delta_i$ be an element of $\pi_1(S_0,*)$ which starts at $*$ following $G$, then goes along $b_i$ once in the counterclockwise direction and ends at $*$ following $G$.
Then $\alpha_i$, $\beta_i$, $\delta_i$ form a system of generators of $\pi_1(S_0,*)$ with one relation
\begin{equation}
\label{eq:as_HNN_extensions}
\alpha_{1} \alpha_{g+1} \dots \alpha_{g} \alpha_{g+g} \delta_1 \dots \delta_b = 1
\end{equation}
where in general the order of the product is a permutation of $\{\alpha_1, \dots,\alpha_{2g}, \delta_1, \dots, \delta_b \}$ 
depending on the choice of $G$ and $T$. 
A presentation of $\pi_1(S,*)$ is obtained from $\pi_1(S_0,*)$ by HNN extensions.
Actually, let $\beta_i$ be a path following $G$ which starts at $*$ through $u_i$ in the direction and ends at $*$, 
then $\{\alpha_1, \dots, \alpha_{2g}, \beta_1,\dots ,\beta_g\}$ satisfy the relations
\[
{\alpha_{g+i}}^{-1} = {\beta_i}^{-1} \alpha_{i} \beta_i \quad (i=1, \dots g), \\
\]
and $\pi_1(S,*)$ has the following presentation: 
\[
\begin{split}
\langle \alpha_1, \dots, \alpha_{2g}, \beta_1, \dots, \beta_g, \delta_1, \dots, \delta_b \mid 
&{\alpha_{g+i}}^{-1} = {\beta_i}^{-1} \alpha_{i} \beta_i \quad (i=1, \dots g), \\
&\alpha_1 \alpha_{g+1} \alpha_2 \alpha_{g+2} \dots \alpha_g \alpha_{2g} \delta_1 \dots \delta_b = 1 \rangle.
\end{split}
\]
Eliminating $\alpha_{g+i}$ $(i=1,\dots, g)$, we obtain a presentation
\[
\langle \alpha_1,\beta_1, \dots, \alpha_g, \beta_g, \delta_1, \dots, \delta_b \mid  [\alpha_1,\beta_1^{-1}] \cdots [\alpha_g,\beta_g^{-1}] \delta_1 \cdots \delta_b = 1 \rangle,
\]
although the relation may differ from the above one depending on the choice of $G$ and $T$.

We will introduce our parameters in Section \ref{sec:eigenvalue_paramter} and Section \ref{sec:twist_paramter}.
Then we will give explicit matrix generators corresponding to $\alpha_i$, $\beta_i$ and $\delta_j$ 
for our parametrization of $\PSLC$- or $\SLC$-representations in Section \ref{sec:matrix_generators}.

%% file: sec_eigen.tex
%
%
\subsection{$\SLC$ representations}
\label{subsec:eigen_param_for_SL2C}
Let $C = c_1 \cup \dots \cup c_{3g-3+b}$ be a a pants decomposition and $G$ a dual graph.
We denote the components of $\partial S$ by $b_1, \dots, b_b$.
Let $\rho: \pi_1(S) \to \SLC$ be a representation satisfying the two conditions (C1)--(C2).
To each interior edge of $G$ which intersects the interior pants curve $c_i$, we assign one of the eigenvalues of $\rho(c_i)$.
(We remark that the set of the eigenvalues of $\rho(\gamma)$ for $\gamma \in \pi_1(S)$ only depends on the unoriented free homotopy class of $\gamma$.)
To each boundary edge of $G$ whose univalent vertex is on $b_i$, we assign one of the eigenvalues of $\rho(b_i)$.
In this way, we assign a complex number $e_i$ to each edge of $G$.
We call $e_i$ an \emph{eigenvalue parameter} of the edge.
Since there are two choices of the eigenvalues ${e_i}^{\pm 1}$ for each edge, 
there are $2^{3g-3+2b}$ choices of eigenvalue parameters. 

Define a subset $E(S,C)$ of $\mathbb{C}^{3g-3+2b}$ by
\begin{equation}
\label{eq:irred_subset_for_pants}
\begin{split}
E(S,C) = \{ (e_1, e_2, \dots, e_{3g-3+2b}) &\mid e_i \in \mathbb{C} \setminus \{ 0, \pm 1\} ,  \\ 
& e_i^{\pm 1} e_j^{\pm 1} e_k^{\pm 1} \neq 1
\textrm{ for $\{e_i,e_j,e_k\} \in \mathcal{P}$} \}
\end{split}
\end{equation}
where $\mathcal{P}$ is the set of triples of eigenvalue parameters each of which belongs to a pair of pants  $P \subset S \setminus N(C)$.
By Proposition \ref{prop:domain}, the set of the eigenvalue parameters is contained in $E(S,C)$.
Since eigenvalues are invariant under conjugation, we obtained a multivalued (one to $2^{3g-3+2b}$) map $X_{SL}(S,C) \to E(S,C)$.
By Corollary \ref{cor:conj_classes_of_pants_group_rep_into_SL}, we have a map 
\[
X_{SL}(S,C) \to E(S,C)/(\mathbb{Z}/2\mathbb{Z})^{3g-3+2b} 
\]
where the action of $(\mathbb{Z}/2\mathbb{Z})^{3g-3+2b}$ on $E(S,C)$ is defined by (\ref{eq:action_of_Z/2Z^3}).

\subsection{$\PSLC$ representations}
Let $\rho$ be a $\PSLC$-representation satisfying the conditions (C1)--(C2).
To each edge of $G$, we can assign one of the eigenvalues as in the $\SLC$ case.
But there are two choices of signs for each pants curve, since the matrix is well-defined up to sign.
To fix the signs, we take a lift of the restriction of $\rho$ to each pair of pants $P \subset S \setminus N(C)$ 
and assign the eigenvalues of the lifted representation.
But we may not be able to fix signs globally, since there are two pairs of pants adjacent to a pants curve in general.
\begin{proposition}
\label{prop:to_define_signs_of_eigenvalues}
Let $\rho : \pi_1(S) \to \PSLC$ be a representation satisfying (C1) and (C2).
One can choose the signs of the eigenvalue parameters globally if and only if $\rho$ lifts to a $\SLC$-representation.
\end{proposition}
\begin{proof}
When $S$ has non-empty boundary, there is nothing to prove in this case since any $\PSLC$-representation of $\pi_1(S)$ lifts to an $\SLC$-representation.
So we assume that $S$ is a closed surface of genus $g$.
If $\rho$ lifts to an $\SLC$-representation, it is clear that we can choose the signs of the eigenvalue parameters globally.
Conversely we assume that the signs of the eigenvalue parameters are globally defined.
We take a maximal tree $T$ of $G$ and let $S_0$ be the sphere with $2g$ holes constructed in \S \ref{subsec:dual_graph}.
Let $\alpha_i$ $(i=1, \dots, 2g)$ be the elements of $\pi_1(S_0)$ and $\beta_i$ $(i=1,\dots,g)$ be the elements of $\pi_1(S)$ defined in \S \ref{subsec:dual_graph}.
We denote the eigenvalue parameter corresponding to the free homotopy class of $\alpha_i$ by $e_i$ for $i=1, \dots, g$.
We choose a lift $\widetilde{\rho}(\alpha_i)$ of $\rho(\alpha_i)$ to $\SLC$ so that $\tr(\widetilde{\rho}(\alpha_i)) = e_i + {e_i}^{-1}$.
Since $\alpha_{g+i}$ is a conjugation of ${\alpha_i}^{-1}$ in $\pi_1(S)$, we have $\tr( \rho(\alpha_{i}) ) =\tr( \rho(\alpha_{g+i}) )$ up to sign.
We choose a lift $\widetilde{\rho}(\alpha_{g+i})$ of $\rho(\alpha_{g+i})$ to $\SLC$ so that $\tr(\widetilde{\rho}(\alpha_{g+i})) = e_i + {e_i}^{-1}$ for $i=1,\dots,g$.
Next we choose a lift of $\widetilde{\rho}(\beta_i)$ to $\SLC$ for $i=1, \dots, g$ arbitrary.
Now we have $\widetilde{\rho}(\alpha_{g+i})^{-1} = \widetilde{\rho}(\beta_i)^{-1} \widetilde{\rho}(\alpha_i) \widetilde{\rho}(\beta_i)$ for $i=1,\dots, g$.
By our choice of the signs, they satisfy
\[
[\widetilde{\rho}(\alpha_1),\widetilde{\rho}(\beta_1)^{-1}] \cdots [\widetilde{\rho}(\alpha_g),\widetilde{\rho}(\beta_g)^{-1}] 
= \widetilde{\rho}(\alpha_{1}) \widetilde{\rho}(\alpha_{g+1}) \dots \widetilde{\rho}(\alpha_{g}) \widetilde{\rho}(\alpha_{g+g}) 
= I
\]
in $\SLC$.
(As we remarked in \S \ref{subsec:dual_graph}, the order of the product is a permutation of $\widetilde{\rho}(\alpha_1), \dots, \widetilde{\rho}(\alpha_g)$ 
and $\widetilde{\rho}(\beta_1), \dots, \widetilde{\rho}(\beta_g)$ in general depending on the choice of $G$ and $T$.)
This means that $\widetilde{\rho}(\rho)$ gives an $\SLC$-representation.
\end{proof}

Once we fix the signs of the eigenvalue parameters, the other lifts are obtained by the action of the following group. 
Consider the set
\[
\epsilon(S,C) = \{ (\epsilon_1, \dots, \epsilon_{3g-3+2b}) \mid \epsilon_i = \pm 1, \quad  \epsilon_i \epsilon_j \epsilon_k = 1 \textrm{ for } \{e_i, e_j, e_k\} \in \mathcal{P} \}, 
\] 
where $\mathcal{P}$ is the set defined in the previous subsection.
This set has a group structure by componentwise multiplication, and acts on $E(S,C)$ by 
$(e_1, \dots, e_{3g-3+2b}) \mapsto (\epsilon_1 e_1, \dots, \epsilon_{3g-3+2b} e_{3g-3+2b})$.
All lifts of $\rho$ are related by the action of $\epsilon(S,C)$.
In \cite{fujita}, a group similar to $\epsilon(S,C)$ can be described in terms of the homology of $G$.
In our case, we have the following.
\begin{lemma}
\label{lem:epsilon(S,C)_as_homology}
$\epsilon(S,C)$ is canonically isomorphic to $H_1(G,\partial G ;\mathbb{Z}/2 \mathbb{Z})$ ($H_1(G;\mathbb{Z}/2 \mathbb{Z})$ if $S$ has no boundary), 
where $\partial G$ is the subset consisting of the univalent vertices. 
\end{lemma}
\begin{proof}
Let $(\epsilon_1, \dots, \epsilon_{3g-3+2b}) \in \epsilon(S,C)$.
Assign $\epsilon_i$ to the corresponding edge of $G$, we can regard $(\epsilon_1, \dots, \epsilon_{3g-3+2b})$ as a 1-chain with $\mathbb{Z}/2\mathbb{Z}$-coefficients.
Since we have $\epsilon_i \epsilon_j \epsilon_k = 1$ for $\{e_i, e_j, e_k\} \in \mathcal{P}$, this chain is a cycle.
Conversely a 1-cycle of $G$ with $\mathbb{Z}/2\mathbb{Z}$-coefficients gives an element of $\epsilon(S,C)$. 
\end{proof}
As a consequence of Lemma \ref{lem:epsilon(S,C)_as_homology}, we have $\epsilon(S,C) \cong (\mathbb{Z}/2\mathbb{Z})^{g+b-1}$ if $S$ has non-empty boundary and 
$\epsilon(S,C) \cong (\mathbb{Z}/2\mathbb{Z})^{g}$ if $S$ is closed.

We denote the subset of $R_{PSL}(S,C)$ consisting of the liftable representations by $R'_{PSL}(S,C)$
and denote the quotient of $R'_{PSL}(S,C)$ by conjugation by $X'_{PSL}(S,C)$.
As in the case of $X_{SL}(S,C)$, we have a map
\[
X'_{PSL}(S,C) \to ( E(S,C)/H_1(G,\partial G;\mathbb{Z}/2 \mathbb{Z}) )/ (\mathbb{Z}/2\mathbb{Z})^{3g-3+2b}.
\]

\subsection{Non-liftable $\PSLC$ representations}
We denote the subset of $R_{PSL}(S,C)$ consisting of the non-liftable representations by $R''_{PSL}(S,C)$
and denote the quotient of $R''_{PSL}(S,C)$ by conjugation by $X''_{PSL}(S,C)$.
So we have $R_{SL}(S,C) = R'_{SL}(S,C) \sqcup R''_{SL}(S,C)$ and $X_{SL}(S,C) = X'_{SL}(S,C) \sqcup X''_{SL}(S,C)$.
$R''_{PSL}(S,C)$ (resp. $X''_{PSL}(S,C)$) is empty if $S$ has a boundary.

For a non-liftable representation, we can not choose the signs of the eigenvalues globally, 
but we also define the eigenvalue parameters as follows.

We assume that $S$ is a closed surface of genus $g$. (Thus $g>0$.)
We take a pants decomposition $C$ and a dual graph $G$.
Let $\rho : \pi_1(S) \to \PSLC$ be a non-liftable representation satisfying (C1)-(C2).
We fix one interior edge $f_1$ of $G$.
Consider the surface $S'$ obtained from $S$ cut along the interior pants curve corresponding to $f_1$.
We denote the two boundary components of $S'$ by $c_{1,+}$ and $c_{1,-}$.
Since $S'$ has a boundary, we have a lift $\widetilde{\rho} : \pi_1(S') \to \SLC$ of $\rho$. 
Since $\rho : \pi_1(S) \to \PSLC$ does not lift to an $\SLC$-representation, we have $\tr(\widetilde{\rho}(c_{1,+})) = - \tr(\widetilde{\rho}(c_{1,-}))$ 
by Proposition \ref{prop:to_define_signs_of_eigenvalues}.
We define the eigenvalue parameter for each of the other edges $f_i$ $(i \neq 1)$ by one of the eigenvalue of $\widetilde{\rho}(c_i)$ 
where $c_i$ is the pants curve corresponding to $f_i$. 
We define the eigenvalue parameter for $f_1$ by one of the eigenvalue of $\widetilde{\rho}(c_{1,+})$. 

The group $\epsilon(S,C) \cong H_1(G,\mathbb{Z}/2 \mathbb{Z})$ also acts on the parameter space $E(S,C)$ and we also have a map
\[
X''_{PSL}(S,C) \to ( E(S,C)/H_1(G;\mathbb{Z}/2 \mathbb{Z}) )/(\mathbb{Z}/2\mathbb{Z})^{3g-3+2b} .
\]

%% file: sec_twist.tex
%
%


In this section, we introduce the twist parameters of a surface group representation with a pants decomposition and a dual graph, which describes how 
two representations of the pants group are combined along their common interior pants curve.

\subsection{}
\label{subsec:definition_of_twist_parameters}
Let $C$ be a pants decomposition of a surface $S=S_{g,b}$ and $G$ a dual graph. 
Assume $\rho: \pi_1(S) \to \SLC$ satisfies the conditions (C1)--(C2).
We take a system of eigenvalue parameters $e_i$ for all pants curves.
For an interior edge $f_1$ of $G$, we will define the twist parameter at $f_1$.
Let $c_1 \subset C$ be the interior pants curve transverse to $f_1$.
$c_1$ is contained in a four-holed sphere or a one-holed torus.
First we consider the four-holed sphere case.

Let $S' \subset S$ be the four-holed sphere containing $c_1$ in the interior.
The restriction of the dual graph $G$ to $S'$ gives a dual graph $G'$ on $S'$ by adding univalent vertices on the boundary if necessary.
Let $P_1$ and $P_2$ be the two pairs of pants of $S' \setminus N(c_1)$.
We assume that the interior edge $f_1$ directed from $P_1$ to $P_2$.
We regard the trivalent vertices of $G'$ as a base point of $P_1$ and $P_2$. 
We take generators $\gamma_1, \gamma_2, \gamma_3 \in \pi_1(P_1)$ and $\gamma_1',\gamma_2',\gamma_3' \in \pi_1(P_2)$ 
as in Figure \ref{fig:generators_for_pants} so that both $\gamma_1$ and $\gamma_1'$ run around $c_1$.
Connect the base points along $G'$, we obtain a set of generators of $\pi_1(S')$. 
We rewrite $\gamma_2'$ and $\gamma_3'$ by $\gamma_4$ and $\gamma_5$ and $\gamma_1'$ can be identified with $\gamma_1^{-1}$ (Figure \ref{fig:twist_parameter}).
Now we have $\gamma_1\gamma_2\gamma_3 =1$ and $\gamma_1^{-1} \gamma_4 \gamma_5 = 1$, and therefore $\gamma_2\gamma_3 \gamma_4 \gamma_5 = 1$.

Let $\rho'$ be the restriction $\rho|_{\pi_1(S')}$.
This is only well-defined up to conjugation since there are many ways to choose a path from the base point of $S$ to $S'$, but we fix one in the conjugacy class.
Let $f_1$ be the interior edge of $G'$ and $f_2,f_3,f_4,f_5$ be the boundary edges of $G'$ corresponding to $\gamma_2, \gamma_3, \gamma_4, \gamma_5$.
If $f_i$ is outward (resp. inward) oriented, let $x_i$ be the fixed point of $\rho'(\gamma_i)$ corresponding to $e_i$ (resp. $e_i^{-1}$).
We assign the element $x_i \in \mathbb{C}P^1$ to each edge of $G'$.
Since $\rho'(\gamma_1) = M(e_1; x_1,y_1) = M({e_1}^{-1};x_1,y_1)^{-1} = \rho'(\gamma_1)^{-1}$ where $y_1$ is the other fixed point of $\rho'(\gamma_1)$, 
the assignment is well-defined at $f_1$. 

By Lemma \ref{lem:two_fixed_pts_with_one_point_behavior}, there exists a unique complex number $t_1 \in \mathbb{C}^*$ such that
\begin{equation}
\label{eq:gluing_map}
M(\sqrt{-t_1};x_1,y_1) = 
\pm \begin{pmatrix} x_1 & y_1 \\ 1 & 1 \end{pmatrix} 
\begin{pmatrix} \sqrt{-t_1} & 0 \\ 0 & 1/\sqrt{-t_1} \end{pmatrix}
\begin{pmatrix} x_1 & y_1 \\ 1 & 1  \end{pmatrix}^{-1} 
\end{equation}
sends $x_2$ to $x_5$.
(We use $\sqrt{-t_1}$ instead of $\sqrt{t_1}$ since it is natural when we parametrize Fuchsian representations (Section \ref{sec:psl2r}).)
We call this $t_1$ the \emph{twist parameter} of the edge $f_1$.
When we replace $\rho'$ by a conjugation $A^{-1} \rho' A$, the fixed points $x_i$ and $y_i$ are changed to $A x_i$ and $A y_i$.
Therefore $t_1$ only depends on the conjugacy class of $\rho'$, once we fix the eigenvalue parameters $e_i$.  
\begin{theorem}
Let $e_1,\dots, e_5$ be the eigenvalue parameters, $t_1$ be the twist parameter and $x_1,\dots, x_5$ the fixed points of 
$\rho(\gamma_1), \dots ,\rho(\gamma_5)$ corresponding to $e_1, \dots, e_5$ (Figure \ref{fig:twist_parameter}).
Then
\begin{equation}
\label{eq:x4}
\begin{split}
x_4 = \frac{a_1}{a_2}&, \\
a_1 &=
{e_1}(
-({e_1}{e_3}-{e_2})({e_1}{e_4}-{e_5}){t_1}
+{e_3}({e_1}{e_5}-{e_4}) )
{x_1}({x_2}-{x_3}) \\
&+{e_1}^{2}{e_2}({e_1}{e_5}-{e_4}){x_2}({x_3}-{x_1}) 
+{e_2}({e_1}{e_5}-{e_4}) {x_3}({x_1}-{x_2}), \\
a_2 &=
{e_1}(
-({e_1}{e_3}-{e_2})({e_1}{e_4}-{e_5}){t_1}
+{e_3}({e_1}{e_5}-{e_4}) ) ({x_2}-{x_3}) \\
&+{e_1}^{2}{e_2}({e_1}{e_5}-{e_4}) ({x_3}-{x_1})
+{e_2}({e_1}{e_5}-{e_4}) ({x_1}-{x_2}), \\
\end{split}
\end{equation}
\begin{equation}
\label{eq:x5}
x_5 = 
\frac{
(({e_1}{e_3}-{e_2}){t_1}+{e_1}{e_3} ) {x_1}({x_2}-{x_3})  
+{e_1}^{2}{e_2}{x_2}({x_3}-{x_1}) +{e_2}{x_3}({x_1} -{x_2}) 
}
{
(({e_1}{e_3}-{e_2}){t_1}+{e_1}{e_3} )({x_2}-{x_3}) 
+{e_1}^{2}{e_2}({x_3}-{x_1}) +{e_2}({x_1} -{x_2})
}.
\end{equation}
Conversely we have
\begin{equation}
\label{eq:x2}
\begin{split}
x_2 
&= \frac{ ((e_1 e_5 - e_4) {t_1}^{-1}+ e_1 e_5) x_1 (x_5-x_4) + {e_1}^2 e_4 x_5 (x_4-x_1) + e_4 x_4 (x_1-x_5)}
{ ((e_1 e_5 - e_4) {t_1}^{-1} +e_1 e_5) (x_5-x_4) +{e_1}^2 e_4 (x_4-x_1) +e_4 (x_1-x_5) }, \\
\end{split}
\end{equation}
\begin{equation}
\label{eq:x3}
\begin{split}
x_3 = \frac{a_1}{a_2}, & \\
a_1 &= e_1 (-(e_1 {e_5} -{e_4}) (e_1 {e_2} -{e_3}) {t_1}^{-1}+{e_5} (e_1 {e_3}-{e_2})) x_1 (x_5-x_4) \\
+&e_1^2 {e_4} (e_1 {e_3}-{e_2}) x_5 (x_4-x_1)+{e_4} (e_1 {e_3}-{e_2}) x_4 (x_1-x_5), \\
a_2 &= e_1 (-(e_1 {e_5}-{e_4}) (e_1 {e_2}-{e_3}) {t_1}^{-1}+{e_5} (e_1 {e_3}-{e_2})) (x_5-x_4) \\
+&e_1^2 {e_4} (e_1 {e_3}-{e_2}) (x_4-x_1)+{e_4} (e_1 {e_3}-{e_2}) (x_1-x_5). \\
\end{split}
\end{equation}
If some boundary edges are inward oriented, replace the corresponding parameters $e_i$ with $e_i^{-1}$.
\end{theorem}
\begin{proof}
Since $x_2$ is mapped to $x_5$ by the matrix (\ref{eq:gluing_map}), we have
\begin{equation}
\label{eq:tentative_for_x_5}
\begin{split}
x_5 &= \begin{pmatrix} x_1 & y_1 \\ 1 & 1 \end{pmatrix} 
\begin{pmatrix} -t_1 & 0 \\ 0 & 1 \end{pmatrix}
\begin{pmatrix} x_1 & y_1 \\ 1 & 1  \end{pmatrix}^{-1} \cdot x_2 \\
&= \frac{ (x_2-y_1) t_1 x_1 + (x_2-x_1)y_1}{(x_2-y_1)t_1+(x_2-x_1)}.
\end{split}
\end{equation}
Applying (\ref{eq:other_fixed_point}) to $\rho'|_{P_1}$ we have
\begin{equation}
\label{eq:y1_for_x5_and_x4}
y_{1}= \frac{e_{1}^2 e_{2} x_{2} (x_{1} - x_{3}) + e_{2} x_{3} (x_{2} - x_{1}) + e_{1} e_{3} x_{1} (x_{3}- x_{2})}
{e_{1}^2 e_{2} (x_{1} - x_{3}) + e_{2} (x_{2} - x_{1}) + e_{1} e_{3} (x_{3}- x_{2})}.
\end{equation}
Substitute $y_1$ of (\ref{eq:y1_for_x5_and_x4}) into (\ref{eq:tentative_for_x_5}), we obtain (\ref{eq:x5}).
Applying (\ref{eq:other_fixed_point}) to $\rho'|_{P_2}$, we have
\[
y_{1}= \frac{e_{1}^{-2} e_{4} x_{4} (x_{1} - x_{5}) + e_{4} x_{5} (x_{4} - x_{1}) + e_{1}^{-1} e_{5} x_{1} (x_{5}- x_{4})}
{e_{1}^{-2} e_{4} (x_{1} - x_{5}) + e_{4} (x_{4} - x_{1}) + e_{1}^{-1} e_{5} (x_{5}- x_{4})}.
\]
Solve the equation for $x_4$, we have
\begin{equation}
\label{eq:tentative_for_x4}
x_{4}= \frac{e_{1} e_{5} x_{5} (x_{1} - y_{1}) + e_{1}^{2} e_{4} x_{1} (y_{1} - x_{5}) + e_{4} y_{1} (x_{5}- x_{1})}
{e_{1} e_{5} (x_{1} - y_{1}) + e_{1}^{2} e_{4} (y_{1} - x_{5}) + e_{4} (x_{5}- x_{1})}.
\end{equation}
Substitute $y_1$ of (\ref{eq:y1_for_x5_and_x4}) and $x_5$ of (\ref{eq:x5}) into (\ref{eq:tentative_for_x4}), we obtain (\ref{eq:x4}).
Similarly (\ref{eq:x2}) and (\ref{eq:x3}) follow from direct calculation.
\end{proof}

If $c_1$ is on a one-holed torus $S'$, 
we take a covering corresponding to the circle formed by the edge $f_1$ (the right of Figure \ref{fig:twist_for_one_holed}). 
Then $S'$ is lifted to a four-holed sphere and we can define the twist parameter same as in the four-holed sphere case. 

Combined with the eigenvalue parameters, we have constructed a multivalued map
\[
X_{SL}(S,C) \to E(S,C) \times (\mathbb{C}^*)^{3g-3+b}. 
\]
This is a multivalued map, but once we fix the eigenvalue parameters $\{e_i\}_{i=1,\dots, 3g-3+2b}$, 
the twist parameters $t_i$ only depends on the conjugacy class of $\rho$.
We will consider the change of the twist parameters $\{t_j\}_{j=1,\dots, 3g-3+b}$ under the action of $(\mathbb{Z}/2\mathbb{Z})^{3g-3+2b}$
(it acts as $e_i \mapsto e_i^{-1}$ on $E(S,C)$) in Section \ref{sec:action}.

Since $t_1$ is defined in terms of fixed points, it is also defined for a $\PSLC$-representation satisfying (C1)--(C2) 
once we fix the eigenvalue parameters $\{e_i\}_{i=1,\dots, 3g-3+2b}$.
Since the fixed points does not depend on the choice of a lift to an $\SLC$-representation, the action of $H_1(G,\partial G; \mathbb{Z}/2\mathbb{Z})$ on 
the twist parameters is trivial.
(This can also be checked from the formulae of Lemma \ref{lem:twist_parameter_from_fixed_points}.)
So we also have a multivalued map
\[
X'_{PSL}(S,C) \to (E(S,C)/H_1(G,\partial G;\mathbb{Z}/2\mathbb{Z})) \times (\mathbb{C}^*)^{3g-3+b}.
\]
This induces a map 
\begin{equation}
\label{eq:paramtrization_of_PSL2C_characters}
X'_{PSL}(S,C) \to ( (E(S,C)/H_1(G,\partial G;\mathbb{Z}/2\mathbb{Z})) \times (\mathbb{C}^*)^{3g-3+b} ) / (\mathbb{Z}/2\mathbb{Z})^{3g-3+2b}.  
\end{equation}

\begin{figure}
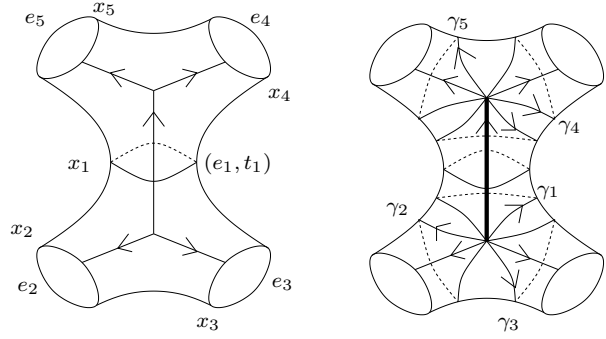

\input{4_holed_with_gr_1.pstex_t}
\hspace{30pt}
\input{gens_4_holed.pstex_t}
\caption{A pants decomposition and a dual graph of the four-holed sphere $S$.
We take the base point $*$ on the central edge and generators $\gamma_1, \dots, \gamma_5 $ of $\pi_1(S, *)$ as above.  
We have $\gamma_1 \gamma_2 \gamma_3 = \gamma_1^{-1} \gamma_4 \gamma_5 = 1$.
Here the matrix which sends $(x_1,x_2,\rho(\gamma_1) x_2)$ to $(x_1, x_5,\rho(\gamma_1) x_5)$ has eigenvalues $\sqrt{-t_1}^{\pm 1}$.}
\label{fig:twist_parameter}
\end{figure}

\begin{figure}
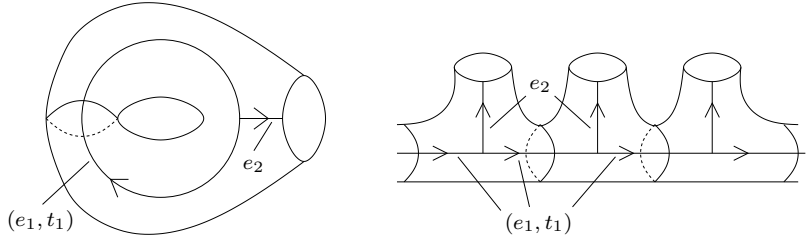

\input{one_holed_with_gr_a.pstex_t}
\hspace{20pt}
\input{one_holed_cover.pstex_t}
\caption{For a one-holed torus, take a covering and define the twist parameter as in the four-holed sphere case.}
\label{fig:twist_for_one_holed}
\end{figure}

We end this section by giving formulae on $t_1$ which will be needed later.
\begin{lemma}
\label{lem:twist_parameter_from_fixed_points}
Let $e_1, \dots, e_5$ be the eigenvalue parameters, the twist parameter $t_1$ and $x_1, \dots, x_5$ the fixed points as before.
Then we have
\begin{align}
\label{eq:t1_from_x1_x2_x3_x5}
t_1 &= -1 +\frac{e_2(1-e_1^2)}{e_2-e_1e_3}[x_5:x_3:x_1:x_2], \\
\label{eq:t1_from_x1_x2_x3_x4}
t_1 &= -\frac{e_1e_5-e_4}{e_1(e_1e_4-e_5)} \left( -1 +\frac{e_2(1-e_1^2)}{e_2-e_1e_3}[x_4:x_3:x_1:x_2] \right), \\
\label{eq:t1_from_x1_x2_x4_x5}
t_1^{-1} &= -1 +\frac{e_4(1-e_1^2)}{e_4-e_1e_5}[x_2:x_4:x_1:x_5], \\
\label{eq:t1_from_x1_x3_x4_x5}
t_1^{-1} &= -\frac{e_1e_3-e_2}{e_1(e_1e_2-e_3)} \left( -1 +\frac{e_4(1-e_1^2)}{e_4-e_1e_5}[x_3:x_4:x_1:x_5] \right). 
\end{align}
\end{lemma}
\begin{proof}
From (\ref{eq:x5}), we have
\[
\begin{split}
(-({e_2}-{e_1}\*{e_3}){t_1}+{e_1}\*{e_3} ) \*(x_1-x_5)({x_2}-{x_3})  
+{e_1}^{2}\*{e_2}\*(x_2-x_5)({x_3}-{x_1})  \\
+{e_2}\*(x_3-x_5)({x_1} -{x_2}) =0.
\end{split}
\]
Since $(x_1-x_5)({x_2}-{x_3})+(x_2-x_5)({x_3}-{x_1}) = -(x_3-x_5)({x_1} -{x_2})$, we have
\[
\begin{split}
t_1 &= \frac{1}{e_2-e_1e_3} \left( e_1 e_3-e_2 - ({e_1}^2e_2-e_2)[x_5:x_3:x_1:x_2]  \right) \\
&= -1 + \frac{e_2(1-e_1^2)}{e_2-e_1e_3}[x_5:x_3:x_1:x_2].
\end{split}
\]
Similarly (\ref{eq:t1_from_x1_x2_x3_x4}), (\ref{eq:t1_from_x1_x2_x4_x5}), (\ref{eq:t1_from_x1_x3_x4_x5}) follow from (\ref{eq:x4}), (\ref{eq:x3}), (\ref{eq:x2}) respectively.
\end{proof}

%% file: sec_generators.tex
%
%


We assigned eigenvalue parameters $e_i$ and twist parameters $t_i$ for a given $\SLC$- or $\PSLC$-representation satisfying (C1)--(C2).
In this section, we construct a representation from these parameters. 
Actually, we can construct explicit matrix generators of the representation.  

\subsection{}
Let $S=S_{g,b}$ be a surface and fix a pants decomposition $C$ and a dual graph $G$.
We assign a complex number $e_i \in \mathbb{C} \setminus \{0,\pm 1\}$ to each edge of $G$.
We assume that $(e_1,\dots, e_{3g-3+2b})$ is in $E(S,C)$ to ensure that the restriction to each subgroup corresponding to a pair of pants $P \subset S \setminus N(C)$ 
to be irreducible (Proposition \ref{prop:domain}).
We also assign a twist parameter $t_i \in \mathbb{C}^*$ to each interior edge of $G$.
Fix a maximal tree $T$ in $G$, then we obtain a presentation of $\pi_1(S)$ as in \S \ref{subsec:dual_graph}.
We denote the oriented edges of $G \setminus T$ by $u_1,\dots u_g$.
Let $\alpha_i, \beta_i$ and $\delta_i$ be generators of $\pi_1(S)$ defined in \S \ref{subsec:dual_graph}.
We will give matrices $\rho(\alpha_i)$, $\rho(\beta_i)$ and $\rho(\delta_i)$ satisfying the relations given in \S \ref{subsec:dual_graph}.

Let $\widetilde{G}$ be the universal covering of $G$. 
(If $g=0$, $\widetilde{G}$ is $G$ itself.)
Each edge of $\widetilde{G}$ inherits an orientation and eigenvalue and twist parameters from those of $G$.
For each trivalent vertex of $G$, we fix a counterclockwise (with respect to the orientation of $S$) order on the edges adjacent to the vertex.
The order at a trivalent vertex of $G$ lifts to an order of the edges adjacent to a trivalent vertex of $\widetilde{G}$.
(So $G$ and $\widetilde{G}$ are not only graphs but also ``fat graphs''.)
Let $\widetilde{T}$ be a lift of $T$ to $\widetilde{G}$.
Fix a trivalent vertex $p$ of $\widetilde{T}$.
We assign three distinct points of $\mathbb{C}P^1$ to the three edges adjacent to $p$ 
to keep track of the fixed points corresponding to the boundary curves of a pair of pants. 
Using (\ref{eq:x4}) and (\ref{eq:x5}) (or (\ref{eq:x2}) and (\ref{eq:x3})) inductively, we assign elements of $\mathbb{C}P^1$ to all edges of $\widetilde{G}$, 
especially all the edges of $\widetilde{T}$ and the edges adjacent to them. 
We call them the \emph{fixed point parameters} associated to the edges of $\widetilde{G}$.

For each vertex $v$ of $\widetilde{G}$, we assign a representation of the fundamental group of a pair of pants as follows.
Let $(x_1,x_2,x_3)$ be the fixed point parameters of the three edges adjacent to $v$ in the order.
We let $(e_1,e_2,e_3)$ be the eigenvalue parameter or its inverse depending on whether $v$ is an initial or terminal vertex of the edge. 
From $(e_1,e_2,e_3)$ and $(x_1,x_2,x_3)$, we can construct an $\SLC$-representation of the pair of pants uniquely (not up to conjugation) by Proposition \ref{prop:pants_rep}.
Consider two vertices of $\widetilde{G}$ joined by an edge.
There are two representations corresponding to these vertices.
At a common interior pants curve, they have same fixed point but inverse eigenvalues each other.
Recall that $M(e;x,y)= M(e^{-1};x,y)^{-1}$. 
These representations can be combined into a representation of four-holed sphere.
Construct representations for all vertices of $\widetilde{T}$, we obtain a representation $\rho : \pi_1(S_0) \to \SLC$,
where $S_0$ is the surface obtained from $S$ by removing the interior pants curves intersecting $G \setminus T$ (see \S \ref{subsec:dual_graph}).
Thus $\rho$ satisfies the relation
\[
\rho(\alpha_{1}) \rho(\alpha_{g+1}) \dots \rho(\alpha_{g}) \rho(\alpha_{g+g}) \rho(\delta_1) \dots \rho(\delta_b) = I.
\]
(As we remarked in \S \ref{subsec:dual_graph}, the relation may be a permutation of the above one depending to the choice of $G$ and $T$.)

Finally we will construct the matrices $\rho(\beta_i)$.
There are two lifts of $u_i$ to $\widetilde{G}$ which are adjacent to $\widetilde{T}$.
We denote the one which emanates from $\widetilde{T}$ by $\widetilde{u}_{i,+}$ 
and the one which terminates at $\widetilde{T}$ by $\widetilde{u}_{i,-}$ (see Figure \ref{fig:rho_beta}). 
Let $\widetilde{v}_{i,+}$ (resp. $\widetilde{v}_{i,-}$) be the initial vertex of $\widetilde{u}_{i,+}$ (resp. $\widetilde{u}_{i,-}$).
We remark that $\widetilde{v}_{i,+}$ and $\widetilde{v}_{i,-}$ are projected to the same vertex of $G$.
We let $e_i$ and $x_{i,1}$ be the eigenvalue parameter and the fixed point parameter associated to $\widetilde{u}_{i,+}$.
Then we let $(x_{i,1}, x_{i,2}, x_{i,3})$ be the fixed point parameters of the edges adjacent to the vertex $\widetilde{v}_{i,+}$ in counterclockwise order 
(see Figure \ref{fig:rho_beta}). 
Similarly, let $x'_{i,1}$ be the fixed point parameter associated to $\widetilde{u}_{i,-}$ and 
$(x'_{i,1}, x'_{i,2}, x'_{i,3})$ be the fixed point parameters of the edges adjacent to the vertex $\widetilde{v}_{i,+}$ in counterclockwise order.
We remark that the eigenvalue parameter associated to $\widetilde{u}_{i,-}$ is $e_i$ since $\widetilde{u}_{i,+}$ and $\widetilde{u}_{i,-}$ are projected to the same edge of $G$.
We define $\rho(\beta_i)$ by the unique element of $\PSLC$ which sends $(x'_{i,1},x'_{i,2},x'_{i,3})$ to $(x_{i,1},x_{i,2},x_{i,3})$ determined by Lemma \ref{lem:thrice_transitive}.
(We remark that if we define $\widetilde{v}_{i,+}$ (resp. $\widetilde{v}_{i,-}$) by the terminal vertex of $\widetilde{u}_{i,+}$ (resp. $\widetilde{u}_{i,-}$) instead, 
the resulting matrix $\rho(\beta_i)$ does not change.)
We fix a lift of $\rho(\beta_i)$ to $\SLC$ although there is no canonical choice.
Recall that $x_{i,1}$ (resp. $x'_{i,1}$) is one of the fixed points of $\rho(\alpha_i)$ (resp. $\rho(\alpha_{g+i})$) corresponding to the eigenvalue $e_i$ (resp. ${e_i}^{-1}$).  
We denote the other fixed point of $\rho(\alpha_i)$ (resp. $\rho(\alpha_{g+i})$) by $y_{i,1}$ (resp. $y'_{i,1}$).
Since we associated the fixed point parameters to the edges of $\widetilde{T}$ equivariantly, 
we have $\rho(\beta_i) y'_{i,1} = y_{i,1}$.
Thus we have
\[
\begin{split}
\rho(\alpha_{g+i})^{-1} &= M({e_i}^{-1}; x'_{i,1},y'_{i,1}) ^{-1}
= M({e_i}; x'_{i,1},y'_{i,1})  \\
&= \rho(\beta_i)^{-1} M({e_i}; x_{i,1}, y_{i,1}) \rho(\beta_i) 
= \rho(\beta_i)^{-1} \rho(\alpha_i) \rho(\beta_i).
\end{split}
\]
This means that $\rho$ satisfies the relations (\ref{eq:as_HNN_extensions}), thus $\rho$ gives rise to an $\SLC$-representation of $\pi_1(S)$.
Although $\rho$ depends on the choices of the signs of $\rho(\beta_i)$ as an $\SLC$-representation, 
it is uniquely determined as a $\PSLC$-representation by the eigenvalue and twist parameters and 
the three points of $\mathbb{C}P^1$ assigned to the edges adjacent to $p$.

\begin{figure}
\input{rho_beta.pstex_t}
\caption{}
\label{fig:rho_beta}
\end{figure}

If we change the assignment of the first three points of $\mathbb{C}P^1$ to the edges adjacent to $p$, 
the resulting fixed point parameters of $\widetilde{G}$ are obtained by the action of an element of $\PSLC$ from the original one  
since (\ref{eq:x4})--(\ref{eq:x3}) are equivariant with respect to the action of $\PSLC$. 
Thus the resulting representation differs from the original one only by conjugation as $\PSLC$-representations.
As a result, we have the following. 
\begin{theorem}
\label{thm:main_PSLC}
Let $(e_i,t_j) \in E(S,C) \times (\mathbb{C}^*)^{3g-3+b}$.
There exists a $\PSLC$-representation whose eigenvalue parameters and twist parameters coincide with $(e_i,t_j)$ 
up to the action of $(\mathbb{Z}/2\mathbb{Z})^{3g-3+2b}$ and $H_1(G,\partial G; \mathbb{Z}/2\mathbb{Z})$.
This gives a rational map
\[
E(S,C) \times (\mathbb{C}^*)^{3g-3+b} \to X'_{PSL}(S,C)
\]
where $E(S,C)$ is the subset of $\mathbb{C}^{3g-3+2b}$ defined in (\ref{eq:irred_subset_for_pants}).
$H_1(G, \partial G; \mathbb{Z}/2\mathbb{Z})$ acts on $E(S,C)$ and the above map induces
\[
(E(S,C) / H_1(G, \partial G;\mathbb{Z}/2\mathbb{Z}) ) \times (\mathbb{C}^*)^{3g-3+b} \to X'_{PSL}(S,C).
\]
Furthermore $(\mathbb{Z}/2\mathbb{Z})^{3g-3+2b}$ acts on the left hand side and induces a bijective map 
\begin{equation}
\label{eq:inverse_of_paramtrization_of_PSL2C_characters}
( (E(S,C) / H_1(G, \partial G;\mathbb{Z}/2\mathbb{Z}) ) \times (\mathbb{C}^*)^{3g-3+b} )/(\mathbb{Z}/2\mathbb{Z})^{3g-3+2b} \to X'_{PSL}(S,C).
\end{equation}
\end{theorem}
\begin{proof}
The last statement follows from the fact that the map (\ref{eq:inverse_of_paramtrization_of_PSL2C_characters}) gives 
the inverse of the map (\ref{eq:paramtrization_of_PSL2C_characters}).
\end{proof}

We can also construct a rational map to $X''_{PSL}(S,C)$, but in this case the construction is more complicated since there is no natural choice of the signs of 
the eigenvalue parameters.

For a point of $E(S,C) \times (\mathbb{C}^*)^{3g-3+b}$, 
we have constructed a homomorphism $\pi_1(S_0) \to \SLC$ and $\rho(\beta_i) \in \PSLC$ which give a $\PSLC$-representation of $\pi_1(S)$.
If we fix a lift $\widetilde{\rho}(\beta_i) \in \SLC$ of $\rho(\beta_i)$, we obtain an $\SLC$-representation of $\pi_1(S)$.
Since any other lift of $\rho(\beta_i)$ is written as $\epsilon_i \widetilde{\rho}(\beta_i) \in \SLC$ for some $\epsilon_i = \pm 1$ and  
$\{\epsilon_i\}_{i=1,\dots,g}$ can be regarded as an element of $\Hom(H_1(G), \mathbb{Z}/2\mathbb{Z})$,  
we can construct a covering map $Y \to E(S,C) \times (\mathbb{C}^*)^{3g-3+b}$ with covering group $H^1(G;\mathbb{Z}/2\mathbb{Z}) \cong \Hom(H_1(G), \mathbb{Z}/2\mathbb{Z})$ 
satisfying the following commutative diagram:
\[
\xymatrix{
Y \ar[r] \ar[d]^{H^1(G;\mathbb{Z}/2\mathbb{Z})} & R_{SL}(S,C) \ar[r] \ar[d] & X_{SL}(S,C) \ar[d] \\
E(S,C) \times (\mathbb{C}^*)^{3g-3+b} \ar[r] & R'_{PSL}(S,C) \ar[r] & X'_{PSL}(S,C) \\
}
\]
As a result, we have the following theorem.
\begin{theorem}
\label{thm:main_SLC}
There exists a covering map 
\[
Y \to E(S,C) \times (\mathbb{C}^*)^{3g-3+b} 
\]
with covering group $H^1(G; \mathbb{Z}/2\mathbb{Z})$ such that there is a map $Y \to X_{SL}(S,C)$ which induces a bijection
\[
Y / (\mathbb{Z}/2\mathbb{Z})^{3g-3+2b} \to X_{SL}(S,C).
\]
\end{theorem}
\begin{proof}
We construct an inverse of $Y / (\mathbb{Z}/2\mathbb{Z})^{3g-3+2b} \to X_{SL}(S,C)$ as follows.
For an element $[\rho]$ of $X_{SL}(S,C)$, fix a representative $\rho$ in $R_{SL}(S,C)$.
Choose one of the eigenvalues of $\rho(c)$ for each pants curve $c$.
Then the twist parameters $\{t_i\}$ are uniquely determined.
By the signs of $\rho(\beta_i)$, we determine a point of $Y$.
If we choose other eigenvalue parameters, they are related by the action of $(\mathbb{Z}/2\mathbb{Z})^{3g-3+2b}$. 
Therefore the map $X_{SL}(S,C) \to Y / (\mathbb{Z}/2\mathbb{Z})^{3g-3+2b}$ is well-defined.
Conversely we can reconstruct an $\SLC$-representation from these parameters as discussed before.
\end{proof}

Our results are summarized in the following diagram:
\[
\xymatrix{
Y \ar[rr]^{(\mathbb{Z}/2\mathbb{Z})^{3g-3+2b}}  \ar[d]^{H^1(G;\mathbb{Z}/2\mathbb{Z})} & & X_{SL}(S,C) \ar[dd]^{H^1(S;\mathbb{Z}/2\mathbb{Z})} \ar@{^{(}->}[r] & X_{SL}(S) \ar[dd] \\
E(S,C) \times (\mathbb{C}^*)^{3g-3+b} \ar[d]^{H_1(G, \partial G;\mathbb{Z}/2\mathbb{Z})} & &  & \\
(E(S,C)/H_1(G, \partial G;\mathbb{Z}/2\mathbb{Z}) ) \times (\mathbb{C}^*)^{3g-3+b} \ar[rr]^{\textrm{\hspace{65pt}}(\mathbb{Z}/2\mathbb{Z})^{3g-3+2b}} 
& &  X'_{PSL}(S,C) \ar@{^{(}->}[r] & X_{PSL}(S) \\
}
\]
We remark that 
\[
X_{SL}(S,C) \to X'_{PSL}(S,C) 
\]
and
\[
(E(S,C)/H_1(G,\partial G;\mathbb{Z}/2\mathbb{Z}) )\times (\mathbb{C}^*)^{3g-3+b} \to X'_{PSL}(S,C) 
\]
are not covering maps in general but just quotient maps by the action of $H^1(S;\mathbb{Z}/2\mathbb{Z})$ and $(\mathbb{Z}/2\mathbb{Z})^{3g-3+2b}$ respectively.
The first one was remarked in \cite{morgan-shalen} and the second one is because $-e_i = 1/e_i$ if $e_i = \pm \sqrt{-1}$.
The above diagram induces the following diagram:
\[
\xymatrix{
Y/(\mathbb{Z}/2\mathbb{Z})^{3g-3+2b} \ar[r]^{\cong}  \ar[d]^{H^1(G;\mathbb{Z}/2\mathbb{Z})} & X_{SL}(S,C) \ar[dd]^{H^1(S;\mathbb{Z}/2\mathbb{Z})} \ar@{^{(}->}[r] & X_{SL}(S) \ar[dd] \\
(E(S,C) \times (\mathbb{C}^*)^{3g-3+b})/(\mathbb{Z}/2\mathbb{Z})^{3g-3+2b} \ar[d]^{H_1(G,\partial G;\mathbb{Z}/2\mathbb{Z})} &  & \\
( (E(S,C)/H_1(G,\partial G;\mathbb{Z}/2\mathbb{Z}) ) \times (\mathbb{C}^*)^{3g-3+b} )/(\mathbb{Z}/2\mathbb{Z})^{3g-3+2b}
 \ar[r]^{\textrm{\hspace{90pt}}\cong} &  X'_{PSL}(S,C) \ar@{^{(}->}[r] & X_{PSL}(S) \\
}
\]
Again we remark that  
\[
\begin{split}
(E(S,C) \times &(\mathbb{C}^*)^{3g-3+b})/(\mathbb{Z}/2\mathbb{Z})^{3g-3+2b} \\
&\to ( (E(S,C)/H_1(G,\partial G;\mathbb{Z}/2\mathbb{Z}) ) \times (\mathbb{C}^*)^{3g-3+b} )/(\mathbb{Z}/2\mathbb{Z})^{3g-3+2b}
\end{split}
\]
is just a quotient map by the action of $H^1(G, \partial G;\mathbb{Z}/2\mathbb{Z})$.

%% file: sec_examples.tex
%
%


In this section, we apply the construction of \S \ref{sec:matrix_generators} to $S_{0,4}$, $S_{1,1}$ and $S_{2,0}$.
The first two examples play an important role in the transformation formulae in \S \ref{sec:transformation}. 

\subsection{Four-holed sphere}
\label{subsec:four_holed_sphere}
Take a pants decomposition $C$ and a dual graph $G$ of $S_{0,4}$ as in Figure \ref{fig:twist_parameter}. 
Then $G$ itself is a maximal tree and $\widetilde{G} = G$.
We assign a pair of eigenvalue and twist parameters $(e_1,t_1)$ to the interior edge of $G$,  
and eigenvalue parameters $e_2, e_3, e_4, e_5$ to the boundary edges of $G$.
Let $x_i$ be the fixed point parameter of the edge corresponding to $e_i$.
We assume that the fixed point parameters corresponding to the lower trivalent vertex are $(x_1,x_2,x_3) = (\infty, 1, 0)$.
We apply (\ref{eq:x4}) and (\ref{eq:x5}) for $(x_1,x_2,x_3) = (\infty, 1, 0)$ and $(e_1,e_2,e_3,e_4,e_5) = (e_1,e_2,e_3,e_4, e_5)$.
Then we have 
\[
\begin{split}
x_4 &= \frac{{e_1}\*({e_1}\*{e_3}-{e_2})\*({e_1}\*{e_4}-{e_5})\*{t_1}+{e_1}\*({e_1}\*{e_2}-{e_3})\*({e_1}\*{e_5}-{e_4})}{({e_1}^{2}-1)\*{e_2}\*({e_1}\*{e_5}-{e_4})}, \\
x_5 &= \frac{-({e_1}\*{e_3}-{e_2})\*{t_1}+{e_1}\*({e_1}\*{e_2}-{e_3})}{({e_1}^{2}-1)\*{e_2}}.
\end{split}
\]
Applying (\ref{eq:pants_representation}) for $(e_1,e_2,e_3)$ and $(\infty,1,0)$, we have
\begin{equation}
\label{eq:gamma_1_2_3_for_four_holed}
\begin{split}
\rho(\gamma_1) = \begin{pmatrix} e_1 & \frac{e_3}{e_2}-{e_1} \\ 0 & \frac{1}{e_1} \end{pmatrix},& \quad 
\rho(\gamma_2) = \begin{pmatrix} -\frac{e_1}{e_3}+{e_2}+\frac{1}{e_2} & \frac{e_1}{e_3}-\frac{1}{e_2} \\
{e_2}-\frac{e_1}{e_3} & \frac{e_1}{e_3} \\ \end{pmatrix}, \\
\rho(\gamma_3) =& \begin{pmatrix} \frac{1}{e_3} & 0 \\ \frac{1}{e_3}-\frac{e_2}{e_1} & {e_3} \end{pmatrix}.
\end{split}
\end{equation}
Applying  (\ref{eq:pants_representation}) for $({e_1}^{-1},e_4,e_5)$ and $(\infty, x_4, x_5)$, we have
\begin{equation}
\label{eq:gamma_4_for_four_holed}
\begin{split}
\rho(\gamma_4) = &\begin{pmatrix} a_{11} & a_{12} \\ a_{21} & a_{22} \end{pmatrix}, \\
a_{11} =& 
\frac{{e_1}({e_4}+{e_4}^{-1})}
{{e_1}-{e_1}^{-1}}
-\frac
{({e_5}+{e_5}^{-1})}
{{e_1}-{e_1}^{-1}}
-\frac
{(1-{e_1}{e_4}{e_5})({e_1}{e_5}-{e_4})({e_1}{e_2}-{e_3})}
{({e_1}^2-1)({e_1}{e_3}-{e_2}){e_4}{e_5}{t_1}}, \\
a_{12} =& \frac{{e_1}}{({e_1}^2-1)^{2}{e_2}({e_1}{e_3}-{e_2}){e_4}{e_5}{t_1}} \\
& \times (({e_1}{e_3}-{e_2})({e_1}{e_4}-{e_5}){t_1}+({e_1}{e_5}-{e_4})({e_1}{e_2}-{e_3})) \\
& \times (({e_1}{e_3}-{e_2})({e_4}{e_5}-{e_1}){t_1}+({e_1}{e_2}-{e_3})(1-{e_1}{e_4}{e_5})), \\
a_{21} =&
\frac{{e_2}({e_1}{e_5}-{e_4})({e_1}{e_4}{e_5}-1)}
{{e_1}({e_1}{e_3}-{e_2}){e_4}{e_5}{t_1}}, \\
a_{22} =& 
-\frac{{e_4}-{e_4}^{-1}}{{e_1}({e_1}-{e_1}^{-1})}
+\frac{{e_5}+{e_5}^{-1}}{{e_1}-{e_1}^{-1}}
+\frac{
(1-{e_1}{e_4}{e_5})({e_1}{e_5}-{e_4})({e_1}{e_2}-{e_3})}
{({e_1}^2-1)({e_1}{e_3}-{e_2}){e_4}{e_5}{t_1}}.
\end{split}
\end{equation}
We omit the matrix $\rho(\gamma_5)$ since it is equal to $(\rho(\gamma_2) \rho(\gamma_3) \rho(\gamma_4))^{-1}$.
We compute some traces for later use.
\begin{equation}
\label{eq:four_holed_tr_1}
\begin{split}
\tr(\rho(\gamma_3 \gamma_4)& ) = \frac{1}{({e_1}-{e_1}^{-1})^{2}} \biggl( 
-\frac{(e_2e_3-e_1)(e_1e_3-e_2)}{e_1 e_2 e_3}\frac{(e_4e_5-e_1)(e_1e_4-e_5)}{e_1 e_4 e_5} \*{t_1} \\
&-\frac{(1-e_1e_2e_3)(e_1e_2-e_3)}{e_1 e_2 e_3}\frac{(1-e_1e_4e_5)(e_1e_5-e_4)}{e_1 e_4 e_5} \*\frac{1}{t_1} \\
&+({e_1}+{e_1}^{-1})\*\bigl(({e_3}+{e_3}^{-1})({e_5}+{e_5}^{-1})+({e_2}+{e_2}^{-1})({e_4}+{e_4}^{-1}) \bigr) \\
&-2\bigl( ({e_2}+{e_2}^{-1})({e_5}+{e_5}^{-1})+({e_3}+{e_3}^{-1})({e_4}+{e_4}^{-1}) \bigr) \biggr). \\
\end{split}
\end{equation}
\begin{equation}
\label{eq:four_holed_tr_2}
\begin{split}
\tr(\rho(\gamma_2 \gamma_4)& ) = \frac{1}{({e_1}-{e_1}^{-1})^{2}} 
\biggl(
\frac{(e_2e_3-e_1)(e_1e_3-e_2)}{e_1 e_2 e_3}\frac{(e_4e_5-e_1)(e_1e_4-e_5)}{e_1 e_4 e_5} \*e_1 {t_1} \\
&+\frac{(1-e_1e_2e_3)(e_1e_2-e_3)}{e_1 e_2 e_3}\frac{(1-e_1e_4e_5)(e_1e_5-e_4)}{e_1 e_4 e_5} \* \frac{1}{e_1 t_1} \\
&+({e_1}+{e_1}^{-1})\bigl(
({e_2}+{e_2}^{-1})({e_5}+{e_5}^{-1})
+ ({e_3}+{e_3}^{-1})({e_4}+{e_4}^{-1}) \bigr)\\
&-2\bigl(
({e_2}+{e_2}^{-1})({e_4}+{e_4}^{-1})
+({e_3}+{e_3}^{-1})({e_5}+{e_5}^{-1})
\bigr) \biggr). \\
\end{split}
\end{equation}
\begin{equation}
\label{eq:four_holed_tr_3}
\begin{split}
\tr(\rho(\gamma_3 \gamma_5)& ) = \frac{1}{({e_1}-{e_1}^{-1})^{2}} \biggl(
\frac{(e_2e_3-e_1)(e_1e_3-e_2)}{e_1 e_2 e_3}\frac{(e_4e_5-e_1)(e_1e_4-e_5)}{e_1 e_4 e_5} \*\frac{t_1}{e_1} \\
&+\frac{(1-e_1e_2e_3)(e_1e_2-e_3)}{e_1 e_2 e_3}\frac{(1-e_1e_4e_5)(e_1e_5-e_4)}{e_1 e_4 e_5} \*\frac{e_1}{t_1} \\
&+({e_1}+{e_1}^{-1}) \bigl( (e_2+e_2^{-1})(e_5+e_5^{-1})+(e_3+e_3^{-1})(e_4+e_4^{-1}) \bigr) \\
&-2 \bigl( (e_2+e_2^{-1})(e_4+e_4^{-1})+(e_3+e_3^{-1})(e_5+e_5^{-1}) \bigr) \biggr). \\
\end{split}
\end{equation}
We remark that $\tr(\gamma_2 \gamma_4)$ is obtained from $\tr(\gamma_3 \gamma_5)$ by Proposition \ref{prop:right_Dehn_twist}
since $\gamma_2 \gamma_4 =  \gamma_1^{-1} \gamma_3^{-1} \gamma_1 \gamma_5^{-1} \sim \gamma_3^{-1} \gamma_1 \gamma_5^{-1} \gamma_1^{-1}$ 
and thus $\gamma_2 \gamma_4$ is obtained from $\gamma_3^{-1} \gamma_5^{-1}$ by the right handed Dehn twist along $\gamma_1$.
From (\ref{eq:four_holed_tr_2}) and (\ref{eq:four_holed_tr_3}), we have
\[
\begin{split}
&({e_1}-{e_1}^{-1})^{2} \bigl( e_1 \tr(\rho(\gamma_2 \gamma_4) - {e_1}^{-1} \tr(\rho(\gamma_3 \gamma_5) \bigr)\\
&=\frac{(e_2e_3-e_1)(e_1e_3-e_2)}{e_1 e_2 e_3}\frac{(e_4e_5-e_1)(e_1e_4-e_5)}{e_1 e_4 e_5} ({e_1}^2-\frac{1}{{e_1}^2}) {t_1} \\
&+({e_1}-{e_1}^{-1}) ({e_1}+{e_1}^{-1}) \bigl( (e_2+e_2^{-1})(e_5+e_5^{-1})+(e_3+e_3^{-1})(e_4+e_4^{-1}) \bigr) \\
&-2 ({e_1}-{e_1}^{-1}) \bigl( (e_2+e_2^{-1})(e_4+e_4^{-1})+(e_3+e_3^{-1})(e_5+e_5^{-1}) \bigr). \\
\end{split}
\]
Thus we have 
\begin{equation}
\label{eq:twist_from_equation_0_4}
\begin{split}
t_1 &= \frac{1}{{e_1}+{e_1}^{-1}} \frac{e_1 e_2 e_3}{(e_2e_3-e_1)(e_1e_3-e_2)}\frac{e_1 e_4 e_5}{(e_4e_5-e_1)(e_1e_4-e_5)} \\
&\biggl( ({e_1}-{e_1}^{-1}) \bigl(  e_1 \tr(\rho(\gamma_2 \gamma_4) - {e_1}^{-1} \tr(\rho(\gamma_3 \gamma_5) \bigr)\\
&-({e_1}+{e_1}^{-1}) \bigl( (e_2+e_2^{-1})(e_5+e_5^{-1})+(e_3+e_3^{-1})(e_4+e_4^{-1}) \bigr) \\
&+2 \bigl( (e_2+e_2^{-1})(e_4+e_4^{-1})+(e_3+e_3^{-1})(e_5+e_5^{-1}) \bigr) \biggr). \\
\end{split}
\end{equation}
This means that the twist parameter can be computed from traces and eigenvalue parameters.

Rewriting (\ref{eq:four_holed_tr_1}), we have  
\[
\begin{split}
\tr(\rho(\gamma_3 \gamma_4)& ) = \frac{1}{({e_1}-{e_1}^{-1})^{2}} \biggl(
-\biggl(
(\frac{e_1}{e_2}+\frac{e_2}{e_1}) - ({e_3+e_3^{-1}})
\biggr)
\biggl(
(\frac{e_1}{e_5}+\frac{e_5}{e_1}) - ({e_4+e_4^{-1}})
\biggr) \*{t_1} \\
&-\biggl(
 (e_1 e_2+\frac{1}{e_1 e_2}) - ({e_3}+{e_3}^{-1}) 
\biggr)
\biggl(
 (e_1 e_5+\frac{1}{e_1 e_5}) - ({e_4}+{e_4}^{-1}) 
\biggr)\*\frac{1}{t_1} \\
&+({e_1}+{e_1}^{-1})\*\biggl(({e_3}+{e_3}^{-1})({e_5}+{e_5}^{-1})+({e_2}+{e_2}^{-1})({e_4}+{e_4}^{-1}) \biggr) \\
&-2\biggl( ({e_2}+{e_2}^{-1})({e_5}+{e_5}^{-1})+({e_3}+{e_3}^{-1})({e_4}+{e_4}^{-1}) \biggr) \biggr). \\
\end{split}
\]
We can observe that $\tr(\rho(\gamma_3 \gamma_4))$ is invariant under ${e_3} \leftrightarrow {e_3}^{-1}$
and ${e_4} \leftrightarrow {e_4}^{-1}$.

\subsection{One-holed torus}
\label{subsec:once_punctured}
The surface $S_{1,1}$ decomposes into one pair of pants. 
We take a dual graph $G$ and give the eigenvalue parameters $e_1, e_2,$ and the twist parameter $t_1$ as in Figure \ref{fig:once_punctured}.
There is a unique maximal tree $T$ consisting of the boundary edge. 
The associated presentation of $\pi_1(S_{1,1})$ is given by
\[
\begin{split}
\langle & \alpha_1, \alpha_2, \beta_1, \delta_1 \mid \alpha_1 \delta_1 \alpha_2 = 1, \quad  \alpha_2^{-1} = \beta_1^{-1} \alpha_1 \beta_1\rangle \\
&= \langle \alpha_1, \beta_1, \delta_1 \mid  [\beta_1^{-1}, \alpha_1^{-1}] = \delta_1^{-1} \rangle
\end{split}
\]
(see Figure \ref{fig:once_punctured}).
Then we define fixed point parameters $x_1, \dots, x_5$ of some edges of $\widetilde{G}$ as in the right of Figure \ref{fig:once_punctured}.
We let $(x_1,x_2,x_3) = (\infty, 0, 1)$.
Applying (\ref{eq:pants_representation}) for $(e_1,e_2,e_1^{-1})$ and $(\infty, 0, 1)$, we have 
\[
\begin{split}
\rho(\alpha_1)  &= \begin{pmatrix} e_1 & e_1^{-1}-e_1^{-1}e_2^{-1} \\ 0 &e_1^{-1} \end{pmatrix}, \quad 
\rho(\delta_1)  = \begin{pmatrix} e_2^{-1} & 0 \\ e_1^{2}-e_2  & e_2 \end{pmatrix}, \\
\rho(\alpha_2)  &= \begin{pmatrix} e_1^{-1}e_2 & e_1^{-1}-e_1^{-1}e_2 \\  e_1^{-1}e_2 - e_1 & e_1 + e_1^{-1} - e_1^{-1}e_2 \end{pmatrix}. \\
\end{split}
\]
\begin{figure}
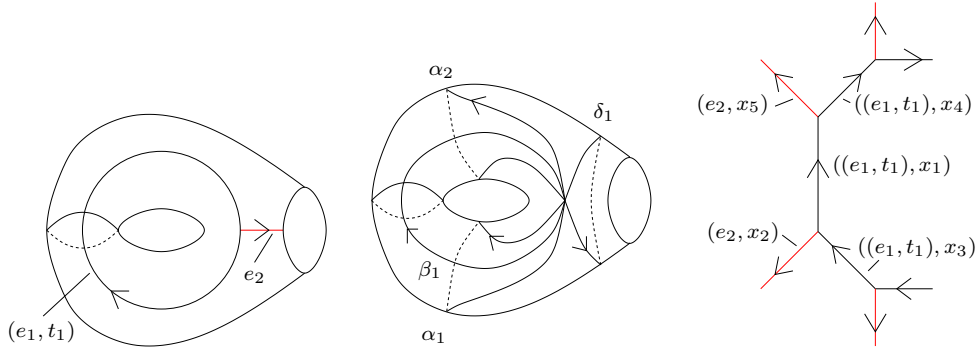

\input{one_holed_with_mx_tr_a.pstex_t}
\hspace{10pt}
\input{gens_one_holed_a.pstex_t}
\hspace{10pt}
\input{tr_one_holed_a.pstex_t}
\caption{The left is a dual graph $G$ and a maximal tree $T$ (colored red). The center indicates its associated generators of $\pi_1(S_{1,1})$.
The right is $\widetilde{G}$ and there $x_i$ are fixed point parameters.
(We put $\widetilde{G}$ on the plane to be compatible with the order at each vertex.)}
\label{fig:once_punctured}
\end{figure}
Applying (\ref{eq:x4}) and (\ref{eq:x5}) to $(e_1,e_2,e_3,e_4,e_5) = (e_1,e_2, e_1^{-1},e_1,e_2)$ and $(x_1,x_2,x_3) = (\infty, 0, 1)$, 
then we have 
\[
\begin{split}
x_4 &= \frac{e_1^2-e_2}{e_2(e_1^2-1)} t_1 + \frac{1-e_2}{e_2(e_1^2-1)}, \\
x_5 &= \frac{(t_1+1)(1-e_2)}{e_2(e_1^2-1)}. \\
\end{split}
\]
Because $\rho(\beta_1)$ is the matrix which sends $(\infty, 0, 1)$ to $(x_4,x_5,\infty)$, we have
\begin{equation}
\label{eq:rho_beta_1_for_one_holed}
\begin{split}
\rho(\beta_1) &= \frac{1}{\sqrt{-e_2 t_1}(e_1^2-1)}
\begin{pmatrix}
(e_2-e_1^2)t_1+ (e_2-1) & (t_1+1)(1-e_2) \\
-e_2(e_1^2-1) & e_2(e_1^2-1) \\
\end{pmatrix} \\
\end{split}
\end{equation}
by Lemma \ref{lem:thrice_transitive}.
Actually these matrices satisfy the equation
\[
\rho(\beta_1)^{-1} \rho(\alpha_1)^{-1} \rho(\beta_1)\rho(\alpha_1) = \rho(\delta_1)^{-1}.
\] 
If we fix a sign of $\sqrt{-e_2 t_1}$, we obtain a lift to a $\SLC$-representation.
We have 
\begin{equation}
\label{eq:one_holed_tr}
\begin{split}
\tr(\rho(\beta_1)) &=  -\frac{({e_1}^2-e_2)t_1 + 1 - e_1^2 e_2}{(e_1^2-1)\sqrt{-e_2 t_1}} \\ 
&= - \frac{1}{e_1-{e_1}^{-1}}  \biggl( (\frac{e_1}{\sqrt{-e_2}}+\frac{\sqrt{-e_2}}{e_1})\sqrt{t_1} + (e_1 \sqrt{-e_2} + \frac{1}{e_1\sqrt{-e_2}})\frac{1}{\sqrt{t_1}} \biggr), \\ 
\end{split}
\end{equation}
\begin{equation}
\begin{split}
\tr(\rho(\alpha_1 \beta_1)) &=
 -\frac{({e_1}^{2}-{e_2})\*{e_1}^{2}\*{t_1}+1-{e_1}^{2}\*{e_2}}{({e_1}^{2}-1)\*{e_1}\*\sqrt{-{e_2}\*{t_1}}} \\
&= - \frac{1}{e_1-{e_1}^{-1}}  \biggl( (\frac{e_1}{\sqrt{-e_2}}+\frac{\sqrt{-e_2}}{e_1})e_1\sqrt{t_1} 
+ (e_1 \sqrt{-e_2} + \frac{1}{e_1\sqrt{-e_2}})\frac{1}{e_1\sqrt{t_1}} \biggr) \\ 
\end{split}
\end{equation}
where we choose $\sqrt{t_1}$ and $\sqrt{-e_2}$ so that $\sqrt{t_1} \sqrt{-e_2} = \sqrt{-e_2t_1}$.
Here we have 
\[
\tr(\rho(\beta_1))  - e_1 \tr(\rho(\alpha_1 \beta_1))
= \frac{({e_1}^{2}-{e_2})(-1+{e_1}^2)t_1}{(e_1^2-1)  \sqrt{-e_2 t_1}} 
= \frac{{e_1}^{2}-{e_2}}{\sqrt{-e_2}} \sqrt{t_1}, 
\]
therefore
\begin{equation}
\label{eq:twist_from_equation_1_1}
\begin{split}
t_1 &= \frac{- e_2}{({e_1}^{2}-{e_2})^2} \left( \tr(\rho(\beta_1))  - e_1 \tr(\rho(\alpha_1 \beta_1)) \right)^2 \\
&= \frac{1}{e_1\left(\frac{e_1}{\sqrt{-e_2}}+\frac{\sqrt{-e_2}}{e_1}\right)^2} \left( \frac{1}{\sqrt{e_1}}\tr(\rho(\beta_1)) - \sqrt{e_1} \tr(\rho(\alpha_1 \beta_1)) \right)^2. \\
\end{split}
\end{equation}


When $e_2= -1$, we have 
\[
\begin{split}
\rho(\alpha_1) &= \begin{pmatrix} e_1 & 2 e_1^{-1} \\ 0 &e_1^{-1} \end{pmatrix}, \\
\rho(\beta_1) &= \frac{1}{\sqrt{t_1}(e_1^2-1)}
\begin{pmatrix} (e_1^2+1)t_1 + 2 & -2(t_1+1) \\ -e_1^2+1 & e_1^2-1 \\ \end{pmatrix}. \\
\end{split}
\]
Let $A = \rho(\alpha_1)$ and $B= \rho(\beta_1)$. The traces of $A$, $B$ and $AB$ are given by
\[
\begin{split}
\tr(A) = &e_1+e_1^{-1}, \quad \tr(B) = \frac{1}{\sqrt{t_1}} \frac{(e_1+e_1^{-1})(t_1+1)}{(e_1-e_1^{-1})},  \\
&\tr(AB) = \frac{1}{\sqrt{t_1}} \frac{ (e_1+e_1^{-1})(e_1 t_1 + e_1 ^{-1} )}{(e_1-e_1^{-1}) }.
\end{split}
\]
Clearly this triple satisfy the Markov identity 
\[
\tr(A)^2 + \tr(B)^2 +\tr(AB)^2 -\tr(A)\tr(B)\tr(AB) = 0.
\]

\subsection{Closed surface of genus 2}
\label{sec:genus2}
\begin{figure}
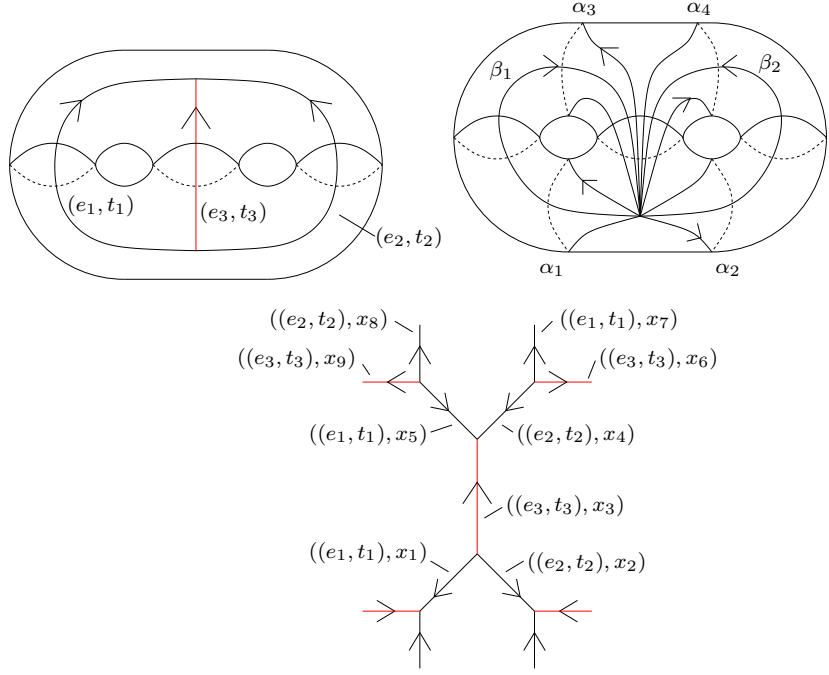

\input{genus2.pstex_t}
\hspace{20pt}
\input{gens_genus2.pstex_t}

\vspace{10pt}
\input{genus2_tr.pstex_t}
\caption{The upper left is a pants decomposition and a dual graph $G$ with a maximal tree $T$ (colored red). 
The upper right is the generators of $\pi_1(S_{2,0})$ associated to $T$.
The relations are $\alpha_3\alpha_1\alpha_2\alpha_4=1$, $\alpha_3^{-1} = \beta_1^{-1} \alpha_1 \beta_1$ and $\alpha_4^{-1} = \beta_2^{-1} \alpha_2 \beta_2$. 
The bottom is $\widetilde{G}$ and indicates fixed point parameters $x_i \in \mathbb{C}P^1$.
(We put $\widetilde{G}$ on the plane to be compatible with the order at each vertex.)}
\label{fig:genus2}
\end{figure}
Let $C$ be a pants decomposition of $S_{2,0}$, $G$ a dual graph and $T$ a maximal tree as indicated in Figure \ref{fig:genus2}.
The presentation of $\pi_1(S_{2,0})$ associated to $T$ is
\[
\begin{split}
\langle& \alpha_1, \dots, \alpha_4, \beta_1, \dots, \beta_4 \mid
\alpha_3\alpha_1\alpha_2\alpha_4=1, \quad \alpha_3^{-1}=\beta_1^{-1}\alpha_1\beta_1, \quad \alpha_4^{-1}=\beta_2^{-1}\alpha_2\beta_2 \rangle \\
&= \langle \alpha_1, \alpha_2, \beta_1, \beta_2 \mid [\beta_1^{-1},\alpha_1^{-1}][\alpha_2,\beta_2^{-1}]=1 \rangle.
\end{split}
\]
Let $(e_i,t_i)$ be the eigenvalue and twist parameters as in Figure \ref{fig:genus2}.
We assign fixed point parameters $x_1, \dots, x_9$ to some edges of the universal cover $\widetilde{G}$ of $G$ as in the bottom of Figure \ref{fig:genus2}.
Assume that $(x_1,x_2,x_3) = (\infty, 0, 1)$.
Then $x_4, \dots, x_9$ are computed by using (\ref{eq:x4})--(\ref{eq:x3}) inductively.
The matrices $\rho(\alpha_1)$ and $\rho(\alpha_2)$ are immediately obtained by applying  (\ref{eq:pants_representation}) to 
$(e_1,e_2,e_3)$ and $(x_1,x_2,x_3)=(\infty,0,1)$:
\[
\rho(\alpha_1) = \begin{pmatrix} e_1^{-1} & 0 \\ -e_1 + e_2^{-1} e_3 & e_1 \end{pmatrix}, \quad 
\rho(\alpha_2) = \begin{pmatrix} e_1 e_3^{-1} & e_2 - e_1 e_3^{-1} \\ -e_2^{-1}+ e_1 e_3^{-1} & e_2+e_2^{-1} - e_1 e_3^{-1} \end{pmatrix}, 
\]
Now $\rho(\beta_1)$ is the matrix which sends $(x_5,x_8, x_9)$ to $(\infty, 0, 1)$ and $\rho(\beta_2)$ is the matrix which sends $(x_4,x_6, x_7)$ to $(\infty, 0, 1)$.
From Lemma \ref{lem:thrice_transitive}, we obtain 
\[
\begin{split}
\rho(\beta_1) =&
\frac
{1}
{\sqrt{t_1 t_3}}
\begin{pmatrix} a_{11} & a_{12} \\ a_{21} & a_{22}\end{pmatrix}, \\
a_{11} &= 1 \\
a_{12} &=
-\frac{({e_2}{e_3}-{e_1})\*({t_3}+1)}{{e_1}({e_3}^2-1)}, \\
a_{21} &= 
\frac{{e_1}\*({t_1}+1)\*({e_1}\*{e_2}-{e_3})}{({e_1}^2-1){e_2}}, \\
a_{22} &= 
\frac
{
({e_1}{e_2}{e_3}-1)({e_1}{e_3}-{e_2}){t_1}{t_3}
-({e_1}\*{e_2}-{e_3})({e_2}{e_3}-{e_1})({t_1}+{t_3}+1)
}
{
({e_1}^2-1){e_2}\*({e_3}^2-1)
}, \\
\rho(\beta_2) =&
\frac
{1}
{({e_2}^2-1){e_3}\*\sqrt{t_2 t_3}}
\begin{pmatrix} b_{11} & b_{12} \\ b_{21} & b_{22}\end{pmatrix}, \\
b_{11} &= 
({e_1}\*{e_2}-{e_3})\*{t_2}-{e_2}({e_2}{e_3}-{e_1}),  \\
b_{12} &= 
-({e_2}{e_3}-{e_1})\*(
{e_3}\*({e_1}\*{e_2}{e_3}-1){t_2}\*{t_3}
+({e_3}-{e_1}\*{e_2})\*{t_2} \\
& \quad \quad +{e_2}{e_3}({e_2}-{e_1}{e_3})\*{t_3}
-{e_2}({e_1}-{e_2}{e_3}) 
)/(
{e_1}({e_3}^2-1)
), \\
b_{21} &= 
({e_1}\*{e_2}-{e_3}) ({t_2}+1), \\
b_{22} &=
-(
{e_3}({e_1}\*{e_2}{e_3}- 1)({e_2}{e_3}-{e_1}){t_2}\*{t_3} 
-{e_3}\*({e_1}\*{e_2}-{e_3})({e_1}{e_3}-{e_2})\*{t_3} \\
& \quad \quad +({e_1}\*{e_2}-{e_3})({e_2}{e_3}-{e_1})(1+{t_2})
)/(
{e_1}({e_3}^2-1) 
). \\
\end{split}
\]
Once we fix lifts of $\rho(\beta_i)$ to $\SLC$, we obtain an $\SLC$-representation $\rho$.
Actually we can check that these matrices satisfy the equality
\[
\rho(\beta_1)^{-1} \rho(\alpha_1)^{-1}
\rho(\beta_1) \rho(\alpha_1)
\rho(\alpha_2) \rho(\beta_2)^{-1}
\rho(\alpha_2)^{-1}  \rho(\beta_2) = I,
\]
for any $(e_1,e_2,e_3, t_1, t_2, t_3) \in (\mathbb{C} - \{0,\pm 1\})^3 \times (\mathbb{C} -\{0\})^3 $.
(To obtain representations satisfying (C2), restrict the parameters to $E(S,C) \times (\mathbb{C} -\{0\})^3$.)
Some trace functions are given as follows: 
\[
\tr(\rho(\alpha_1)) = e_1 + e_1^{-1}, \quad
\tr(\rho(\alpha_2)) = e_2 + e_2^{-1},
\]
\[
\begin{split}
\tr(\rho(\beta_1))&=
 -\frac{
( {e_2}-\frac{1}{e_1 e_3})\*({e_2}-{e_1}{e_3})({t_1}{t_3}+1)
+({e_2}-\frac{e_3}{e_1})({e_2}-\frac{e_1}{e_3})\*({t_1}+{t_3})
}
{({e_1}-{e_1}^{-1})\*{e_2}\*({e_3}-{e_3}^{-1})\*\sqrt{{t_1}\*{t_3}}}, \\
\tr(\rho(\beta_2))&=
 -\frac{
({e_1}-\frac{1}{e_2 e_3})({e_1}-{e_2}{e_3})\*({t_2}{t_3}+1)
+({e_1}-\frac{e_3}{e_2})({e_3}-\frac{e_2}{e_1})\*({t_2}+{t_3})
}
{{e_1}\*({e_2}-{e_2}^{-1})\*({e_3}-{e_3}^{-1})\*\sqrt{{t_2}\*{t_3}}}. \\
\end{split}
\]

%% file: sec_action.tex
%
%


We have constructed a map 
\[
\xymatrix{
Y \ar[r] \ar[d] & X_{SL}(S,C) \ar[d] \\
E(S,C) \times (\mathbb{C} \setminus \{0\})^{3g-3+b} \ar[r] & X'_{PSL}(S)
}
\]
where $E(S,C)$ corresponds to the eigenvalue parameters, $(\mathbb{C} \setminus \{0\})^{3g-3+b}$ corresponds to the twist parameters 
and $Y \to E(S,C) \times (\mathbb{C} \setminus \{0\})^{3g-3+b}$ is a covering map with covering group isomorphic to $H^1(G; \mathbb{Z}/2\mathbb{Z})$.
By changing the choices of eigenvalues, $(\mathbb{Z}/2\mathbb{Z})^{3g-3+2b}$ acts on $E(S,C)$.
This action also affects the twist parameters.
In this section we describe the action of the group $(\mathbb{Z}/2\mathbb{Z})^{3g-3+2b}$ on the twist parameters.

\subsection{Four-holed sphere}
In this subsection, we describe the behaviour of the twist parameter corresponding to the interior pants curve on a four-holed sphere.

We take a dual graph as indicated in Figure \ref{fig:twist_parameter} and assign eigenvalue parameters $e_1, \dots, e_5$ and twist parameter $t_1$.
We simply denote the transformation $\langle e_1, \dots, e_i, \dots, e_5\rangle \to \langle e_1, \dots, {e_i}^{-1}, \dots, e_5\rangle$ by $e_i \to {e_i}^{-1}$. 
Then $\{ e_i \to {e_i}^{-1} \}_{i=1,\dots, 5}$ generates the group $(\mathbb{Z}/2\mathbb{Z})^5$.

\begin{theorem}
Define a dual graph and eigenvalue and twist parameters as indicated in Figure \ref{fig:twist_parameter}.
Then the action of $(\mathbb{Z}/2\mathbb{Z})^5$ on the eigenvalue and twist parameters is given by
\begin{equation}
\label{eq:action_of_Z2_for_four_holed}
\begin{split}
(e_1 \to {e_1}^{-1}) \cdot (e_1, t_1) &= ({e_1}^{-1}, {t_1}^{-1}), \\
(e_2 \to {e_2}^{-1}) \cdot (e_2, t_1) &= ({e_2}^{-1}, \frac{e_2e_3-e_1}{1-e_1e_2e_3}\frac{e_1e_3-e_2}{e_1e_2-e_3} t_1), \\
(e_3 \to {e_3}^{-1}) \cdot (e_3, t_1) &= ({e_3}^{-1}, t_1), \\
(e_4 \to {e_4}^{-1}) \cdot (e_4, t_1) &= ({e_4}^{-1}, t_1), \\
(e_5 \to {e_5}^{-1}) \cdot (e_5, t_1) &= ({e_5}^{-1}, \frac{e_4e_5-e_1}{1-e_1e_4e_5}\frac{e_1e_4-e_5}{e_1e_5-e_4} t_1). \\
\end{split}
\end{equation}
(We omit the parameters which are invariant under the action.) 
If the orientation of the edge corresponding to $e_i$ $(i=2,3,4,5)$ is reversed, replace $e_i$ by ${e_i}^{-1}$ in the coefficient of $t_1$.
\end{theorem}
We remark that these commute and give an action of $(\mathbb{Z}/2\mathbb{Z})^5$.
For example, we have
\[
\begin{split}
(e_1 \to {e_1}^{-1}) & \cdot ( (e_2 \to {e_2}^{-1}) \cdot (e_1,e_2, t_1) )  \\
&= (e_1 \to {e_1}^{-1}) \cdot (e_1, {e_2}^{-1}, \frac{e_2e_3-e_1}{1-e_1e_2e_3}\frac{e_1e_3-e_2}{e_1e_2-e_3} t_1) \\
&= ({e_1}^{-1}, {e_2}^{-1}, \frac{1-e_1e_2e_3}{e_2e_3-e_1}\frac{e_1e_2-e_3}{e_1e_3-e_2} {t_1}^{-1}) \\
&= ({e_1}^{-1}, {e_2}^{-1}, \frac{e_2e_3-{e_1}^{-1}}{1-{e_1}^{-1}e_2e_3}\frac{{e_1}^{-1}e_3-e_2}{{e_1}^{-1}e_2-e_3} {t_1}^{-1}) \\
&= (e_2 \to {e_2}^{-1}) \cdot ({e_1}^{-1}, e_2, {t_1}^{-1}) \\
&= (e_2 \to {e_2}^{-1}) \cdot ( (e_1 \to {e_1}^{-1}) \cdot ({e_1}, e_2, {t_1}) ). \\
\end{split}
\]
We remark that these only describe the effects on $t_1$ and other twist parameters adjacent to these edges may change.
\begin{proof}
By (\ref{eq:t1_from_x1_x2_x3_x5}), we have 
\[
t_1 = -1 +\frac{e_2(1-e_1^2)}{e_2-e_1e_3}[x_5:x_3:x_1:x_2]
\]
If we replace $e_4$ by ${e_4}^{-1}$, then $x_4$ is replaced by $y_4$ where $y_4$ is the other fixed points.
Therefore $t_1$ does not change.
If we replace $e_2$ by $e_2^{-1}$, then $x_2$ is replaced by $y_2$ and 
\[
y_2 = \frac{e_2^2 e_3 x_3(x_2-x_1) +e_3 x_1 (x_3-x_2) + e_2 e_1 x_2 (x_1-x_3)}{e_2^2 e_3 (x_2-x_1) +e_3 (x_3-x_2) + e_2 e_1 (x_1-x_3)}
\]
by (\ref{eq:other_fixed_point}).
Thus the new twist parameter $t_1'$ is given by 
\[
\begin{split}
t_1' =& -1 + \frac{e_2^{-1}(1-e_1^2)}{e_2^{-1}-e_1e_3}[x_5:x_3:x_1:y_2] \\
=& -1 + \frac{(1-e_1^2)}{1-e_1e_2e_3}\frac{x_1-x_3}{x_1-x_5} \\
& \times \frac{e_2^2 e_3 (x_3-x_5)(x_2-x_1) +e_3 (x_1-x_5) (x_3-x_2) + e_2 e_1 (x_2-x_5)(x_1-x_3)}{e_3 (x_1-x_3) (x_3-x_2) + e_2 e_1 (x_2-x_3) (x_1-x_3)} \\
=& -1 + \frac{(1-e_1^2)}{1-e_1e_2e_3} \frac{(e_3-e_2^2 e_3) (x_1-x_5) (x_3-x_2) + (e_2 e_1-e_2^2 e_3 ) (x_2-x_5)(x_1-x_3)}{(e_2 e_1 -e_3)(x_2-x_3) (x_1-x_5)} \\
=& -1 + \frac{(1-e_1^2)}{1-e_1e_2e_3} \left( -\frac{e_3(1-e_2^2)}{e_1 e_2 -e_3} + \frac{e_2 (e_1-e_2 e_3 )}{e_1 e_2 -e_3}[x_5:x_3:x_1:x_2] \right) \\
=& \frac{e_2 e_3 -e_1}{e_1 e_2 -e_3} \frac{e_1 e_3 -e_2}{1-e_1e_2e_3} \left( -1 +  \frac{e_2(1-e_1^2)}{e_2-e_1 e_3}[x_5:x_3:x_1:x_2] \right) \\
=& \frac{e_2 e_3 -e_1}{e_1 e_2 -e_3} \frac{e_1 e_3 -e_2}{1-e_1e_2e_3} t_1. \\
\end{split}
\]
By a similar calculation for (\ref{eq:t1_from_x1_x2_x4_x5}), we obtain the behavior under ${e_3} \to e_3^{-1}$ and $e_5 \to {e_5}^{-1}$.
The action of $e_1 \to {e_1}^{-1}$ follows from a similar calculation or the definition of the twist parameter.
\end{proof}

\subsection{One holed torus}
Next  we describe the behaviour of the twist parameter corresponding to the interior pants curve on a one-holed torus.

\begin{figure}
\input{one_holed_with_gr_a.pstex_t}
\hspace{20pt}
\input{one_holed_with_gr_b.pstex_t}
\caption{}
\label{fig:two_dual_graphs}
\end{figure}

\begin{theorem}
Take a dual graph and parameters as in the left of Figure \ref{fig:two_dual_graphs}.
Then $(\mathbb{Z}/2\mathbb{Z})^2$ acts as 
\begin{equation}
\label{eq:action_of_Z2_for_one_holed_a}
\begin{split}
(e_1 \to {e_1}^{-1}) \cdot (e_1, t_1) &= ({e_1}^{-1}, {t_1}^{-1}), \\
(e_2 \to {e_2}^{-1}) \cdot (e_2, t_1) &= ({e_2}^{-1}, \left( \frac{e_2 -{e_1}^2}{{e_1}^2 e_2 -1} \right)^{2} t_1). \\
\end{split}
\end{equation}
Take a dual graph and parameters as in the right of Figure \ref{fig:two_dual_graphs}.
Then $(\mathbb{Z}/2\mathbb{Z})^2$ acts as \begin{equation}
\label{eq:action_of_Z2_for_one_holed_b}
\begin{split}
(e_1 \to {e_1}^{-1}) \cdot (e_1, t_1) &= ({e_1}^{-1}, {t_1}^{-1}), \\
(e_2 \to {e_2}^{-1}) \cdot (e_2, t_1) &= ({e_2}^{-1}, t_1). \\
\end{split}
\end{equation}
\end{theorem}
\begin{proof}
Considering a covering space as in Figure \ref{fig:twist_for_one_holed}, we reduce the calculation to the case of four-holed sphere. 
In the case of the left of Figure \ref{fig:two_dual_graphs}, by (\ref{eq:action_of_Z2_for_four_holed}), 
we have $(e_1 \to {e_1}^{-1}) \cdot (e_1,t_1) = ({e_1}^{-1}, {t_1}^{-1})$ and 
\[
\begin{split}
(e_2 \to {e_2}^{-1}) \cdot t_1 &=   
\left( \frac{e_2{e_1}^{-1}-e_1}{1-e_1e_2{e_1}^{-1}}\frac{e_1{e_1}^{-1}-e_2}{e_1e_2-{e_1}^{-1}} \right) 
\left( \frac{e_1e_2-e_1}{1-e_1e_1e_2}\frac{e_1e_1-e_2}{e_1e_2-e_1} \right) t_1 \\
&= \frac{e_2{e_1}^{-1}-e_1}{e_1e_2-{e_1}^{-1}} \frac{e_1e_1-e_2}{1-e_1e_1e_2} t_1 = \left( \frac{e_2-{e_1}^2}{{e_1}^2e_2-1} \right)^2 t_1. \\
\end{split}
\]
In the case of the right of Figure \ref{fig:two_dual_graphs}, again by (\ref{eq:action_of_Z2_for_four_holed}), 
we have $(e_2 \to {e_2}^{-1}) \cdot (e_2,t_1) = ({e_2}^{-1}, t_1)$ and 
\[
\begin{split}
(e_1 \to {e_1}^{-1}) \cdot t_1 &=   
\left( \left( \frac{{e_1}^{-1}e_2-e_1}{1-e_1{e_1}^{-1}e_2}\frac{e_1e_2-{e_1}^{-1}}{e_1{e_1}^{-1}-e_2} \right)
\left( \frac{e_2e_1-e_1}{1-e_1e_2e_1}\frac{e_1e_2-e_1}{e_1e_1-e_2} \right) t_1 \right)^{-1} \\
&= {t_1}^{-1}.
\end{split}
\]
\end{proof}

\subsection{Example: closed surface of genus two}
It is easy to check that the trace functions of the four-holed sphere given in \S \ref{subsec:four_holed_sphere} 
are invariant under the action of (\ref{eq:action_of_Z2_for_four_holed}).
It is also easy to check that the trace functions of the one-holed torus given in \S \ref{subsec:once_punctured} 
are invariant under the action of (\ref{eq:action_of_Z2_for_one_holed_a}).

We apply the transformation formulae to the surface of genus two.
Applying the formula in (\ref{eq:action_of_Z2_for_four_holed}), the effect on $t_1$ under the action of $e_2 \to {e_2}^{-1}$ is given by
\[
(e_2 \to {e_2}^{-1}) \cdot t_1  = 
\left( \frac{e_2e_3-e_1}{1-e_1e_2e_3} \cdot \frac{e_1e_3-e_2}{e_1e_2-e_3} \right)
\left( \frac{{e_3}^{-1}{e_2}^{-1}-e_1}{1-e_1{e_3}^{-1}{e_2}^{-1}} \cdot \frac{e_1{e_3}^{-1}-{e_2}^{-1}}{e_1{e_2}^{-1}-{e_3}^{-1}} \right) t_1 
= t_1.
\]
By similar calculations, the action of $(\mathbb{Z}/2\mathbb{Z})^3$ on the parameter space is given by 
\[
\begin{split}
(e_1,e_2,e_3,t_1,t_2,t_3) &\mapsto ({e_1}^{-1},e_2,e_3,{t_1}^{-1},t_2,t_3), \\
(e_1,e_2,e_3,t_1,t_2,t_3) &\mapsto (e_1,{e_2}^{-1},e_3,t_1,{t_2}^{-1},t_3), \\
(e_1,e_2,e_3,t_1,t_2,t_3) &\mapsto (e_1,e_2,{e_3}^{-1},t_1,t_2,{t_3}^{-1}). \\
\end{split}
\]
Actually the trace functions in \S \ref{sec:genus2} are invariant under the action.

%% file: sec_transformation.tex
%
%


Our parametrization is defined for a given pants decomposition with a dual graph.
In this section we give transformation formulae under changes of dual graphs and pants decompositions.
We define five types of moves among pants decompositions with dual graphs.
We will show that any two pants decompositions with dual graphs are related by these moves.
Then we give transformation formulae for these moves.

\subsection{Elementary move and dual graph}
We define five types of moves between pants decompositions with dual graphs:
\begin{description}
\item[(I) Reverse orientation] Reverse the orientation of an edge of the dual graph.
\item[(II) Dehn twist] Change the dual graph by a (left or right) Dehn twist along a pants curve.
\item[(III) Vertex move] For a vertex of the dual graph, change the edges adjacent to the vertex by their right half-twists as Figure \ref{fig:vertex_move}.
\item[(IV) Graph automorphism] Composition with an automorphism $\varphi: G \to G$ preserving the orientations of the edges. 
(Change $(C,(g,G))$ to $(C,(g \circ \varphi, G))$.)
\item[(V) Elementary move] On a subsurface homeomorphic to a one-holed torus or a four-holed sphere, 
we define the moves as indicated in Figure \ref{fig:elementary_moves_with_dual_graphs}.
In the one-holed torus case, this is obtained by a clockwise rotation of angle $\pi/2$ (Figure \ref{fig:elementary_moves_with_dual_graphs_for_one_holed_in_a_cover}).
\end{description}
Except elementary moves, these moves do not change the pants decomposition.

\begin{figure}
\input{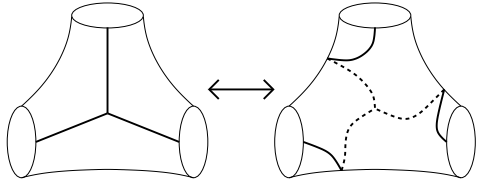}
\caption{(III) Vertex move}
\label{fig:vertex_move}
\end{figure}

\begin{lemma}
\label{lem:moves_do_not_change_pants_decomposition}
Let $C$ be a pants decomposition of $S$. Let $g_1: G_1 \to S$ and $g_2 : G_2 \to S$ be dual graphs dual to $C$.
Then $(g_1,G_1)$ and $(g_2,G_2)$ are related by a sequence of type (I)--(IV) moves and their inverses.
\end{lemma}
\begin{proof}
Consider the restriction of the images $g_1(G_1)$ and $g_2(G_2)$ on a pair of pants $P \subset S \setminus N(C)$.
These are tripods on $P$ whose endpoints of the legs are on different boundary components.
We remark that there are two such tripods on $P$ up to isotopy preserving boundary of $P$ setwise (not pointwise), 
and they are related by a type (III) move each other.
Hence we assume that $g_1(G_1)$ and $g_2(G_2)$ coincide except on annular neighborhoods of the pants curves. 
Performing Dehn twists near the pants curves, we assume that $g_1(G_1)$ and $g_2(G_2)$ coincide.
Then there exists a homeomorphism $\varphi : G_1 \to G_2$ satisfying $g_1 \circ \varphi = g_2$ which may not preserve the orientations of the edges.
If $\varphi$ does not preserve the orientations, perform type (I) moves.
\end{proof}

\begin{proposition}
\label{prop:related_by_five_types_of_moves}
Any two pants decompositions with dual graphs are related by a sequence of type (I)--(V) moves and their inverses.
\end{proposition}
\begin{proof}
Hatcher and Thurston \cite{hatcher-thurston} showed that any two pants decompositions are related by \emph{elementary moves}.
Two pants decompositions $C$ and $C'$ are related by an elementary move if $C'$ is obtained from $C$ by replacing a pants curve $c \in C$ 
by a curve $c'$ which intersects $c$ minimally and does not intersect other pants curves.
(Locally, elementary moves are indicated in Figure \ref{fig:elementary_moves_with_dual_graphs} by forgetting dual graphs.)
There are infinitely many ways to perform elementary moves with respect to a given interior pants curve, 
but there is a unique elementary move for a pants decomposition with a dual graph.

Let $(C_1,G_1)$ and $(C_2,G_2)$ be two pants decompositions with dual graphs.
By Hatcher and Thurston's result, $C_1$ and $C_2$ are related by a sequence of elementary moves of pants decompositions.
Therefore we only have to show that an elementary move can be realized by a sequence of (I)--(V) moves.
We assume that $C_1$ and $C_2$ are related by one elementary move which exchange a pants curve $c \in C_1$ to $c' \in C_2$.
When $c$ (and $c'$) is on a one-holed torus, then we can change the dual graph $G_1$ by type (II) and (III) moves so that $G_1$ does not intersect $c'$. 
Then a type (V) move exchanges $c$ to $c'$. 
When $c$ (and $c'$) is on a four-holed sphere, then we can also change the dual graph $G_1$ by type (II) and (III) so that a type (V) move exchanges $c$ to $c'$.
Therefore we assume that $C_1$ is equal to $C_2$.
Applying Lemma \ref{lem:moves_do_not_change_pants_decomposition}, we conclude that $(C_1,G_1)$ is related by a sequence of (I)--(V) moves to $(C_2,G_2)$.
\end{proof}

\begin{figure}
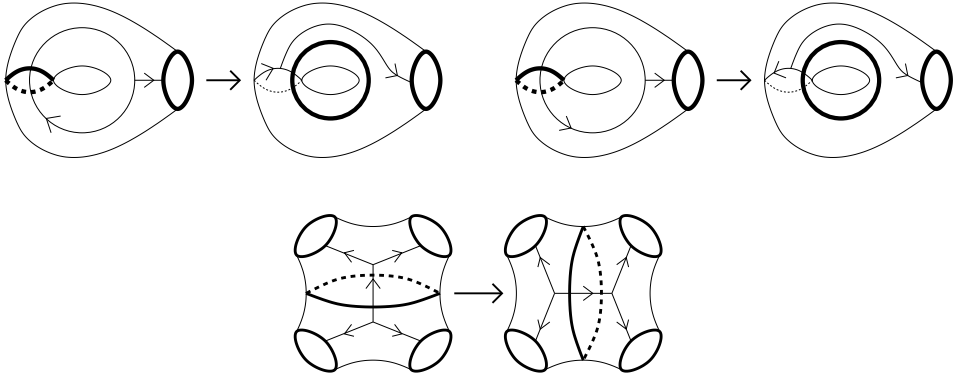

\input{el_mv_I_with_gr_1.pstex_t}
\hspace{20pt}
\input{el_mv_I_with_gr_2.pstex_t}

\vspace{20pt}
\input{el_mv_II_with_gr.pstex_t}
\caption{Elementary moves of pants decompositions with dual graphs.}
\label{fig:elementary_moves_with_dual_graphs}
\end{figure}

\begin{figure}
\input{el_mv_I_with_gr_1_cv.pstex_t}
\caption{}
\label{fig:elementary_moves_with_dual_graphs_for_one_holed_in_a_cover}
\end{figure}

\subsection{Transformation formula}
By Proposition \ref{prop:related_by_five_types_of_moves}, we only have to describe the transformation formulae for type (I)--(V) moves.
Since type (IV) move results in a permutation of parameters, we give formulae for type (I), (II), (III) and (V) moves.

\begin{proposition}[Type (I) move]
Let $e_1, \dots, e_5$ and $t_1$ be the eigenvalue and twist parameters as in the left of Figure \ref{fig:oriengation_change_of_dual_graphs_of_four_holed}.
When we replace the orientation of the edge corresponding to $(e_1, t_1)$, 
then the new parameters $(e_1', t_1')$ are given by
\begin{equation}
\label{eq:reverse_the_orientation_of_an_edge}
(e_1', t_1') = (e_1^{-1}, \frac{e_1 e_2 -e_3}{e_1 e_3 - e_2} \frac{e_1 e_5 - e_4}{e_1 e_4 - e_5} t_1^{-1}). 
\end{equation}
\end{proposition}
\begin{remark}
If one of the edges corresponding to $e_2, \dots, e_5$ is inward oriented, say $e_i$, replace the parameter in the formula by its inverse ${e_i}^{-1}$.
We also apply this rule to other formulae in this section.
When we reverse the orientation of the edge on a one-holed torus, apply the formula by taking a cover as in Figure \ref{fig:twist_for_one_holed}.
\end{remark}
\begin{remark}
There are two choices of eigenvalues in the formula, but it is natural to take ${e'_1} = {e_1}^{-1}$ 
to be consistent with the notation of \S \ref{subsec:definition_of_twist_parameters}.
\end{remark}
\begin{proof}
Let $x_1, \dots, x_5$ be the fixed points as in the Lemma \ref{lem:twist_parameter_from_fixed_points}.
By (\ref{eq:t1_from_x1_x2_x3_x5}), we have
\[
\begin{split}
t_1' &= -1 +\frac{e_4(1-e_1^{-2})}{e_4-{e_1}^{-1}e_5}[x_3:x_5:x_1:x_4] \\
&= -1 +\frac{e_4(1-e_1^{-2})}{e_4-{e_1}^{-1}e_5}(1-[x_3:x_4:x_1:x_5]).
\end{split}
\]
By (\ref{eq:t1_from_x1_x3_x4_x5}), we have
\[
\begin{split}
t_1' &= -1 +\frac{e_4(1-e_1^{-2})}{e_4-{e_1}^{-1}e_5}\left(1- \frac{e_4-e_1 e_5}{e_4(1-{e_1}^2)}\left(1-\frac{e_1(e_1e_2-e_3)}{t_1(e_1e_3-e_2)}\right) \right) \\
&= \frac{e_1 e_2 - e_3}{e_1 e_3 - e_2} \frac{e_1 e_5 - e_4}{e_1 e_4 - e_5} \frac{1}{t_1}. \\
\end{split}
\]
\end{proof}

\begin{proposition}[Type (II) move]
\label{prop:right_Dehn_twist}
Let $(C_1,G_1)$ be a pants decomposition with a dual graph.
Apply a right Dehn twist at an edge of $G_1$ whose eigenvalue and twist parameter are $(e_i,t_i)$.
Then the new parameters are given by 
\begin{equation}
\label{eq:right_Dehn_twist}
(e_i',t_i') = (e_i, {e_i}^2 t_i).
\end{equation}
\end{proposition}
This follows from the definition of the twist parameter, but we prove as a corollary of Proposition \ref{prop:half_twists}.

The type (III) move consists of three `half twists' of the dual graph (Figure \ref{fig:change_of_dual_graphs_of_four_holed}).
\begin{figure}
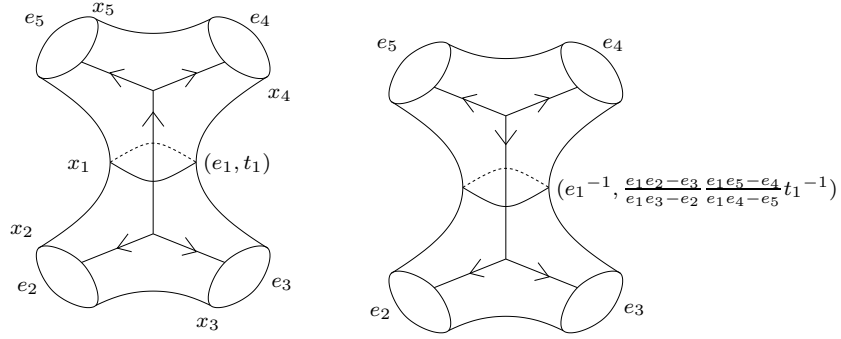

\input{4_holed_with_gr_1.pstex_t}
\hspace{30pt}
\input{4_holed_with_gr_2.pstex_t}
\caption{(I) Reverse orientation.}
\label{fig:oriengation_change_of_dual_graphs_of_four_holed}
\end{figure}

\begin{figure}
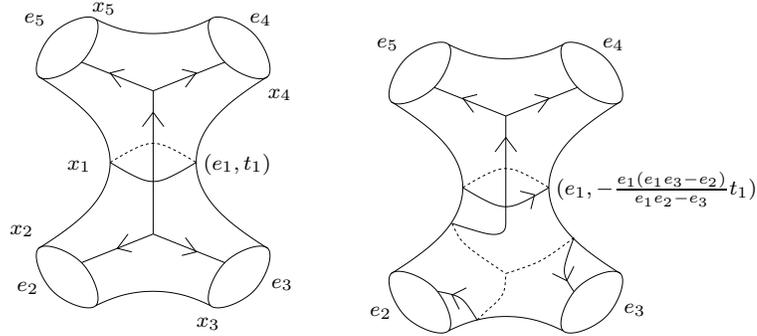

\input{4_holed_with_gr_1.pstex_t}
\hspace{30pt}
\input{4_holed_with_gr_3.pstex_t}
\caption{(III) Vertex move.}
\label{fig:change_of_dual_graphs_of_four_holed}
\end{figure}

\begin{proposition}[A part of Type (III) move]
\label{prop:half_twists}
Let $e_1, \dots, e_5$ and $t_1$ be the eigenvalue and twist parameters as in the left of Figure \ref{fig:change_of_dual_graphs_of_four_holed}.
After doing a half twist, the new parameters $(e'_1,t'_1)$ are given by
\begin{equation}
\label{eq:right_half_twist}
(e'_1,t'_1) = (e_1, -\frac{e_1(e_1e_3-e_2)}{e_1e_2-e_3} t_1).
\end{equation}
\end{proposition}
The transformation formula for a type (III) move is obtained by applying (\ref{eq:right_half_twist}) to three edges adjacent to the vertex.
\begin{proof}[Proof of Propositions \ref{prop:right_Dehn_twist} and \ref{prop:half_twists}.]
By (\ref{eq:t1_from_x1_x3_x4_x5}) and (\ref{eq:t1_from_x1_x2_x4_x5}), we have
\[
{t_1'}^{-1} = 
-\frac{e_1e_2-e_3}{e_1(e_1e_3-e_2)} \left( -1 +\frac{e_4(1-e_1^2)}{e_4-e_1e_5}[x_2:x_4:x_1:x_5] \right)=-\frac{e_1e_2-e_3}{e_1(e_1e_3-e_2)} t_1^{-1}.
\]
This shows that (\ref{eq:right_half_twist}) holds.
Applying (\ref{eq:right_half_twist}) twice, we obtain a proof of (\ref{eq:right_Dehn_twist}): 
\[
t'_1 = \left( -\frac{e_1(e_1e_2-e_3)}{e_1e_3-e_2} \right) \left( -\frac{e_1(e_1e_3-e_2)}{e_1e_2-e_3} \right) t_1 = {e_1}^2 t_1.
\]
\end{proof}

The trace function of the new pants curve in type (V) move has already been given by (\ref{eq:four_holed_tr_1}) for a four-holed sphere case and 
(\ref{eq:one_holed_tr}) for a one-holed torus case.
Thus the new eigenvalue parameter is one of $\frac{\tr\pm\sqrt{\tr^2-4}}{2}$ where $\tr$ is the trace function of the new pants curve.
Differently from type (I)--(IV) moves, there is no natural choice of the eigenvalue parameter for the new pants curve in type (V) move.
But if we fix one of them, the twist parameter is uniquely determined.
Sometimes we fix the new eigenvalue parameter by $\frac{\tr-\sqrt{\tr^2-4}}{2}$ where we take the branch of the square root satisfying $0 \leq \arg (\sqrt{\tr^2-4}) < \pi$
for convenience.

\begin{proposition}[Type (V) move for four-holed sphere]
\label{prop:type_V_move_four_holed}
Define eigenvalue parameters $e_1, \dots, e_{13}$ and $t_1, \dots, t_5$ as in the left of Figure \ref{fig:peripheral}.
Let $e'_1$ be the eigenvalue parameter corresponding to the new pants curve and $t'_1, \dots, t'_5$ the twist parameters as in the right of Figure \ref{fig:peripheral}.
Then $e'_1$ is one of the solutions of 
\begin{equation}
\label{eq:equation_for_eigenvalues}
x^2 - \tr(\rho(\gamma_3 \gamma_4)) x + 1 = 0
\end{equation}
where $\tr(\rho(\gamma_3 \gamma_4))$ is given by (\ref{eq:four_holed_tr_1}).
The twist parameters $t'_1, \dots, t'_5$ are given by 
\begin{subequations}
\begin{equation}
\label{eq:T1_for_four_holed_a}
\begin{split}
t'_1 &= \frac{1}{{e'_1}+{e'_1}^{-1}} \frac{e'_1 e_5 e_2}{(e_5e_2-e'_1)(e'_1e_2-e_5)} \frac{e'_1 e_3 e_4}{(e_3e_4-e'_1)(e'_1e_3-e_4)} \\
&\biggl( ({e'_1}-{e'_1}^{-1}) \bigl(  e'_1 \tr(\rho(\gamma_3 \gamma_5) - {e'_1}^{-1} \tr(\rho(\gamma_2 \gamma_4) \bigr)\\
&-({e'_1}+{e'_1}^{-1}) \bigl( (e_2+e_2^{-1})(e_3+e_3^{-1})+(e_4+e_4^{-1})(e_5+e_5^{-1}) \bigr) \\
&+2 \bigl( (e_3+e_3^{-1})(e_5+e_5^{-1})+(e_2+e_2^{-1})(e_4+e_4^{-1}) \bigr) \biggr) \\
\end{split}
\end{equation}
or
\begin{equation}
\label{eq:T1_for_four_holed_b}
{t'_1}  = \frac{
({e_4}{e'_1}-{e_3})( {e_3}{e_5}(-({e_1}{e_3}-{e_2}){t_1}+{e_1}({e_1}{e_2-{e_3}}) ){e'_1} \\
+{e_1}({e_1}{e_3}-{e_2}){t_1}-({e_1}{e_2}-{e_3}))
}
{
({e_3}{e'_1}-{e_4})(( {e_1}({e_1}{e_3}-{e_2}){t_1}-({e_1}{e_2}-{e_3})){e'_1}
+{e_3}{e_5}(-({e_1}{e_3}-{e_2}){t_1}+{e_1}({e_1}{e_2}-{e_3}))
},
\end{equation}
\end{subequations}
and
\begin{align}
\label{eq:T2_for_four_holed}
t'_2 &=  
\frac{({e_1}{e_2}-{e_3})({e_1}{e_3}-{e_2})({t_1}+1)}
{({e_2}{e_3}-{e_1})({e_1}{e_3}-{e_2}){t_1}+(1-{e_1}{e_2}{e_3})({e_1}{e_2}-{e_3})} 
\frac{{e_2}{e'_1}-{e_5}}{{e_2}{e_5}-{e'_1}} {t_2}, \\
\label{eq:T3_for_four_holed}
t'_3 &=
\frac{({e_2}{e_3}-{e_1})(({e_1}{e_3}-{e_2})({e_1}{e_4}-{e_5}){t_1}+({e_1}{e_2}-{e_3})({e_1}{e_5}-{e_4}))}
{({e_1}{e_3}-{e_2})(({e_2}{e_3}-{e_1})({e_1}{e_4}-{e_5}){t_1}+(1-{e_1}{e_2}{e_3})({e_1}{e_5}-{e_4}))} {t_3}, \\
\label{eq:T4_for_four_holed}
t'_4 &=
\frac{({e_1}{e_3}-{e_2})({e_1}{e_4}-{e_5}){t_1}+({e_1}{e_2}-{e_3})({e_1}{e_5}-{e_4})}
{({e_1}{e_3}-{e_2})({e_4}{e_5}-{e_1}){t_1}+({e_1}{e_2}-{e_3})(1-{e_1}{e_4}{e_5})} 
\frac{{e_3}{e'_1}-{e_4}}{1-{e_3}{e_4}{e'_1}} {t_4}, \\
\label{eq:T5_for_four_holed}
t'_5 &=
\frac{({e_1}{e_5}-{e_4})({e_1}{e_4}{e_5}-1)({t_1}+1)}
{({e_1}-{e_4}{e_5})({e_1}{e_4}-{e_5}){t_1}+({e_1}{e_4}{e_5}-1)({e_1}{e_5}-{e_4})} {t_5}. 
\end{align}
where $\tr(\rho(\gamma_3 \gamma_5))$ and $\tr(\rho(\gamma_2 \gamma_4))$ are given in (\ref{eq:four_holed_tr_3}) and (\ref{eq:four_holed_tr_2}) respectively.
Any other parameters does not change under the move.
\end{proposition}
\begin{proof}
We use the notation of \S \ref{subsec:four_holed_sphere}. 
In particular, $\gamma_1, \dots, \gamma_5$ are as in Figure \ref{fig:twist_parameter} and $x_1=\infty, x_2 = 1, x_3 = 0, x_4,x_5 \in \mathbb{C}P^1$.
We let $x_6, \dots, x_{13}$ be the fixed point parameters corresponding to the edges with the eigenvalue parameters $e_6, \dots, e_{13}$ respectively.  
In the following computations, we used the computer algebra system Maxima.

The first statement is trivial. 
(\ref{eq:T1_for_four_holed_a}) follows from (\ref{eq:twist_from_equation_0_4}) by replacing the variables.
To compute the twist parameters, we use the formulae (\ref{eq:t1_from_x1_x2_x3_x5})--(\ref{eq:t1_from_x1_x3_x4_x5}).
To apply the formulae, we need to know the fixed points of the matrix $\rho(\gamma_3 \gamma_4)$.
We let $\rho(\gamma_3\gamma_4) = \begin{pmatrix} a & b \\ c & d \end{pmatrix}$.
From (\ref{eq:gamma_1_2_3_for_four_holed}) and (\ref{eq:gamma_4_for_four_holed}), we have 
\begin{equation}
\label{eq:fixed_point_tmp_1}
\begin{split}
a =& 
\biggl ({e_1}({e_1}{e_3}-{e_2})({e_4}{e_5}({e_1}{e_4}-{e_5})+({e_1}{e_5}-{e_4})){t_1} \\
&+({e_1}{e_2}-{e_3})({e_1}{e_4}{e_5}({e_1}{e_5}-{e_4})-({e_1}{e_5}-{e_4}))  \biggr)
/ \biggl(
{({e_1}^2-1){e_3}({e_1}{e_3}-{e_2}){e_4}{e_5}{t_1}} \biggr), \\
b =& 
\biggl(
{e_1}
(({e_1}{e_3}-{e_2})({e_1}{e_4}-{e_5}){t_1}+({e_1}{e_2}-{e_3})({e_1}{e_5}-{e_4}))(({e_1}{e_3}-{e_2})({e_4}{e_5}-{e_1}){t_1} \\
&+(1-{e_1}{e_4}{e_5})({e_1}{e_2}-{e_3})) 
\biggr) / \biggl(
{e_2}{e_3}{e_4}{e_5}({e_1}^2-1)^{2}({e_1}{e_3}-{e_2}){t_1}
\biggr). \\
\end{split}
\end{equation}
Let $x'_1$ (resp. $y'_1$) be the fixed point of $\rho(\gamma_3 \gamma_4)$ corresponding to $e'_1$ (resp. ${e'_1}^{-1}$).
By (\ref{eq:two_fixed_points_case}), we have 
\[
a = \frac{{e'_1}{x'_1}-{e'_1}^{-1}{y'_1}}{x'_1-y'_1}, \quad 
b = -\frac{({e'_1}-{e'_1}^{-1}) x'_1 y'_1}{x'_1-y'_1}, \quad 
c = \frac{{e'_1}-{e'_1}^{-1}}{x'_1-y'_1}, \quad 
c = \frac{-{e'_1}{y'_1}+{e'_1}^{-1}{x'_1}}{x'_1-y'_1}.
\]
So we have
\begin{equation}
\label{eq:fixed_point_tmp_2}
\frac{b}{{e'_1}-a} = \frac{-({e'_1}-{e'_1}^{-1}) x'_1 y'_1}{{e'_1}(x'_1-y'_1)-({e'_1}{x'_1}-{e'_1}^{-1}{y'_1})} = x'_1.
\end{equation}
From (\ref{eq:fixed_point_tmp_1}) and (\ref{eq:fixed_point_tmp_2}), we can compute $x'_1$.
Applying the formulae (\ref{eq:x4})--(\ref{eq:x5}), we can also compute the fixed point parameters $x_6, \dots, x_{13}$.
Apply (\ref{eq:t1_from_x1_x2_x4_x5}) to compute $t'_1$, $t'_3$, $t'_4$ and $t'_5$, we have
\begin{align}
\label{eq:tmp_for_T1_for_four_holed}
{t'_1}^{-1} &= -1 + \frac{e_3 (1-{e'_1}^2)}{e_3-{e'_1} e_4} [x_5:x_3:{x'_1}:x_4]  \\
\label{eq:tmp_for_T3_for_four_holed}
{t'_3}^{-1} &= -1 +\frac{e_8(1-{e_3}^2)}{e_8-e_3e_9}[x_4:x_8:x_3:x_9] \\
\label{eq:tmp_for_T4_for_four_holed}
{t'_4}^{-1} &= -1 +\frac{e_{10}(1-{e_4}^2)}{e_{10}-e_4e_{11}}[x'_1:x_{10}:x_4:x_{11}] \\
\label{eq:tmp_for_T5_for_four_holed}
{t'_5}^{-1} &= -1 +\frac{e_{12}(1-{e_5}^2)}{e_{12}-e_5e_{13}}[x_2:x_{12}:x_5:x_{13}]
\end{align}
Substitute $x_2, \dots, x_{13}$ and $x'_1$ into (\ref{eq:tmp_for_T1_for_four_holed})--(\ref{eq:tmp_for_T5_for_four_holed}), 
we obtain (\ref{eq:T1_for_four_holed_b}) and (\ref{eq:T3_for_four_holed})--(\ref{eq:T5_for_four_holed}).
We apply (\ref{eq:t1_from_x1_x3_x4_x5}) to compute $t'_2$, then we have
\begin{align}
\label{eq:tmp_for_T2_for_four_holed}
{t'_2}^{-1} = -\frac{e_2e_5-e'_1}{e_2(e_2e'_1-e_5)} \left( -1 +\frac{e_6(1-{e_2}^2)}{e_6-e_2e_7}[x_5:x_6:x_2:x_7] \right). 
\end{align}
Substitute $x_2$, $x_6$, $x_2$, $x_7$ into (\ref{eq:tmp_for_T2_for_four_holed}), we obtain (\ref{eq:T2_for_four_holed}).
\end{proof}

\begin{figure}
\input{periph_4_before.pstex_t}
\hspace{20pt}
\input{periph_4_after.pstex_t}
\caption{}
\label{fig:peripheral}
\end{figure}

\begin{proposition}[Type (V) move for one-holed trous]
\label{prop:type_V_move_one_holed}
Define eigenvalue parameters $e_1,\dots, e_4$ and $t_1, t_2$ as in the left of Figure \ref{fig:peripheral_one_holed}.
Let $e'_1$ be the eigenvalue parameter corresponding to the new pants curve and $t'_1, t'_2$ the twist parameters as in the right of Figure \ref{fig:peripheral_one_holed}.
Then $e'_1$ is one of the solutions of 
\begin{equation}
\label{eq:equation_for_eigenvalues_one_holed}
x^2 - \tr(\rho(\beta_1)) x + 1 = 0
\end{equation}
where $\tr(\rho(\beta_1))$ is given by (\ref{eq:one_holed_tr}).
The twist parameters $t'_1$ and $t'_2$ are given by 
\begin{align}
\label{eq:T1_for_one_holed_torus}
t'_1 &= \frac{- e_2}{({e'_1}^{2}-{e_2})^2} \left( (e_1+{e_1}^{-1})  - e'_1 \tr(\rho({\alpha_1}^{-1} \beta_1)) \right)^2,  \\
\label{eq:T2_for_one_holed_torus}
{t'_2} &=
\frac
{{e_2}({e_2}-{e_1}^{2})(\sqrt{-{e_2}{t_1}}{e'_1}+{t_1})}
{(1-{e_1}^{2}{e_2})\sqrt{-{e_2}{t_1}}{e'_1}+({e_2}-{e_1}^{2}){e_2}{t_1}}
{t_2},
\end{align}
where 
\[
\tr(\rho({\alpha_1}^{-1} \beta_1))=
-\frac{({e_1}^{2}-{e_2}){t_1}+{e_1}^{2}(1-{e_1}^{2}{e_2})}{{e_1}({e_1}^2-1)\sqrt{-{e_2}{t_1}}}.
\]
\end{proposition}
\begin{proof}
We use the notation of \S \ref{subsec:once_punctured}. 
In particular, $\alpha_1$ and $\beta_1$ are as in Figure \ref{fig:once_punctured} and $x_1=\infty, x_2 = 0, x_3 = 1, x_4,x_5 \in \mathbb{C}P^1$.
In the following computations, we also used Maxima. 

The first statement is trivial. 
(\ref{eq:T1_for_one_holed_torus}) follows from (\ref{eq:twist_from_equation_1_1}) by replacing the variables.
To compute the twist parameters, we use the formulae (\ref{eq:t1_from_x1_x2_x3_x5})--(\ref{eq:t1_from_x1_x3_x4_x5}) 
in a cover of the one-holed torus as in Figure \ref{fig:peripheral_one_holed_cover}.
We let $x_6 $ $x_7$ be the fixed point parameters corresponding to the edges with the eigenvalue parameters $e_3$ and $e_4$ respectively.  
Applying the formulae (\ref{eq:x4})--(\ref{eq:x5}), we can compute the fixed point parameters $x_6$ and $x_7$.
To apply (\ref{eq:t1_from_x1_x2_x3_x5})--(\ref{eq:t1_from_x1_x3_x4_x5}), we need to know the fixed points of the matrix $\rho(\beta_1)^{-1}$.
Let $x'_1$ be the fixed point of $\rho(\beta_1)^{-1}$ corresponding to $e'_1$.
We let $\rho(\beta_1)^{-1} = \begin{pmatrix} a & b \\ c & d \end{pmatrix}$.
Similarly as (\ref{eq:fixed_point_tmp_2}), we can compute $x'_1$ from (\ref{eq:rho_beta_1_for_one_holed}): 
\[
x'_1 = \frac{a-{e_1}^{-1}}{c}
= \frac{{e'_1}e_2-\sqrt{-e_2t_1}}{{e'_1}{e_2}}
= 1-\frac{\sqrt{-{e_2}{t_1}}}{{e'_1}{e_2}}
= 1+\frac{{t_1}}{\sqrt{-{e_2}\*{t_1}}\*{e'_1}}.
\]
Applying (\ref{eq:t1_from_x1_x3_x4_x5}) to compute ${t'_2}^{-1}$, we have 
\begin{equation}
\label{eq:tmp_for_T2_for_one_holed_torus}
{t'_2}^{-1} = -\frac{e_2e_1-{e_1}^{-1}}{e_2(e_2{e_1}^{-1}-e_1))} \left( -1 + \frac{e_3(1-{e_2}^2)}{e_3-e_2 e_4} [x'_1:x_6:x_2:x_7] \right).
\end{equation}
Substitute $x'_1$, $x_2=0$, $x_6$ and $x_7$ into (\ref{eq:tmp_for_T2_for_one_holed_torus}), we obtain (\ref{eq:T2_for_one_holed_torus}).
\end{proof}

\begin{figure}
\input{periph_one_holed.pstex_t}
\caption{}
\label{fig:peripheral_one_holed}
\end{figure}

\begin{figure}
\input{periph_one_holed_cv.pstex_t}
\caption{}
\label{fig:peripheral_one_holed_cover}
\end{figure}

%% file: sec_developing.tex
%
%


In this section, we recall the notion of a developing map. 
Then we reinterpret the eigenvalue and twist parameters in terms of the developing map.
In \S \ref{subsec:developing_map}, we define ideal triangulations of surfaces in a way which is suitable for our purpose, and define the developing map.
Then we study the developing map of a pair of pants in \S \ref{subsec:developoing_map_of_pants}.
These are completely determined by the complex parameters associated to the edges of the ideal triangulation.
In \S \ref{subsec:gluing_developing_maps}, we will discuss how two developing maps of a pair of pants are glued along a boundary curve depending on the twist parameter.

\subsection{Ideal triangulation and developing map}
\label{subsec:developing_map}
Let $S = S_{g,b}$ be an oriented surface of genus $g$ with $b$ boundary components.
We assume that $2g-2+b > 0$, then $S$ admits a hyperbolic metric with geodesic boundaries. 
We fix a hyperbolic metric on $S$.
An \emph{ideal triangle} is a triangle isometric to the convex hull in $\mathbb{H}^2$ of three distinct points on the ideal boundary 
(in other words, geodesic triangle with vertices at infinity). 
In this paper we say that a union of ideal triangles $\Delta_1 \cup \dots \cup \Delta_n$ on $S$ is an \emph{ideal triangulation} 
if the interiors of $\Delta_i$'s are disjoint. 
For example, the complement of a maximal geodesic lamination on $S$, after completion of each component by path metric, is an ideal triangulation. 
Because the volume of $S_{g,b}$ is equal to $2 \pi (2g-2+b)$ and the volume of an ideal triangle is $\pi$, there are at most $2(2g-2+b)$ ideal triangles.
In this paper, we mainly concern ideal triangulations obtained by spinning vertices around simple closed geodesics, 
e.g. Figure \ref{fig:pantscover} indicates an ideal triangulation of a pair of pants.

By the hyperbolic metric on $S$, we can regard the universal covering $\widetilde{S}$ of $S$ as a subset of the hyperbolic plane $\mathbb{H}^2$. 
Let $p : \widetilde{S} \to S$ be the covering map.
Each lift of $\Delta_i$ is an ideal triangle in $\mathbb{H}^2$, which is uniquely characterized by three points of $\partial \overline{\mathbb{H}^2}$. 
Fix a lift $\widetilde{\Delta}_i$ of $\Delta_i$ in $\mathbb{H}^2$.
Then other lifts are obtained by deck transformations of $\widetilde{\Delta}_i$.
For a representation $\rho : \pi_1(S) \to \PSLC$, 
we say that $D : \widetilde{S} \to \mathbb{H}^3$ is a \emph{developing map} of $\rho$ if it satisfies 
\begin{itemize}
\item $D(\gamma \widetilde{\Delta}_i) = \rho(\gamma) D(\widetilde{\Delta}_i)$ for any $\gamma \in \pi_1(S)$, and
\item $D|_{\widetilde{\Delta}_i} :\widetilde{\Delta}_i \to D(\widetilde{\Delta}_i)$ is an isometry into $\mathbb{H}^3$ for any $i$.
\end{itemize}
Especially boundary geodesics of $\widetilde{\Delta}_i$ are mapped by $D$ to geodesics in $\mathbb{H}^3$.
We do not assume that $D$ is continuous away from the lifts of $\Delta_1 \cup \dots \cup \Delta_n$.
A typical discontinuous points appear at simple closed geodesics spiralled by other geodesics (e.g. boundary curves of Figure \ref{fig:pantscover}).
In fact $D$ need not to be defined away from the lifts of $\Delta_1 \cup \dots \cup \Delta_n$ in our arguments.
Since a matrix which sends $D(\widetilde{\Delta})$ to $D(\gamma \widetilde{\Delta})$ is uniquely determined by Lemma \ref{lem:thrice_transitive}, 
$\rho$ is determined by $D$.
We say that $\rho$ is the \emph{holonomy representation} of $D$.

Consider a common geodesic of two ideal triangles $\Delta_1$ and $\Delta_2$ in $S$.
For a developing map $D$, we assign a complex parameter to such geodesic by using cross ratio.
Let $\widetilde{\Delta}_1$ and $\widetilde{\Delta}_2$ be a pair of lifts adjacent in $\mathbb{H}^2$.
Let $(v_0,v_1,v_2)$ be the ideal vertices of $\widetilde{\Delta}_1$ and $(v_3,v_1,v_0)$ be the ideal vertices of $\widetilde{\Delta}_2$ .
We assume that $(v_0, v_1, v_2)$ and $(v_3,v_1,v_0)$ are in the clockwise order on $\partial \overline{\mathbb{H}^2}$, see Figure \ref{fig:ori_conv}.
Since $D$ maps the ideal triangle spanned by $(v_0,v_1,v_2)$ isometrically into $\mathbb{H}^3$, 
$D$ naturally maps $v_i$ to a point $x_i \in \mathbb{C}P^1$.
In this situation, we simply denote $x_i = D(v_i)$ by abuse of notation.
We also define $x_3 \in \mathbb{C}P^1$ by $x_3 = D(v_3)$.
Now we define the complex parameter at the edge $(v_0, v_1)$ by the cross ratio defined in \S \ref{subsec:basic_facts_on_PSL}:
\[
[x_0:x_1:x_2:x_3]=\frac{x_3-x_0}{x_3-x_1}\frac{x_2-x_1}{x_2-x_0} \in (\mathbb{C} -\{0,1\}).
\]
This does not depends on the choice of the pair of lifts $\widetilde{\Delta}_1$ and $\widetilde{\Delta}_2$ 
since the cross ratio is invariant under the action of $\PSLC$.

\begin{figure}
\input{ori_conv.pstex_t}
\caption{}
\label{fig:ori_conv}
\end{figure}

If any two ideal triangles $\widetilde{\Delta}$ and $\widetilde{\Delta'}$ in $\widetilde{S}$ are related by a sequence of ideal triangles 
$\widetilde{\Delta} = \widetilde{\Delta}_{1}, \widetilde{\Delta}_{2}, \dots, \widetilde{\Delta}_{k} = \widetilde{\Delta'}$ in $\widetilde{S}$ such that 
$\widetilde{\Delta}_{i}$ and $\widetilde{\Delta}_{i+1}$ share a geodesic, 
we can reconstruct the developing map $D$ from these complex parameters.
(For example, if there exists at least one boundary component, 
we obtain such an ideal triangulation from a usual ideal triangulation of a surface of genus $g$ with $b$ punctures by spinning ideal vertices.
In this case, the number of complex parameters is equal to $6g-6+3b$.)
In fact, fix a lift $\widetilde{\Delta}$ of an ideal triangle on $\widetilde{S}$, and put an ideal triangle $D(\widetilde{\Delta})$ in $\mathbb{H}^3$ arbitrarily.
Then we can develop adjacent ideal triangles in $\mathbb{H}^3$ according to the complex parameters of the common edges.
Inductively we obtain a developing map $D:\widetilde{S} \to \mathbb{H}^3$.
If we replace the first ideal triangle by $D'(\widetilde{\Delta})$, we also obtain a developing map $D':\widetilde{S} \to \mathbb{H}^3$. 
Since there exists $A \in \PSLC$ such that  $D'(\widetilde{\Delta}) = A D(\widetilde{\Delta})$ by Lemma \ref{lem:thrice_transitive}, 
we conclude that  $D' = A D$ on each ideal triangle of $\widetilde{S}$ by construction. 
The holonomy representation of $D'$ is obtained from the holonomy representation of $D$ by conjugation by $A$.
Therefore we obtain a map from the parameter space $(\mathbb{C} \setminus \{0,1\} )^{N}$ to $X_{\PSLC}(S)$, 
where $N$ is the number of the edges each of which belongs to two ideal triangles of the ideal triangulation. 
We call these coordinates the \emph{exponential shear-bend coordinates} of the ideal triangulation (see \cite{bonahon96}).
By construction, this map is a rational map.
If the complex parameters are in a subfield $\mathbb{F} \subset \mathbb{C}$, then we obtain a $\mathrm{PGL}(2,\mathbb{F})$-representation.

\subsection{Developing map of a pair of pants}
\label{subsec:developoing_map_of_pants}
Let $P$ be a pair of pants. We fix a hyperbolic metric with geodesic boundary on $P$.
We put a tripod $G$ on $P$ whose endpoints of the legs are on different boundary components.
We denote the boundary components of $P$ by $c_1$, $c_2$, $c_3$ so that they are in counterclockwise order viewing from the tripod.
We give an orientation on each $c_i$.
Connect two boundary components by an arc along $G$ and spiral the arc toward the direction of $c_i$, 
we obtain a geodesic spiralling the boundary components of $P$ (Figure \ref{fig:pantscover}).
These geodesics give an ideal triangulation of $P$ by two ideal triangles. 
We will construct a developing map of a representation $\rho:\pi_1(P) \to \SLC$ satisfying the conditions of Proposition \ref{prop:pants_rep}, 
and compute the complex parameters of the three geodesics of the ideal triangulation.

Take a base point $*$ on the trivalent vertex of $G$ and let $\gamma_1$, $\gamma_2$, $\gamma_3$ be the elements of $\pi_1(P, *)$ as indicated in Figure \ref{fig:pantscover}. 
We regard the universal cover $\widetilde{P}$ of $P$ as a subset of $\mathbb{H}^2$.
We fix a lift of $G$ to $\mathbb{H}^2$ and denote it by $\widetilde{G}$. 
Then there exists a unique lift $\widetilde{c}_i$ of $c_i$ such that $\widetilde{G}$ touches $\widetilde{c}_i$. 
We denote the terminal endpoint of $\widetilde{c}_i$ by $v_i$ and the ideal triangle $(v_1,v_2,v_3)$ by $\widetilde{\Delta}_0$.
We take an ideal triangle $\widetilde{\Delta}_1=(v_1,v_3,\gamma_1 v_2)$, which shares the geodesic $(v_1,v_3)$ with $\widetilde{\Delta}_0$ (see Figure \ref{fig:pantscover}).
We denote the projection of $\widetilde{\Delta}_i$ to $P$ by $\Delta_i$ for $i=0,1$.
Then $\Delta_0$ and $\Delta_1$ are two ideal triangles of the ideal triangulation of $P$.
There exists a one-to-one correspondence between the lifts of $\Delta_0$ and $\pi_1(P,*)$ by deck transformations.

We let $e_i$ $(i=1,2,3)$ be one of the eigenvalues of $\rho(\gamma_i)$, and $x_i$ (resp. $y_i$) be the fixed point corresponding to $e_i$ (${e_i}^{-1}$), in other words, 
$x_i$ is the repelling fixed point ($y_i$ is the attractive fixed point) if $|e_i|>1$ (same as \S \ref{subsec:_matrices_of_pants_rep}).
Since we assume that $\rho$ satisfies the conditions of Proposition \ref{prop:pants_rep}, $x_i$'s are distinct.  
We construct a developing map $D: \widetilde{P} \to \mathbb{H}^3$ as follows.
Any lift of $\Delta_0$ (resp. $\Delta_1$) has the form $\gamma (v_1,v_2,v_3)$ (resp. $\gamma (v_1,v_3, \gamma_1 v_2)$) for some $\gamma \in \pi_1(P,*)$.
Thus we define a developing map $D : \widetilde{P} \to \mathbb{H}^3$ by 
$D( \gamma(v_1, v_2,v_3) ) = \rho(\gamma) (x_1,x_2,x_3)$ and $D(\gamma (v_1,v_3, \gamma_1 v_2)) = \rho(\gamma) (x_1,x_3, \rho(\gamma_1) x_2) $.
By construction, this gives a developing map of $\rho$.
Then the complex parameter of each edge has a simple form:
\begin{proposition}
\label{prop:shear_bend_parameter_of_a_pair_of_pants}
The complex parameter of the edge spiralling around both $c_i$ and $c_{i+1}$ is equal to $\displaystyle\frac{e_{i+2}}{e_i e_{i+1}}$.
\end{proposition}
\begin{proof}
Since the cross ratio is invariant under the action of M\"obius transformation, we assume that $(x_1,x_2,x_3) = (0, \infty, 1)$.
By Proposition \ref{prop:pants_rep}, we have 
\[
\rho(\gamma_1) = \begin{pmatrix} e_1^{-1} & 0 \\ e_1^{-1} - e_2^{-1} {e_3} & e_1 \end{pmatrix}.
\]
Thus we have $\rho(\gamma_1) \cdot x_2  =\rho(\gamma_1) \cdot \infty = \frac{e_1^{-1}}{{e_1}^{-1}-{e_2}^{-1} e_3}$.
The complex parameter of the geodesic spiralling both $c_1$ and $c_3$ is 
\[
[x_3:x_1:\rho(\gamma_1)\cdot x_2:x_2] = 
[1: 0: \rho(\gamma_1) \cdot \infty: \infty] = \frac{\rho(\gamma_1) \cdot \infty}{\rho(\gamma_1) \cdot \infty-1}= \frac{e_2}{e_3 e_1}
\]
by definition. By a similar calculation or symmetry, we can show for other edges. 
\end{proof}


Conversely we can construct a $\PSLC$-representation from these complex parameters.
Since 
\[
\frac{e_2}{e_1 e_3} \frac{e_3}{e_1 e_2} = \frac{1}{{e_1}^2}
\]
we can recover the eigenvalue parameters $e_1, e_2, e_3$ up to sign.
For each $e_i$, there exists two choices of signs, but they depend on each other so that $\frac{e_{i+2}}{e_{i} e_{i+1}}$ is invariant.
This means that they are well-defined up to the action of $H^1(P;\mathbb{Z}/2\mathbb{Z})$ (see \S \ref{subsec:pants_rep_of_PSL2C}).
Therefore we can determine a unique $\PSLC$-representation up to conjugation from $\{\frac{e_{i+2}}{e_i e_{i+1}} \}_{i=1,2,3}$. 
(But there is no canonical choice of a lift to an $\SLC$-representation.)

\begin{figure}
\input{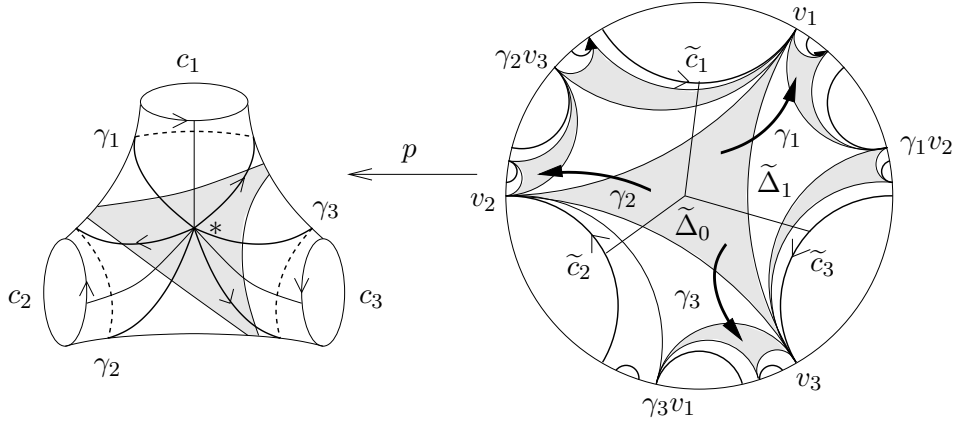}
\caption{
Three arrows emanating from $\widetilde{\Delta}_0$ mean deck transformations corresponding to $\gamma_1,\gamma_2,\gamma_3 \in \pi_1(P,*)$ respectively. 
}
\label{fig:pantscover}
\end{figure}

\subsection{Gluing developing maps of a pair of pants}
\label{subsec:gluing_developing_maps}
In this subsection, we will describe how two developing maps of a pair of pants are glued along their common boundary curve. 

We let $S$ be a four-holed sphere and decompose it into two pairs of pants $P$ and $P'$.
We denote the common interior pants curve by $c_1$ and boundary pants curves by $c_2,\dots, c_5$ as in Figure \ref{fig:glue_two_pairs}.
Take a dual graph $G$ of the pants decomposition and give an orientation on $c_i$ to the right from the dual edge.
We define $\gamma_1, \dots, \gamma_5 \in \pi_1(S)$ as in Figure \ref{fig:twist_parameter}.
We fix a lift $\widetilde{G}$ to $\mathbb{H}^2$, then there exists a unique lift $\widetilde{c}_i$ of $c_i$ such that $\widetilde{G}$ touches $\widetilde{c}_i$. 
We denote the terminal endpoint of $\widetilde{c}_i$ by $v_i$.
The projections of the ideal triangles $(v_1,v_2,v_3)$ and $(v_1, v_3, \gamma_1 v_2)$ form an ideal triangulation of $P$ 
and the projections of $(v_1,v_4,v_5)$ and $(v_1, \gamma_1 v_4, v_5)$ form an ideal triangulation of $P'$ (see Figure \ref{fig:glue_two_pants_covers}). 

Let $\rho$ (resp. $\rho'$) be a $\PSLC$-representation of $\pi_1(P)$ (resp. $\pi_1(P')$) satisfying the conditions of Proposition \ref{prop:pants_rep}.
We will glue the developing maps of $\rho$ and $\rho'$ along $\widetilde{c}_1$.
To glue them, we assume that $\rho(\gamma_1)$ and $\rho'(\gamma_1)$ are conjugate.
First we construct developing maps of $\rho$ and $\rho'$ as in \S \ref{subsec:developoing_map_of_pants}.
We let $e_i$ $(i=1,2,3)$ be one of the eigenvalues of $\rho(\gamma_i)$, and $x_i$ (resp. $y_i$) be the fixed point corresponding to $e_i$ (resp. ${e_i}^{-1}$).
By sending $(v_1,v_2,v_3)$ and $(v_1,v_3,\gamma_1 v_2)$ to $(x_1,x_2,x_3)$ and $(x_1, x_3, \rho(\gamma_1) x_2)$, and developing them by the action of $\rho(\pi_1(P))$, 
we obtain a developing map $D: \widetilde{P} \to \mathbb{H}^3$ of $\rho$.
For $\rho'$, we let $e_i$ $(i=1,4,5)$ be one of the eigenvalues of $\rho'(\gamma_i)$, and $x'_i$ (resp. $y'_i$) be the fixed point corresponding to $e_i$ (${e_i}^{-1}$).
To glue the representation $\rho'$ to $\rho$ along the curve $c_1$, we conjugate $\rho'$ to satisfy $(x'_1,y'_1) = (x_1,y_1)$. 
There are many ways to conjugate $\rho'$ to be $(x'_1,y'_1) = (x_1,y_1)$, but we fix one of them.
We denote the fixed points after the conjugation corresponding to $x'_4$ and $x'_5$ by $x_4$ and $x_5$ respectively.
By sending $(v_1,v_4,v_5)$ and $(v_1,v_4,\gamma_1 v_5)$ to $(x_1,x_4,x_5)$ and $(x_1, x_4, \rho(\gamma_1) x_5)$ and developing them, 
$D$ extends to a map $D: \widetilde{P} \cup_{\widetilde{c}_1} \widetilde{P'} \to \mathbb{H}^3$.
Consider two $\rho( \langle \gamma_1\rangle )$-invariant families of ideal triangles 
\[
\begin{split}
\{\rho(\gamma_1)^{i}(x_1,x_2,\rho(\gamma_1) x_2) \}_i &=  \{ (x_1, \rho(\gamma_1)^{i} x_2, \rho(\gamma_1)^{i+1}x_2) \}_i, \\
\{\rho(\gamma_1)^{i}(x_1,x_5,\rho(\gamma_1) x_5) \}_i &=  \{ (x_1, \rho(\gamma_1)^{i} x_5, \rho(\gamma_1)^{i+1}x_5) \}_i
\end{split}
\]
as indicated in Figure \ref{fig:glue_two_pants_covers}.
These two families of ideal triangles are related by a matrix $M(\sqrt{-t_1};x_1,y_1)$ for some $t_1 \in \mathbb{C}^*$ in $\rho( \langle \gamma_1\rangle )$-equivariant way.
Since $M(\sqrt{-t_1};x_1,y_1)$ sends $x_2$ to $x_5$, $t_1$ is the twist parameter defined in Section \ref{sec:twist_paramter}.

Since there exist infinite number of lifts of $c_1$, we have to carry out this process to all other lifts.
After that, we obtain a developing map for $S = P \cup P'$.

\begin{figure}
\input{glue_two_pants.pstex_t}
\caption{}
\label{fig:glue_two_pairs}
\end{figure}

\begin{figure}
\input{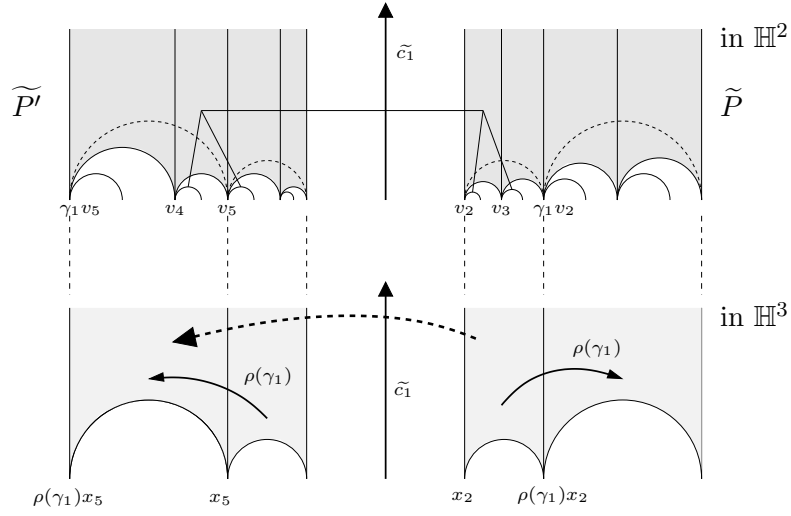}
\caption{The gluing of the developing maps.} 
\label{fig:glue_two_pants_covers}
\end{figure}

\begin{remark}
Bonahon gave a parametrization of $\mathrm{PSL}(2,\mathbb{C})$ representations of the fundamental group of a surface 
by using the \emph{shear-bend cocycle} of a maximal geodesic lamination $\lambda$ in Section 10 of \cite{bonahon96}.
Roughly speaking, our parametrization is a special case of the shear-bend cocycle
for the maximal geodesic lamination obtained from a pants decomposition by adjoining geodesics spiralling to the pants curves.
\end{remark}

%% file: sec_psl2r.tex
%
%


In this section, we restrict our attention to $\PSLR$-representations, in particular, Fuchsian representations.
Then we give a geometric interpretation of our twist parameters and compare them with the usual Fenchel-Nielsen twist parameters.

\subsection{$\PSLR$ and $\PGLR$}
Let $\PGLR = \mathrm{GL}(2; \mathbb{R}) / \mathbb{R}^*$ where $\mathbb{R}^*$ acts on $\mathrm{GL}(2; \mathbb{R})$ by scalar multiplication.
The action of $\PGLR$ on $\mathbb{C}P^1$ (resp. $\mathbb{H}^3$) preserves $\mathbb{R}P^1$ (resp. $\mathbb{H}^2 = \{(x,y,t) \mid  y=0, t > 0 \} \subset \mathbb{H}^3 $ ). 
$\PGLR$ consists of two connected components:
\[
\PGLR = \PSLR \cup \PSLR \cdot \begin{pmatrix} -1 & 0 \\ 0 & 1 \end{pmatrix}.
\]
One component consists of orientation preserving isometries of $\mathbb{H}^2$ and 
the other component consists of orientation reversing isometries.

Let $S=S_{g,b}$ be a surface.
We fix a pants decomposition $C$ and a dual graph $G$.
We restrict our parametrization to 
\begin{equation}
\label{eq:part_of_PGLR_rep}
\{ (e_i,t_i) \mid e_i, t_i  \in  \mathbb{R}^{*}, \quad (e_1,\dots, e_{3g-3+2b}) \in E(S,C) \}.
\end{equation}
In the explicit construction of representations in Section \ref{sec:matrix_generators}, if we place the first triple of the fixed points on $\mathbb{R}P^1$,
we obtain representations whose entries are in real numbers.
Thus we obtain $\PGLR$-representations from the parameter space.

To reduce the $\PGLR$-representation to a $\PSLR$-representation, it is convenient to use the dual graph $G$.
Let $\rho$ be a $\PGLR$-representation obtained from the parameter space (\ref{eq:part_of_PGLR_rep}).
Take a maximal tree $T$ in $G$, we fix a system of generators $\{\alpha_1, \dots \alpha_g, \beta_1, \dots, \beta_g. \delta_1, \dots, \delta_b\}$ 
of $\pi_1(S)$ as in \S \ref{subsec:dual_graph}.
By the construction of \S \ref{sec:matrix_generators}, we see that $\rho(\alpha_i)$ and $\rho(\delta_i)$ are in $\PSLR$.
So we only have to check that $\rho(\beta_i) \in \PSLR$ or not.
Recall that $\rho(\beta_i)$ is the matrix which sends a triple of points on $\mathbb{R}P^1$ to another triple of points on $\mathbb{R}P^1$.
Thus $\rho(\beta_i)$ preserves the order of the triples on $\mathbb{R}P^1$ if and only if $\rho(\beta_i) \in \PSLR$.
By assigning $+ 1$ or $-1$ depending on whether $\rho(\beta_i)$ preserves the order of the triples or not, we obtain a homomorphism $\pi_1(G) \to \mathbb{Z}/2\mathbb{Z}$.  
Since $\Hom(\pi_1(G), \mathbb{Z}/2\mathbb{Z}) \cong H^1(G; \mathbb{Z}/2\mathbb{Z})$, we can obtain an obstruction in $H^1(G; \mathbb{Z}/2\mathbb{Z})$, 
which vanishes if and only if $\rho$ reduces to a $\PSLR$-representation.  

There exist more $\PSLR$-representations other than coming from the parameter space (\ref{eq:part_of_PGLR_rep}).
In fact, the matrix corresponding to a pants curve is elliptic, then the eigenvalue parameter $e_i$ is not a real number but satisfies $|e_i| = 1$. 
It is known that: 
\begin{theorem}[Theorem 4.3 of \cite{goldman}]
\label{thm:goldman_SL2R_or_SU2}
Let $e_1$, $e_2$, $e_3$ be the eigenvalue parameters of a pair of pants.
Assume that $e_i$ is real or $|e_i| = 1$. 
The representation corresponding to $(e_1,e_2,e_3)$ is conjugate to a $\SLR$-representation if and only if either 
one of $|e_i+{e_i}^{-1}|$ is $\geq 2$ or $\kappa \geq 2$ where  
\[
\kappa = (e_1+{e_1}^{-1})^2 + (e_2+{e_2}^{-1})^2+(e_3+{e_3}^{-1})^2 - (e_1+{e_1}^{-1})(e_2+{e_2}^{-1})(e_3+{e_3}^{-1}) - 2.
\]
Otherwise the representation is conjugate to a $\mathrm{SU}(2)$-representation.
\end{theorem}

\subsection{Teichm\"uller space}
The set of all marked hyperbolic structures on $S$ is called the Teich\"uller space.
(If $S$ has boundary, we consider hyperbolic structures with geodesic boundary.)
Each marked hyperbolic structure gives rise to a discrete faithful representation of $\pi_1(S)$ into $\PSLR$.
So there exists a subset in our parametrization corresponding to the Teichm\"uller space.
We will show that the Teichm\"uller space is parametrized by the following subset of $E(S,C) \times \mathbb{C}^{3g-3+b}$:
\begin{equation}
\label{eq:representative_for_Teichmuller_space}
\{ (e_1, \dots, e_{3g-3+2b}, t_1, \dots, t_{3g-3+b}) \in \mathbb{R}^{6g-6+3b} \mid  e_i < -1, \  t_i > 0  \}.
\end{equation}
Actually we can construct an explicit developing map corresponding to each point of this subset.
There are other subsets producing Fuchsian representations but they are obtained from (\ref{eq:representative_for_Teichmuller_space}) 
by the action of $(\mathbb{Z}/2\mathbb{Z})^{3g-2+2b}$ and $H_1(G,\partial G; \mathbb{Z}/2\mathbb{Z})$.

First, we will show that the restriction of the $\PSLR$-representation obtained from (\ref{eq:representative_for_Teichmuller_space}) 
to each pair of pants is discrete and faithful: 
\begin{lemma}
\label{lem:discreteness_of_pslr_rep}
Let $P$ be a pair of pants and $e_1$, $e_2$ and $e_3$ are eigenvalue parameters on the boundary curve.
Assume $e_i \in \mathbb{R} \setminus \{0, \pm 1\}$.
If $e_1 e_2 e_3 < 0$, then the representation is discrete and faithful.
\end{lemma}
\begin{proof}
By the action of $(\mathbb{Z}/2\mathbb{Z})^3$ (see (\ref{eq:action_of_Z/2Z^3})) and $H^1(P;\mathbb{Z}/2\mathbb{Z})$ (see (\ref{eq:action_of_H^1(P,Z/2Z)})), 
we assume that $e_1, e_2 < -1$ and $|e_3| >1$.
Since $e_1 e_2 e_3 < 0$, we have $e_1, e_2, e_3 < -1$.
Let $x_i$ (resp. $y_i$) be the fixed point corresponding to $e_i$ (resp. ${e_i}^{-1}$) as in Proposition \ref{prop:pants_rep}.
Assume $x_1 = \infty$, $y_1=0$ and $x_2 =1$. We will show that 
\begin{equation}
\label{eq:inequality_for_a_fundamental_domain}
0 < x_2=1 < y_2 < x_3 < y_3 < {e_1}^2,
\end{equation}
see Figure \ref{fig:fundamental_domain}.
Then we can take a fundamental domain bounded by four geodesics
\[
(\infty, 0), \quad (1,y_2), \quad (x_3, y_3), \quad ({e_1}^2, {e_1}^2 y_1)
\]
and other four geodesics perpendicular to them as indicated in Figure \ref{fig:fundamental_domain}.

Considering the complex parameter at the geodesic $(x_3, \infty)$, we have 
\[
\frac{e_2}{e_3 e_1} = [x_3 : \infty : {e_1}^2 : 1] = \frac{1-x_3}{{e_1}^2-x_3},
\]
and 
\[
x_3 = \frac{-e_1(e_1e_2-e_3)}{e_1e_3-e_2}. 
\]
By (\ref{eq:other_fixed_point}), we have
\[
\begin{split}
y_2 &= \frac{-{e_2}^2 e_3 x_3 + e_3(x_3-1)+e_2 e_1}{ -{e_2}^2 e_3 +e_2 e_1} = \frac{(e_1e_2-e_3)(1-e_1e_2e_3)}{(e_2e_3-e_1)(e_1e_3-e_2)} \\
y_3 &= \frac{{e_3}^2 e_1 (x_3-1) + e_1 -e_3 e_2 x_3}{ e_1 - e_3 e_2} = \frac{-e_1(1-e_1e_2 e_3)}{e_2e_3-e_1} \\
\end{split}
\]
From these, we have
\[
\begin{split}
y_2 - 1   &= \frac{-e_3({e_1}^2-1)({e_2}^2-1)}{(e_2e_3-e_1)(e_1e_3-e_2)} > 0, \quad 
x_3 - y_2 = \frac{(e_1e_2-e_3)({e_1}^2-1)}{(e_2e_3-e_1)(e_1e_3-e_2)} > 0, \\
y_3 - x_3 &= \frac{e_1e_2({e_1}^2-1)({e_3}^2-1)}{(e_2e_3-e_1)(e_1e_3-e_2)} > 0, \quad
{e_1}^2 - y_3 = \frac{-e_1({e_2}^2-1)}{e_2e_3-e_1} > 0.
\end{split}
\]
Thus (\ref{eq:inequality_for_a_fundamental_domain}) holds.

\begin{figure}
\input{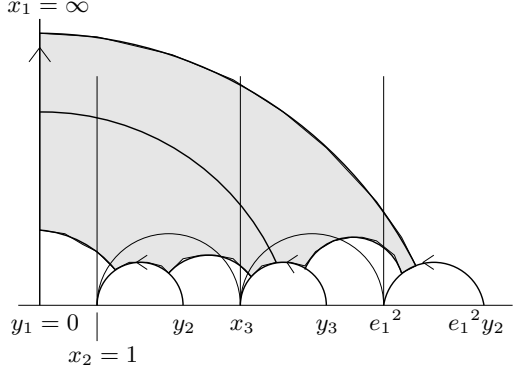}
\caption{A fundamental domain for a pair of pants.}
\label{fig:fundamental_domain}
\end{figure}

\end{proof}

So if $e_i < -1$ for all $i$, the restriction of the representation to any pair of pants is discrete and faithful.
If further $t_i >  0$ for all $i$, the $\PSLR$-representation we obtained is discrete faithful by the Maskit combination theorem.

In Propositions \ref{prop:type_V_move_four_holed} and \ref{prop:type_V_move_one_holed}, if we take $\frac{\tr-\sqrt{\tr^2-4}}{2} < -1$ as the new eigenvalue parameter, 
type (V) moves preserve the parameter space (\ref{eq:representative_for_Teichmuller_space}).

\subsection{Geometric interpretation of twist parameters}
\label{subsec:geom_iterpretation}
When we restrict the parameter space to (\ref{eq:representative_for_Teichmuller_space}), the twist parameters are described in terms of hyperbolic geometry.

We consider a pants decomposition and a dual graph of a four-holed sphere $S_{0,4}$ as in Figure \ref{fig:twist_parameter}.
We let $e_1, \dots ,e_5$ be the eigenvalue parameters and the twist parameter $t_1$.
Let $\rho$ be a representation obtained from the parameter space  (\ref{eq:representative_for_Teichmuller_space}).
We define the elements $\gamma_1, \dots, \gamma_5 \in \pi_1(S_{0,4})$ as in Figure \ref{fig:twist_parameter} and 
denote the fixed point of $\rho(\gamma_i)$ corresponding to $e_i$ (resp. ${e_i}^{-1}$) by $x_i$ (rep. $y_i$).
If we assume that $x_i$'s are in $\mathbb{R}P^1$ (hence also $y_i \in \mathbb{R}P^1$)
and $x_1 = \infty$ and $y_1 = 0$, then $\rho(\gamma_1) = \pm \begin{pmatrix} \sqrt{-t_1} & 0 \\ 0 & \sqrt{-t_1}^{-1} \end{pmatrix}$.
Consider the nearest point projections of $x_2$ and $x_5$ to the geodesic $(x_1,y_1)$, 
the hyperbolic (signed) distance between them is $\log(t_1)$ (see Figure \ref{fig:twist_in_teich}).
Thus if we denote the geodesic representative of the interior pants curve by $c_1$ and the boundary pants curve corresponding to $e_i$ by $c_i$,  
our twist parameter is the exponential of the hyperbolic distance between the geodesic perpendicular to $c_1$ and spiralling to $c_2$ and 
the geodesic perpendicular to $c_1$ and spiralling to $c_5$.

By taking a covering, we also have a geometric interpretation of the twist parameters of a one-holed torus.

\begin{figure}
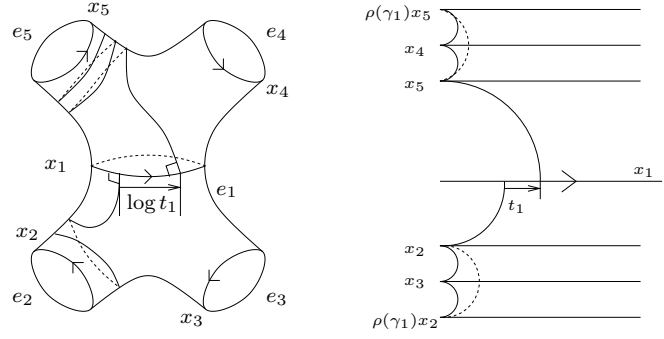

\input{twist_fuchsian.pstex_t}
\hspace{30pt}
\input{twist_fuchsian_dev.pstex_t}
\caption{Two pairs of pants are glued with the twist of the hyperbolic length $\log(t_i)$.}
\label{fig:twist_in_teich}
\end{figure}

\subsection{Fenchel-Nielsen twist parameters}
\label{subsec:FN-coordinates}
We compare our twist parameters with the Fenchel-Nielsen twist parameters.

There exists various normalization of Fenchel-Nielsen twist parameters, we use the following.
We consider a four-holed sphere and use the notations of \S \ref{subsec:geom_iterpretation}.
The Fenchel-Nielsen twist parameter is defined by the hyperbolic (signed) distance between the geodesic perpendicular to both $c_1$ and $c_2$ 
and the geodesic perpendicular to both $c_1$ and $c_5$ (see Figure \ref{fig:FN_twist.pstex_t}). 
We denote the signed distance by $\tau_1$ and let $t^{FN}_1 = \exp(\tau_1)$.
Then we have the following:

\begin{figure}
\input{FN_twist.pstex_t}
\caption{}
\label{fig:FN_twist.pstex_t}
\end{figure}

\begin{proposition}
Consider the pants decomposition and the dual graph given in Figure \ref{fig:twist_parameter}.
Let $e_1, \dots, e_5$ be the eigenvalue parameters and $t_1$ be the twist parameter.
Then we have
\begin{equation}
\label{eq:our_twist__to_FN_twist}
t^{FN}_1 = \sqrt{\frac{(e_1e_3-e_2)(e_2e_3-e_1)(e_1e_4-e_5)(e_4e_5-e_1)}{(e_1e_2-e_3)(1-e_1e_2e_3)(e_1e_5-e_4)(1-e_1e_4e_5)}}t_1.
\end{equation}
\end{proposition}
\begin{proof}
We take $\gamma_1, \dots, \gamma_5$ as in Figure \ref{fig:twist_parameter}.
Let $x_i$ (resp. $y_i$) be the eigenvalue of $\rho(\gamma_i)$ corresponding to $e_i$ (resp. ${e_i}^{-1}$).
We assume that $x_1 = \infty$, $y_1 = 0$ and $x_2=1$.
Let $m_2$ be the point on the geodesic $(x_1, y_1)$ nearest to $(x_2,y_2)$ and 
$m_5$ be the point on the geodesic $(x_1, y_1)$ nearest to $(x_5,y_5)$ (see Figure \ref{fig:comp_with_FN.pstex_t}).
Since $m_2$ (resp. $m_5$) is the midpoint between the nearest point projections of $x_2$ and $y_2$ to $(x_1,y_1)$ (resp. $x_5$ and $y_5$ to $(x_1,y_1)$), we have
\[
\frac{y_2}{m_2} = \frac{m_2}{1},  \quad  \frac{t_1}{m_5} = \frac{m_5}{-y_5}.
\]
Thus we have 
\begin{equation}
\label{eq:compare_with_FN_twist_1}
(t^{FN}_1)^2 = \left( \frac{m_5}{m_2} \right)^2 = \frac{- t_1 y_5}{y_2}. 
\end{equation}
So we have to compute $y_2$ and $y_5$.
By the calculations of Lemma \ref{lem:discreteness_of_pslr_rep}, we have 
\[
x_3 = \frac{e_1(e_1 e_2 -e_3)}{e_2 - e_1 e_3}, \quad y_2 = \frac{({e_1}{e_2}-{e_3})(1-{e_1}{e_2}{e_3})}{({e_2}{e_3}-{e_1})({e_1}{e_3}-{e_2})}.
\]
By (\ref{eq:x4}) and (\ref{eq:x5}), we have
\[
\begin{split}
x_4 &= \frac{e_1(-(e_1 e_3-e_2)(e_1e_4-e_5)t_1+e_3(e_1e_5-e_4))(1-x_3)-e_1^2e_2(e_1e_5-e_4)+e_2(e_1e_5-e_4)x_3}{-e_1^2e_2(e_1e_5-e_4)+e_2(e_1e_5-e_4)} \\
&= \frac{{e_1}({e_1}{e_4}-{e_5})}{{e_1}{e_5}-{e_4}} {t_1}, \\
x_5 &=  \frac{(-(e_2-e_1e_3)t_1+e_1e_3)(1-x_3)-e_1^2e_2+e_2x_3}{-e_1^2e_2+e_2} = -{t_1}.
\end{split}
\]
By (\ref{eq:other_fixed_point}), we have
\[
y_5 = \frac{{e_5}^2 {e_1}^{-1} (x_5-x_4)+{e_1}^{-1}x_4- e_5 e_4 x_5}{{e_1}^{-1}- e_5 e_4} 
= -\frac{({e_1}{e_4}-{e_5})({e_4}{e_5}-{e_1})}{({e_1}{e_5}-{e_4})(1-{e_1}{e_4}{e_5})} {t_1}.
\]
Substitute $y_2$, $x_5$ and $y_5$ into (\ref{eq:compare_with_FN_twist_1}), we have 
\[
(t^{FN}_1)^2 = \frac{x_5 y_5}{y_2} {t_1}^2 = \frac{(e_1e_3-e_2)(e_2e_3-e_1)(e_1e_4-e_5)(e_4e_5-e_1)}{(e_1e_2-e_3)(1-e_1e_2e_3)(e_1e_5-e_4)(1-e_1e_4e_5)} {t_1}^2.
\]
Since $e_i < -1$ for $i=1, \dots, 5$, we have $e_i e_j - e_k > 0$ and $1 - e_ie_je_k > 0$ for any $i,j,k = 1, \dots, 5$.
This completes the proof.
\end{proof}

\begin{figure}
\input{compare_with_FN.pstex_t}
\caption{}
\label{fig:comp_with_FN.pstex_t}
\end{figure}

\begin{remark}
By (\ref{eq:our_twist__to_FN_twist}), we can define $t^{FN}_i$ on a simply connected domain in $X_{PSL}(S,C)$ 
containing the parameter space (\ref{eq:representative_for_Teichmuller_space}), e.g. the subset consisting of quasi-Fuchsian representations.
We can observe that the action of $(\mathbb{Z}/2\mathbb{Z})^{3g-3+3b}$ on $t^{FN}_i$ is trivial. 
We can also observe that $t^{FN}_1$ is invariant under type (I) move of \S \ref{sec:change_to_shear_coodinates}, 
thus $t^{FN}_1$ is defined for an unoriented dual graph.
\end{remark}

Applying (\ref{eq:our_twist__to_FN_twist}) to the one-holed torus case of \S \ref{subsec:once_punctured}, we have
\[
t^{FN}_1 = \sqrt{\frac{(e_1{e_1}^{-1}-e_2)(e_2{e_1}^{-1}-e_1)(e_1e_1-e_2)(e_1e_2-e_1)}{(e_1e_2-{e_1}^{-1})(1-e_1e_2{e_1}^{-1})(e_1e_2-e_1)(1-e_1e_1e_2)}}t_1
= \frac{{e_1}^2-e_2}{1-{e_1}^2e_2} t_1.
\]
Thus we have:
\begin{proposition}
Consider the pants decomposition and the dual graph of a one-holed torus given in \S \ref{subsec:once_punctured} 
and let $e_1, e_2$ be the eigenvalue parameters and $t_1$ be the twist parameter.
Then we have
\begin{equation}
\label{eq:our_twist__to_FN_twist_one_holed}
t^{FN}_1 = \frac{{e_1}^2-e_2}{1-{e_1}^2e_2} t_1.
\end{equation}
\end{proposition}

We end this subsection computing some trace function using $t^{FN}_1$.
First we consider the four-holed sphere case studied in \S \ref{subsec:four_holed_sphere}.
Rewrite (\ref{eq:four_holed_tr_1}) by $t^{FN}_1$, we have
\[
\begin{split}
\tr(\rho(\gamma_3 \gamma_4)& ) = 
\frac{1}{({e_1}-{e_1}^{-1})^{2}} \biggl( -\frac{(e_2 e_3-e_1)(e_1 e_3 - e_2)}{e_1 e_2 e_3}
\cdot
\frac{(e_4 e_5 - e_1)(e_1 e_4 -e_5)}{e_1 e_4 e_5} {t_1} \\
&-\frac{(e_1 e_2 -e_3)(1-e_1 e_2 e_3)}{e_1 e_2 e_3}
\cdot
\frac{(e_1 e_5 -e_4)(1 - e_1 e_4 e_5)}{e_1 e_4 e_5} \frac{1}{t_1} \\
&+({e_1}+{e_1}^{-1})\*\bigl(({e_3}+{e_3}^{-1})({e_5}+{e_5}^{-1})+({e_2}+{e_2}^{-1})({e_4}+{e_4}^{-1}) \bigr) \\
&-2\bigl( ({e_2}+{e_2}^{-1})({e_5}+{e_5}^{-1})+({e_3}+{e_3}^{-1})({e_4}+{e_4}^{-1}) \bigr) \biggr) \\
=&
\frac{1}{({e_1}-{e_1}^{-1})^{2}} \biggl( -\frac{1}{{e_1}^2 e_2 e_3 e_4 e_5}\sqrt{(e_1 e_3 - e_2)(e_1 e_2 -e_3)(e_2 e_3 - e_1)(1 - e_1 e_2 e_3)} \cdot \\
&\sqrt{(e_1 e_4 -e_5)(e_1 e_5 -e_4)(e_4 e_5 - e_1)(1 - e_1 e_4 e_5)} \cdot ({t_1^{FN}}+(t_1^{FN})^{-1}) \\
&+({e_1}+{e_1}^{-1})\*\bigl(({e_3}+{e_3}^{-1})({e_5}+{e_5}^{-1})+({e_2}+{e_2}^{-1})({e_4}+{e_4}^{-1}) \bigr) \\
&-2\bigl( ({e_2}+{e_2}^{-1})({e_5}+{e_5}^{-1})+({e_3}+{e_3}^{-1})({e_4}+{e_4}^{-1}) \bigr) \biggr). \\
\end{split}
\]
Now we have
\[
\begin{split}
&\frac{1}{{e_1} e_2 e_3}\sqrt{(e_1 e_3 - e_2)(e_1 e_2 -e_3)(e_2 e_3-e_1)(1 - e_1 e_2 e_3)} \\
&= \sqrt{(e_1+{e_1}^{-1})^2+({e_2}+{e_2}^{-1})^2+({e_3}+{e_3}^{-1})^2-({e_1}+{e_1}^{-1})({e_2}+{e_2}^{-1})({e_3}+{e_3}^{-1})-4} .
\end{split}
\]
Let $\chi_i=e_i+{e_i}^{-1}$ and $e'_1 $ be one of the eigenvalues of $\rho(\gamma_3\gamma_4)$, then we have
\[
\begin{split}
e'_1 + {e'_1}^{-1}
=&\frac{1}{({e_1}-{e_1}^{-1})^{2}}
\biggl( - \sqrt{{\chi_1}^2+{\chi_2}^2+{\chi_3}^2-{\chi_1}{\chi_2}{\chi_3}-4} \\
&\cdot \sqrt{{\chi_1}^2+{\chi_4}^2+{\chi_5}^2-{\chi_1}{\chi_4}{\chi_5}-4} \cdot ({t_1^{FN}}+(t_1^{FN})^{-1}) \\
&+{\chi_1}( \chi_3 \chi_5 + \chi_2 \chi_4 ) -2( \chi_2\chi_5+\chi_3\chi_4 ) )  \biggr).
\end{split}
\]
(We remark that the insides of the square roots coincide with $\kappa$ of Theorem \ref{thm:goldman_SL2R_or_SU2} and are shown to be positive.)
Let $l_i$ be the length and $\tau_i$ the Fenchel-Nielsen twist parameter of the pants curve $c_i$.  
Since $l_i = 2 \log(-e_i)$ and $\tau_i = \log (t^{FN}_i)$, we have 
\[
\begin{split}
e_i+{e_i}^{-1} = - 2 &\cosh(l_i/2), \quad e_1-{e_1}^{-1} = -2 \sinh(l_1/2), \\
&t^{FN}_1+(t^{FN}_1)^{-1}  = 2\cosh(\tau_1).
\end{split}
\]
Using these relations, we obtain Okai's formula \cite{okai_effects}:
\[
\begin{split}
\cosh&(l'_1/2) = \frac{1}{\sinh^2(l_1/2)} \times  \\
\biggl(&\sqrt{(\cosh^2(\frac{l_1}{2})+\cosh^2(\frac{l_2}{2})+\cosh^2(\frac{l_3}{2})+2\cosh(\frac{l_1}{2})\cosh(\frac{l_2}{2})\cosh(\frac{l_3}{2})-1} \times \\
&\sqrt{(\cosh^2(\frac{l_1}{2})+\cosh^2(\frac{l_4}{2})+\cosh^2(\frac{l_5}{2})+2\cosh(\frac{l_1}{2})\cosh(\frac{l_4}{2})\cosh(\frac{l_5}{2})-1} \times \cosh(\tau_1)  \\
&+\cosh(l_1/2)( \cosh(l_3/2) \cosh(l_5/2) + \cosh(l_2/2) \cosh(l_4/2) )  \\
&+( \cosh(l_2/2)\cosh(l_5/2)+\cosh(l_3/2)\cosh(l_4/2) ) )  \biggr).
\end{split}
\]

Next we consider the one-holed sphere case studied in \S \ref{subsec:once_punctured}.
Rewrite (\ref{eq:one_holed_tr}) by $t^{FN}_1$, we have
\[
\tr(\rho(\beta_1)) 
= \frac{1}{e_1 -{e_1}^{-1}} \sqrt{{e_1}^2+{{e_1}}^{-2} -(e_2+{e_2}^{-1}))}((t^{FN}_1)^{1/2}+(t^{FN}_1)^{-1/2}).
\]
Let $e'_1 $ be one of the eigenvalues of $\rho(\beta_1)$ and $l_i$ and $\tau_i$ as before, we have
\begin{equation}
2 \cosh(l'_1/2) = \frac{\cosh(\tau_1/2)}{\sinh(l_1/2)} \sqrt{2\cosh(l_1) + 2\cosh(l_2/2)}.
\end{equation}

%% file: sec_3_mfd.tex
%
%


In this section, we construct $\PSLC$-representations of the fundamental group of a 3-manifold using an ideal triangulation, developed in \cite{thurston} and \cite{neumann-zagier}.
In \S \ref{sec:change_to_shear_coodinates}, we use this construction to transform our coordinates into exponential shear-bend coordinates.

\subsection{Ideal tetrahedra}
\label{sec:ideal_tetrahedra}
An \emph{ideal tetrahedron} is the convex hull of distinct 4 points of $\mathbb{C}P^1$ in $\mathbb{H}^3$.
We assume that every ideal tetrahedron has an ordering on the vertices.
Let $z_0,z_1,z_2,z_3$ be the vertices of an ideal tetrahedron.
This ideal tetrahedron is parametrized by the cross ratio $[z_0:z_1:z_2:z_3]$.
As remarked in \S \ref{subsec:basic_facts_on_PSL}, the cross ratio is invariant under the action of $\PSLC$.
We denote the edge of the ideal tetrahedron spanned by $z_i$ and $z_j$ by $[z_iz_j]$.
Take $(i,j,k,l)$ to be an even permutation of $(0,1,2,3)$.
We define the complex parameter of the edge by the cross ratio $[z_i:z_j:z_k:z_l]$.
This parameter only depend on the choice of the edge $[z_iz_j]$.
We can easily observe that the opposite edge has the same complex parameter  
and the other edges are parametrized by $\frac{1}{1-z}$ and $1-\frac{1}{z}$ where $z = [z_i:z_j:z_k:z_l]$ (see Figure \ref{fig:ideal_tetrahedron}).
\begin{figure}
\input{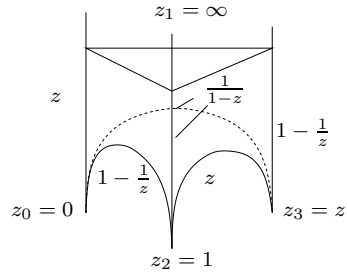}
\caption{The complex parameters of the edges of an ideal tetrahedron.}
\label{fig:ideal_tetrahedron}
\end{figure}

Let $A$ be an element of $\PSLC$ having two fixed points $(x,y)$.
We denote the eigenvalue corresponding to $x$ by $e$. (Thus $e^{-1}$ is the eigenvalue corresponding to $y$.) 
Then $A$ is given by (\ref{eq:two_fixed_points_case}).
Let $z$ be a point of $\mathbb{C}P^1$ distinct from $x$ and $y$.
Consider the ideal tetrahedron spanned by $(x,y,z,Az)$, the complex parameter of the edge $[xy]$ is equal to $[x:y:z:Az] = e^2$ (see Figure \ref{fig:interpretation} 
and compare with Lemma \ref{lem:two_fixed_pts_with_one_point_behavior}).
In other words, for an ideal tetrahedron spanned by $z_0, z_1, z_2, z_3$, the element of $\PSLC$ which sends 
$(z_0,z_1,z_2)$ to $(z_0,z_1,z_3)$ has eigenvalues $(\sqrt{[z_0: z_1: z_2 : z_3]})^{\pm 1}$.
So the cross ratio of an ideal tetrahedron can be interpreted as the square of an eigenvalue of some matrix related to the ideal tetrahedron.

\begin{figure}
\input{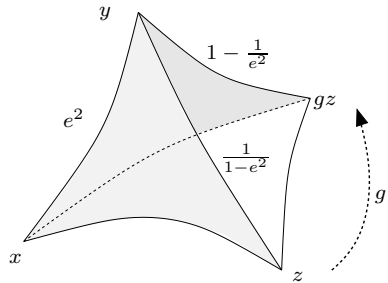}
\caption{If $A$ is an element of $\PSLC$ whose fixed points are $x$ and $y$, and $e$ is the eigenvalue corresponding to $x$.
Then the edge $[xy]$ of the ideal tetrahedron $(x,y,z,Az)$ has the complex parameter $e^2$.}
\label{fig:interpretation}
\end{figure}

\subsection{Ideal triangulation and representation of 3-manifold groups}
In the following sections, a \emph{triangulation} $T$ is a cell complex obtained by gluing 3-dimensional tetrahedra along their 2-dimensional faces in pairs.
We remark that this is not a simplicial complex since some vertices of a tetrahedron may be identified in $T$.
We often distinguish between a $0$-\emph{simplex} of $T$ and a \emph{vertex} of a tetrahedron since various vertices of tetrahedra may be identified with one $0$-simplex in $T$.
We also distinguish between a $1$-simplex of $T$ and an edge of a tetrahedron.
We denote the $k$-skeleton of $T$ by $T^{(k)}$.
\begin{definition}
A (topological) \emph{ideal triangulation} of a compact 3-manifold $M$ is a triangulation $T$ such that $T - N(T^{(0)})$ is homeomorphic to $M$.
\end{definition}
Here $N(T^{(0)})$ is a small open neighborhood of $T^{(0)}$.
Because all $0$-simplices are missing in $M$, we call them ideal vertices.
In the usual definition of ideal triangulations, it is assumed that $\partial N(T^{(0)})$ consists of tori, but here we do not assume this property.

We denote the universal cover of $M$ by $\widetilde{M}$.
From now on, we construct a developing map i.e. an equivariant map $\widetilde{M} \to \mathbb{H}^3$, which gives rise to a $\PSLC$-representation of $\pi_1(M)$.
We assign an ordering on the vertices of each tetrahedron of $T$, then assign a complex parameter $z_i$ for each tetrahedron.
Pick a tetrahedron $\Delta$ of $T$, then put an ideal tetrahedron in $\mathbb{H}^3$ according to the complex parameter of $\Delta$.
Then the tetrahedra adjacent to $\Delta$ can be realized in $\mathbb{H}^3$ according to their complex parameters. 
Continuing in this way, we obtain a map from the universal cover of $T-T^{(1)}$ to $\mathbb{H}^3$.
To obtain a map from the universal cover $\widetilde{M}$, which is homeomorphic to the universal cover of $T - N(T^{(0)})$, 
we have to impose the \emph{gluing equation} around each $1$-simplex of $T$.
Consider the edges of the tetrahedra which belong to a $1$-simplex of $T$ (Figure \ref{fig:gluing_equation}).
When a path goes around the $1$-simplex, we have to make sure that the developed image of the ideal tetrahedra in $\mathbb{H}^3$ returns back to the beginning position.
Since each edge has complex parameter $z_i$, $\frac{1}{1-z_i}$ or $1-\frac{1}{z_i}$, we have to impose the following equation
\[
\prod_{i=1} z_i^{p_{ji}}\left(\frac{1}{1-z_i}\right)^{p'_{ji}}\left(1- \frac{1}{z_i}\right)^{p''_{ji}}  = 1
\]
for each $1$-simplex (indexed by $j$) of $T$.
These equations are simplified to the following form:
\[
\pm \prod_{i=1} z_i^{r'_{ji}}(1-z_i)^{r''_{ji}}  = 1.
\]
We call these equations \emph{gluing equations}.
Let
\[
\mathcal{D}(M,T) = \{ (z_1, \dots, z_n) \in (\mathbb{C}-\{0,1\})^n | \pm \prod_{i=1} z_i^{r'_{ji}}(1-z_i)^{r''_{ji}}  = 1 \quad (\forall j) \},
\]
where $n$ is the number of the tetrahedra of the ideal triangulation $T$.
When the triangulation $T$ is clear from the context we simply denote $\mathcal{D}(M,T)$ by $\mathcal{D}(M)$.
For any point of $\mathcal{D}(M)$, we obtain a map $D: \widetilde{M} \to \mathbb{H}^3$.
For any $\gamma \in \pi_1(M)$, there exists a unique element $\rho(\gamma) \in \PSLC$ 
such that $D(\gamma p) = \rho (\gamma ) D(p)$ for any $p \in \widetilde{M}$, where $\gamma$ acts on $\widetilde{M}$ as a deck transformation.
Then $\rho$ is a homomorphism from $\pi_1(M)$ to $\PSLC$, which is called the \emph{holonomy representation} of $D$.
If we change the position of the first ideal tetrahedron $\Delta$, we obtain a conjugate representation.
So we have a map $\mathcal{D}(M) \to X_{PSL}(M)$ by sending an element of $\mathcal{D}(M)$ to its holonomy representation. 

\begin{figure}
\input{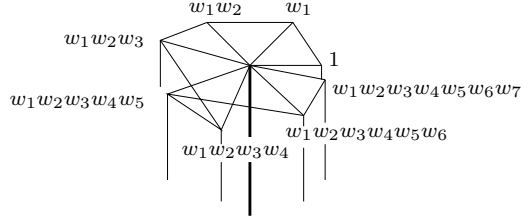}
\caption{Developing map around a $1$-simplex. Each $w_k$ is one of $z_i$, $\frac{1}{1-z_i}$ or $1-\frac{1}{z_i}$.}
\label{fig:gluing_equation}
\end{figure}

Next we discuss the restriction of the holonomy representation to a boundary subgroup.
Each boundary component of $M$ is expressed as $\partial N(p)$ by some $0$-simplex $p$ of $T$.
Therefore the restriction of the holonomy representation to the boundary subgroup $\partial N(p)$ fixes a point of $\mathbb{C}P^1$, i.e. reducible.
The results we have obtained so far are summarized as follows: 
\begin{proposition}
\label{prop:3-dim}
Let $M$ be a compact 3-manifold and $T$ be an ideal triangulation of $M$.
For $(z_1, \dots, z_n) \in \mathcal{D}(M,T)$, there exists a $\PSLC$ representation of $\pi_1(M)$ up to conjugation.
This gives an algebraic map $\mathcal{D}(M,T) \to X_{PSL}(M)$.
The restriction of the representation to any boundary subgroup is reducible.
\end{proposition}

We further study the restriction of the representation of Proposition \ref{prop:3-dim} to a boundary subgroup. 
Let $p$ be a $0$-simplex of $T$.
Here $\partial N(p)$ is triangulated by the truncated vertices (Figure \ref{fig:at_infinity}).
By conjugation, we assume that $p$ is developed to $\infty$.
Then the truncated vertices are developed into the Euclidean plane as triangles.
By the observation in the previous subsection, each triangle can be interpreted as a matrix fixing $\infty$ and whose eigenvalues are given by the complex parameter of 
the ideal tetrahedron.
Therefore the monodromy along a curve on the boundary component can be described by these complex parameters.
For example the monodromy of the path indicated in Figure \ref{fig:at_infinity} is given by
\[
\begin{split}
&\begin{pmatrix} \sqrt{w_1} & * \\ 0 & 1/\sqrt{w_1} \end{pmatrix}
\begin{pmatrix} 1/\sqrt{w_2} & * \\ 0 & \sqrt{w_2} \end{pmatrix}
\begin{pmatrix} \sqrt{w_3} & * \\ 0 & 1/\sqrt{w_3} \end{pmatrix}
\begin{pmatrix} 1/\sqrt{w_4} & * \\ 0 & \sqrt{w_4} \end{pmatrix} \\
&= \begin{pmatrix} \sqrt{\frac{w_1w_3}{w_2w_4}} & * \\ 0 & \sqrt{\frac{w_2w_4}{w_1w_3}} \end{pmatrix}
\end{split}
\]
where each $w_k$ is one of $z_i$, $\frac{1}{1-z_i}$ or $1-\frac{1}{z_i}$.
Therefore the eigenvalue of this matrix corresponding to the fixed point $\infty$ is $\sqrt{\frac{w_1w_3}{w_2w_4}}$.
In this way, the square of one of the eigenvalues of the monodromy along a boundary curve has the form 
\[
\pm \prod_i z_i^{m'_i} (1-z_i)^{m''_i}
\]
by some integers $m'_i$ and $m''_i$. 

\begin{figure}
\input{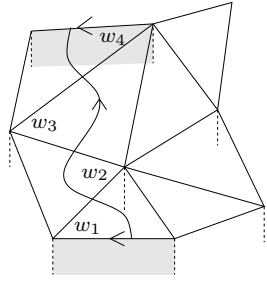}
\caption{Developing map around a $0$-simplex.}
\label{fig:at_infinity}
\end{figure}

The construction of this subsection can work even when $T$ has a free face i.e. some tetrahedra have a face with no pair.
This type of ideal triangulation was used in \cite{kabaya} and will be used in \S \ref{sec:change_to_shear_coodinates}.

%% file: sec_shear_bend.tex
%
%


In this section, we describe a transformation of our coordinates into exponential shear-bend coordinates.
We will give such a transformation for one-holed torus, instead of describing general cases.

\subsection{Representations using exponential shear-bend coordinates}
Consider an ideal triangulation $T_1$ of the one-holed torus as in Figure \ref{fig:one_holed_bending}.
We assign complex parameters $a, b, c \in \mathbb{C} \setminus \{0,1\}$ to the edges of the ideal triangulation as in Figure \ref{fig:one_holed_bending}.
We construct a developing map according to these complex parameters as described in \S \ref{subsec:developing_map}.
Fix an ideal triangle whose ideal vertices consisting of $(0, 1, \infty)$.
Then we can develop this ideal triangle according to the complex parameters $a$, $b$ and $c$. 
A part of the developed image is shown in the right of Figure \ref{fig:one_holed_bending}.
We denote the holonomy representation of the developing map by $\rho$.
Take generators $\alpha_1$ and $\beta_1$ of the fundamental group as in Figure \ref{fig:once_punctured}, 
then $\rho({\alpha_1}^{-1})$ is the matrix which sends $(0,\infty, 1)$ to $(1/a, \frac{c-1}{ac}, \infty)$ (see Figure \ref{fig:one_holed_bending}).
So we have 
\[
\rho({\alpha_1}^{-1}) = \frac{1}{\sqrt{ac}} \begin{pmatrix} c-1 & -c \\ ac & -ac \end{pmatrix}.
\] 
Since $\rho(\beta_1)$ is the matrix which sends $(0,\infty, 1)$ to $(\frac{1}{a(1-c)}, 1/a, 0)$, we have 
\[
\rho(\beta_1) = \frac{1}{\sqrt{ab}} \begin{pmatrix} 1  & -1 \\ a & a(b-1) \end{pmatrix}.
\]
For example, we have 
\begin{equation}
\label{eq:one_holed_by_bending}
\tr^2(\rho(\alpha_1)) = \frac{(ca-c+1)^2}{c a}, \quad 
\tr^2(\rho(\beta_1)) =\frac{(ab-a+1)^2}{a b}.
\end{equation}

\begin{figure}
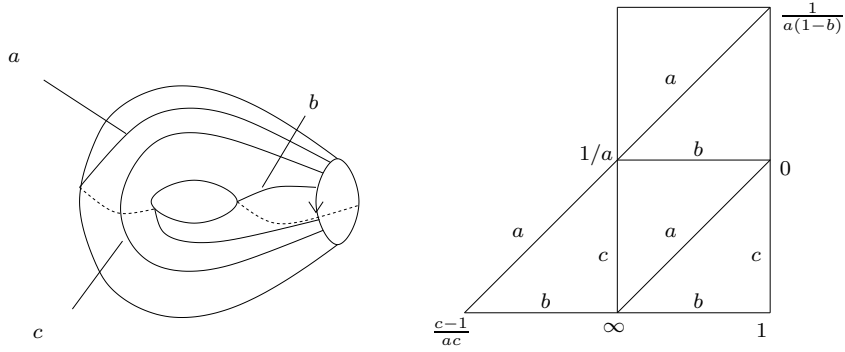

\input{id_tri_of_one_holed.pstex_t}
\hspace{20pt}
\input{id_tri_of_one_holed_cv.pstex_t}
\caption{Ideal triangulation $T_1$.}
\label{fig:one_holed_bending}
\end{figure}

\subsection{Transformation into exponential shearing-bending coordinates}
We give a transformation of our coordinates into exponential shear-bend coordinates by giving a 3-dimensional ideal triangulation 
between two 2-dimensional ideal triangulations.
That is we construct a 3-dimensional ideal triangulation between the ideal triangulation $T_{-1}$ of the left of Figure \ref{fig:one_holed_begin} 
and $T_1$ of Figure \ref{fig:one_holed_bending}.
We also let $T_0$ be the ideal triangulation as indicated in the right of Figure \ref{fig:one_holed_begin}.

Let $e_1,e_2$ be the eigenvalue parameters and $t_1$ the twist parameter as defined in \S \ref{subsec:once_punctured}.
The complex parameters of the edges of $T_{-1}$ can be computed from Proposition \ref{prop:shear_bend_parameter_of_a_pair_of_pants}. 
They are ${e_1}^2/e_2$, $1/({e_1}^2 e_2)$ and $e_2$ as shown in the left of Figure \ref{fig:one_holed_begin}.
  
\begin{figure}
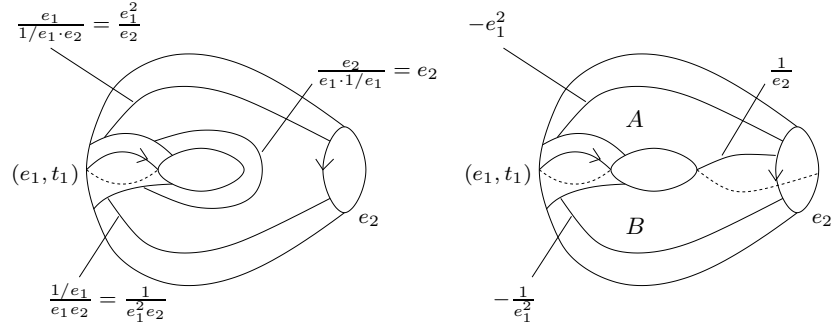

\input{one_holed_begin.pstex_t}
\hspace{30pt}
\input{one_holed_mid.pstex_t}
\caption{The left is the ideal triangulation $T_{-1}$ and the right is $T_0$. 
$T_0$ is obtained from $T_{-1}$ by inserting an ideal tetrahedron with the complex parameter $1/e_2$.}
\label{fig:one_holed_begin}
\end{figure}

\begin{figure}
\input{between_one_holed.pstex_t}
\hspace{30pt}
\input{one_holed_top.pstex_t}
\caption{}
\label{fig:between_once_punt}
\end{figure}

First we exchange the edge of $T_{-1}$ whose complex parameter is $e_2$ by inserting an ideal tetrahedron with the complex parameter $1/e_2$. 
Since we have
\[
\frac{{e_1}^2}{e_2} \frac{1}{1-1/{e_2}} \left( 1- \frac{1}{1/e_2} \right) = -{e_1}^2, \quad 
\frac{1}{{e_1}^2e_2} \frac{1}{1-1/{e_2}} \left( 1- \frac{1}{1/e_2} \right) = -\frac{1}{{e_1}^2},
\]
the edges of $T_0$ have the complex parameters $-{e_1}^2$, $-1/({e_1}^2)$ and $1/{e_2}$ (see the right of Figure \ref{fig:one_holed_begin}).
Next we attach a pair of ideal tetrahedra parametrized by $z_1$ and $z_2$ to $T_0$ as indicated in Figure \ref{fig:between_once_punt}.
In the figure, glue two faces labeled by $C$ in pairs and glue the face labeled by $A$ (resp. $B$) to $A$ (resp. $B$) of $T_0$ 
indicated in the right of Figure \ref{fig:one_holed_begin}. 
Now the remaining two free faces form the ideal triangulation $T_1$ (see Figure \ref{fig:between_once_punt}).
Observe the complex parameters at the edges of $T_1$, we have 
\begin{equation}
\label{eq:a_b_c_by_z1_z2}
a =  \left( 1 - \frac{1}{z_1} \right) \left( 1 - \frac{1}{z_2} \right), \quad
b =  \frac{1}{e_2}  \frac{1}{1-z_1}  \frac{1}{1-z_2}, \quad
c = z_1 z_2.
\end{equation}
Considering the complex parameters around the $1$-simplex at which two edges of $A$ meet, we have
\begin{equation}
\label{eq:relation_of_zi_ei_1}
\frac{1}{1-z_1} \cdot z_2 \cdot \bigl( 1 -\frac{1}{z_2} \bigr) \cdot (-e_1^2) = 1.
\end{equation}
Similarly, considering the complex parameters around the $1$-simplex at which two edges of $B$ meet, we have
\begin{equation}
\label{eq:relation_of_zi_ei_2}
\bigl( 1 -\frac{1}{z_1} \bigr) \cdot z_1 \cdot \frac{1}{1-z_2} \cdot \frac{-1}{e_1^2} = 1.
\end{equation}
(The equation (\ref{eq:relation_of_zi_ei_2}) is equivalent to (\ref{eq:relation_of_zi_ei_1}).)
By the definition of the twist parameter, the matrix which sends $B$ to $A$ has the eigenvalues $\sqrt{-t_1}^{\pm 1}$.
Thus, from the left of Figure \ref{fig:between_once_punt}, we have
\begin{equation}
\label{eq:relation_of_zi_ti}
\frac{z_2}{z_1} = -t_1.
\end{equation}
Solve the equations (\ref{eq:relation_of_zi_ei_1})--(\ref{eq:relation_of_zi_ti}), we have
\[
z_1 = \frac{1-e_1^{2}}{t_1 e_1^{2} + 1}, \quad z_2 = -\frac{t_1(1-e_1^{2})}{t_1 e_1^2 +1 }.
\]
Substitute $z_1$ and $z_2$ into (\ref{eq:a_b_c_by_z1_z2}),  we have
\[
a= - \frac{e_1^2(1+t_1)^2}{t_1(1-e_1^2)^2}, \quad
b= \frac{(t_1e_1^2+1)^2}{e_1^2 e_2 (t_1+1)^2}, \quad
c= - \frac{t_1 (1-e_1^2)^2}{(t_1 e_1^2+1)^2}. 
\]
Substitute these into (\ref{eq:one_holed_by_bending}), we obtain
\[
\begin{split}
\tr^2(\beta_1) &= \frac{(ca-c+1)^2}{c a} = \left(e_1+\frac{1}{e_1}\right)^2, \\
\tr^2(\alpha_1) &=\frac{(ab-a+1)^2}{a b} 
= - \frac{(({e_1}^2-e_2)t_1 + 1 - e_1^2 e_2)^2}{(e_1^2-1)^2 e_2 t_1}.
\end{split}
\]
The results coincide with the calculations of \S \ref{subsec:once_punctured}.